\documentclass[a4paper,12pt]{amsart}
\usepackage{amsmath,amssymb,amsthm}
\usepackage{mathtools}
\usepackage{mathrsfs}
\usepackage{enumitem}
\usepackage{indentfirst}
\usepackage{microtype}
\usepackage[all]{xy}

\numberwithin{equation}{section}
\usepackage{stmaryrd}

\usepackage{xcolor}
\colorlet{darkgreen}{green!45!black} 
\usepackage[
  pagebackref=true,
  colorlinks=true,
  linkcolor=blue,
  citecolor=darkgreen,
  urlcolor=blue
]{hyperref}
\usepackage[nameinlink,capitalize]{cleveref}

\renewcommand*{\backref}[1]{}
\renewcommand*{\backrefalt}[4]{%
  \ifcase #1
    (not cited)
  \or
    (cited on p.~#2)
  \else
    (cited on pp.~#2)
  \fi
}

\makeatletter

\def\@settitle{%
  \begin{center}
    \baselineskip20\p@\relax
    \Large\bfseries
    \@title
  \end{center}
}

\def\@setauthors{%
  \begingroup
  \def\thanks{\protect\thanks@warning}
  \trivlist
  \centering
  \normalsize
  \@topsep30\p@\relax
  \advance\@topsep by -\baselineskip
  \item\relax
  \author@andify\authors
  \def\\{\protect\linebreak}
  \authors
  \ifx\@empty\contribs
  \else
    ,\penalty-3 \space \@setcontribs
    \@closetoccontribs
  \fi
  \endtrivlist
  \endgroup
}

\renewenvironment{abstract}{%
  \ifx\maketitle\relax
    \ClassWarning{\@classname}{Abstract should precede
      \protect\maketitle\space in AMS document classes; reported}%
  \fi
  \global\setbox\abstractbox=\vtop \bgroup
    \normalfont\Small
    \list{}{%
      \labelwidth\z@
      \leftmargin3pc
      \rightmargin\leftmargin
      \listparindent\normalparindent
      \itemindent\z@
      \parsep\z@ \@plus\p@
      
    }%
    \item[\hskip\labelsep\bfseries\abstractname.]%
}{%
  \endlist\egroup
  \ifx\@setabstract\relax
    \@setabstracta
  \fi
}

\def\section{\@startsection{section}{1}%
  \z@
  {.7\linespacing\@plus\linespacing}%
  {.5\linespacing}%
  {\large\normalfont\bfseries\centering}}

\makeatother

\newcommand{\Acon}{A_{\mathrm{con}}}

\newcommand{\Gam}{\Gamma}
\newcommand{\rad}{\operatorname{rad}}

\newcommand{\End}{\operatorname{End}}
\newcommand{\Hom}{\operatorname{Hom}}

\newcommand{\op}{\mathrm{op}}

\newcommand{\underlineEnd}{\underline{\operatorname{End}}}
\newcommand{\underlineHom}{\underline{\operatorname{Hom}}}

\newcommand{\DynA}{\mathsf{A}}
\newcommand{\DynD}{\mathsf{D}}
\newcommand{\DynE}{\mathsf{E}}
\theoremstyle{plain}
\newtheorem{theorem}{Theorem}[section]
\newtheorem{proposition}[theorem]{Proposition}
\newtheorem{lemma}[theorem]{Lemma}
\newtheorem{corollary}[theorem]{Corollary}

\theoremstyle{definition}

\theoremstyle{remark}
\newtheorem{remark}[theorem]{Remark}

\title{Representation type of contraction algebras}

\author{Qi Wang}
\email{wang2025@dlut.edu.cn (cc: infinite-wang@outlook.com)}
\address{School of Mathematical Sciences,
Dalian University of Technology, Dalian City, Liaoning Province 116024, China}



\date{October 24, 2026.}
\begin{document}
\setlength{\baselineskip}{16pt}

\maketitle

\begin{abstract}
We classify the representation-finite, tame, and wild contraction algebras arising from 3-fold flopping contractions of isolated cDV singularities with at worst Gorenstein terminal sources. To achieve this, we first determine the representation type of all contracted preprojective algebras of Dynkin types $\mathsf{A}$, $\mathsf{D}$, and $\mathsf{E}$, and then use the representation type invariance under iterated flops and the radical structure of contraction algebras.
\end{abstract}

\tableofcontents

\section{Introduction}
Contraction algebras arise from the birational geometry of 3-fold flops. Let $f\colon X\to\operatorname{Spec}R$ be a 3-dimensional flopping contraction in the complete local setting. As shown in \cite{VdB04}, the tilting-theoretic description of 3-dimensional flops associates to $f$ a noncommutative endomorphism algebra $\Lambda=\operatorname{End}_R(R\oplus N)$, where the summand $R$ records the non-exceptional part of the geometry. If $e\in\Lambda$ denotes the idempotent corresponding to $R$, the \emph{contraction algebra} is defined in \cite{DW16} by $\Acon:=\Lambda/\Lambda e\Lambda\cong\underline{\operatorname{End}}_R(N)$. Thus, $\Acon$ is obtained by removing the morphisms which factor through the distinguished projective summand, and in this sense isolates the algebraic information carried by the exceptional curves. It is shown in \cite{AugustMZ20,August20} that $\Acon$ is finite-dimensional and symmetric. Moreover, it represents the noncommutative deformation theory of the exceptional curves and controls the flop--flop autoequivalence, giving the contraction algebra a direct geometric meaning beyond its construction as a stable endomorphism algebra; see \cite{DW16,DW19,Wemyss18}.

Since its introduction, the structure of $\Acon$ has been studied from several complementary directions. On the geometric side, the dimensions of contraction algebras encode Gopakumar--Vafa invariants of simple flops as studied in  \cite{DW16,NW24}, while it is established in \cite{BWonGV} that $\Acon$ is a strictly finer invariant than the collection of these numerical invariants. This motivated the expectation, formulated in \cite{DW16} and addressed in \cite{JKM24}, that $\Acon$ should determine the local geometry of the flop. On the homological side, it is shown in \cite{AugustMZ20,August20} that mutation of minimal models is reflected by derived equivalences between their contraction algebras. The corresponding standard derived equivalences were constructed in \cite{August20}, where it was also shown that the mutation graph, together with its associated simplicial hyperplane arrangement, describes the derived equivalence class and controls the standard derived autoequivalences. In particular, it is known in \cite{August20} that the hyperplane arrangement can be recovered intrinsically from the two-term tilting theory of $\Acon$. It was subsequently shown in \cite{AW22} that this structure describes the full Bridgeland stability manifold: the same arrangement appears as the $g$-vector fan, and the stability manifold is its contractible universal cover. Further refinements, including derived and $A_\infty$ enhancements of the contraction algebra and descriptions by noncommutative potentials, are developed in \cite{JKM24,BW25,Zhang26}.

These developments reveal a remarkably homological and geometric structure. There is, however, a basic question from the classical representation theory of finite-dimensional algebras which has not been the organising focus of this theory. Namely, what is the representation type of $\Acon$? More precisely, when is $\Acon$ representation-finite, tame or wild? The purpose of this paper is to study this finite--tame--wild trichotomy systematically and to relate it to the birational and surface geometry from which the contraction algebra arises. We work over $\Bbbk=\mathbb C$ throughout the geometric part.

A first difficulty is that contraction algebras arise from the birational geometry of 3-folds, whereas the most accessible combinatorial information comes from a surface section. As recalled in \cite{Reid83,Reid97}, a sufficiently general hyperplane section of the cDV singularity is a Du Val surface singularity. The induced partial resolution determines a simply-laced Dynkin diagram $\Delta$ together with a subset $\Omega\subseteq\Delta_0$ corresponding to the exceptional curves. One obtains the associated contracted preprojective algebra $\Gam(\Delta,\Omega):=e_\Omega\Pi_\Delta e_\Omega$, as considered in \cite{IyamaWemyssTits,ChanIyamaMarczinzik}. It follows that the essential geometric input is the central quotient established in \cite{DW16,DW19,Wemyss18}, namely $\Acon/g\Acon\cong\Gam(\Delta,\Omega)$ for some $g\in Z(\Acon)\cap\rad\Acon$. Consequently, a wild surface quotient forces the contraction algebra to be wild. This is the first reduction in the paper.

The second difficulty is that the family of contracted preprojective algebras is still infinite. We overcome this by developing a uniform reduction strategy that combines quotient algebras, idempotent truncations, radical-square-zero quotients, and covering arguments to reduce the classification to smaller subdiagrams and low-rank cases. Together with the analysis of explicit boundary algebras, these methods yield the complete classification in types~$\DynA$, $\DynD$, and $\DynE$.

There is a third reduction which is specific to the 3-fold problem. It is shown in \cite{AugustMZ20,August20} that contraction algebras associated with different minimal models of the same isolated cDV singularity are derived equivalent. Since contraction algebras are symmetric, it follows from \cite{Rickard89,Krause97} that their classical representation type is constant under iterated flops. We may therefore test all surface quotients occurring in a mutation orbit: if one of them is wild, then every contraction algebra in that iterated-flop class is wild. 

The remaining lifting problem is governed by a common mechanism. The surface quotient determines the arrows between distinct vertices, while the subset-contraction correspondence in \cite{DW19} allows one-curve subset contractions to detect loops after suitable iterated flops. We then pass to a radical-square-zero quotient; whenever its separated quiver is already wild, no higher relations are needed. Only in the remaining boundary cases do we inspect the first radical multiplication not forced by the surface quotient, and in a few loopless cases, one further radical layer. This reduces the lifting problem to low-radical data and a finite number of boundary cases.

We have the following main result in this paper.
\begin{theorem}\label{thm:intro-main}
Let $f\colon X\to\operatorname{Spec}R$ be a complete local isolated cDV flopping contraction whose source has at worst Gorenstein terminal singularities, and let $\Acon$ be its contraction algebra. Let $(\Delta,\Omega)$ be the Dynkin data of a sufficiently general surface section, and write $n:=|\Omega|$ for the number of irreducible exceptional curves. Then, the representation type of $\Acon$ is determined as follows.
\begin{enumerate}
\item[\textup{(1)}] \textbf{One curve.} If $n=1$, then $\Acon$ is representation-finite if and only if its Gabriel quiver $Q$ has at most one loop, and it is wild if $Q$ has at least three loops. If there are exactly two loops in $Q$, the representation type of $\Acon$ is determined by the surface quotient and the quadratic radical multiplication, as specified in Proposition~\ref{prop:onecurve-classification}.

\item[\textup{(2)}] \textbf{Two curves.}
If $n=2$, then $\Acon$ is representation-finite if and only if
its surface data lie in one of the surviving type~$\DynA$
or type~$\DynD$ orbits and some iterated-flop representative
has a loopless Gabriel quiver.
In the remaining type~$\DynA$, type~$\DynD$, and exceptional
cases surviving the surface-quotient tests, the tame--wild
distinction is determined by the loop spaces and quadratic
radical multiplication in suitable iterated-flop representatives.
The precise criteria are given in
Propositions~\ref{thm:step3-A2-quadratic},
and~\ref{thm:two-curve-remaining-wild},
together with the mutation reductions in
Section~\ref{sec:flop-detection}.

\item[\textup{(3)}] \textbf{Three curves.} If $n=3$, then $\Acon$ is always representation-infinite. The tame cases occur precisely among the surviving type~$\DynA$ and type~$\DynD$ mutation orbits, subject to the explicit conditions of Propositions~\ref{thm:terminal-cA3-trichotomy} and~\ref{thm:terminal-D-rank3}; all other three-curve cases are wild.

\item[\textup{(4)}] \textbf{Four or more curves.} If $n\ge4$, then $\Acon$ is wild.
\end{enumerate}
\end{theorem}

We now describe the organisation of the paper. Section~\ref{sec:prelim} fixes the conventions and collects the representation-theoretic and geometric tools used throughout. Section~\ref{sec:A-stable} carries out Step~1 in type~$\DynA$, beginning with the uniform stable-endomorphism model and ending with the complete classification of all $\Gam(\DynA_m,\Omega)$. Section~\ref{sec:DE} completes Step~1 for types~$\DynD$ and~$\DynE$. Section~\ref{sec:flop-detection} studies the effect of iterated flops on the surface data and isolates the mutation classes which survive to Step~3. Finally, Section~\ref{sec:lifts} applies the unified lifting mechanism above to the remaining one-, two-, three-, and higher-rank cases, completing the classification.

\section{Preliminaries}\label{sec:prelim}
Let $\Bbbk$ be an algebraically closed field and let $A$ be a finite-dimensional basic $\Bbbk$-algebra. All modules are right modules unless otherwise stated, and we denote by $\operatorname{mod}A$ the category of finitely generated right $A$-modules. We write $Z(A)$ for the center of $A$ and $\rad A$ for the Jacobson radical of $A$. Fix a complete set $\{e_i\mid 1\leq i\leq n\}$ of primitive orthogonal idempotents and put $P_i:=e_iA$. With our convention, an irreducible homomorphism $P_i\to P_j$ corresponds to an arrow $j\to i$ in the Gabriel quiver of $A$. We use the traversal-order convention for paths: if $\alpha:j\to i$ and $\beta:k\to j$, then $\beta\alpha:k\to i$. Under the identification $\operatorname{Hom}_A(P_i,P_j)\cong e_jAe_i$, if homomorphisms $f:P_i\to P_j$ and $g:P_j\to P_k$ correspond to $\alpha$ and $\beta$, respectively, then the composition $g\circ f$ corresponds to the path $\beta\alpha$.

Recall that $A$ is \emph{representation-finite} if $\operatorname{mod}A$ contains only finitely many isomorphism classes of indecomposable modules; otherwise, $A$ is said to be \emph{representation-infinite}. A representation-infinite $\Bbbk$-algebra $A$ is said to be \emph{tame} if, for each dimension $d$, all but finitely many $d$-dimensional indecomposable $A$-modules can be organized into finitely many one-parameter families, and it is called \emph{wild} if there exists a representation embedding $\operatorname{mod}\Bbbk\langle x,y\rangle\to\operatorname{mod}A$, that is, an exact $\Bbbk$-linear functor which preserves indecomposability and reflects isomorphism classes. It is well-known from \cite{Drozd80} that every representation-infinite algebra over an algebraically closed field is either tame or wild. Since representation type is invariant under Morita equivalence, we shall freely replace an algebra by its basic algebra.

A basic algebra $A$ is \emph{self-injective} if the regular module $A_A$ is injective, and is \emph{symmetric} if $A\cong DA$ as $A$--$A$-bimodules, where $D(-):=\operatorname{Hom}_{\Bbbk}(-,\Bbbk)$ is the standard $\Bbbk$-duality. Every symmetric algebra is self-injective. A derived equivalence between self-injective algebras induces a stable equivalence (see \cite[Corollary~2.2]{Rickard89}), and stable equivalence preserves the finite--tame--wild representation type; see \cite[Corollaries~3.4--3.5]{Krause97}. It follows that representation type is a derived invariant for symmetric algebras. We note, however, that representation type is not preserved by derived equivalence for arbitrary algebras.

We shall repeatedly use several standard tools for determining representation type. We begin with the degeneration technique, which will be used to establish tameness. For a fixed $d$-dimensional vector space $V$, let $\operatorname{Alg}_d(\Bbbk)$ denote the affine variety of associative unital $\Bbbk$-algebra structures on $V$. The general linear group $\operatorname{GL}_{\Bbbk}(V)$ acts on $\operatorname{Alg}_d(\Bbbk)$ by transport of structure, and its orbits correspond to isomorphism classes of $d$-dimensional algebras. If $A$ and $B$ are two $d$-dimensional algebras in $\operatorname{Alg}_d(\Bbbk)$, we say that $A$ \emph{degenerates to} $B$, and write $A\rightsquigarrow B$, if the point corresponding to $B$ lies in the Zariski closure of the $\operatorname{GL}_{\Bbbk}(V)$-orbit of the point corresponding to $A$. In practice, it suffices to construct an algebraic family $\{A_\lambda\}_{\lambda\in\Bbbk}$ of algebra structures on a fixed vector space such that $A_\lambda\cong A_1$ for all $\lambda\neq0$ and $A_0\cong B$; one then obtains $A_\lambda\rightsquigarrow B$. It is proved in \cite{Geiss} that if $A\rightsquigarrow B$ and $B$ is tame, then $A$ is tame. Hence, a degeneration to a tame algebra provides a convenient method for excluding wildness. A classical class of tame algebras is provided by special biserial algebras; we refer to \cite{WaldWaschbusch} for details.

For detecting representation-infiniteness and wildness, we shall use quotients and idempotent truncations. If there is a surjection $A\twoheadrightarrow B$ and $B$ is representation-infinite (resp. wild), then $A$ is representation-infinite (resp. wild). We shall also use the idempotent truncation criterion: if $e\in A$ is an idempotent and a quotient algebra of $eAe$ is representation-infinite (resp. wild), then $A$ is representation-infinite (resp. wild). Notice that $eAe$ is not, in general, a quotient algebra of $A$. These reductions make algebras with very few simple modules particularly useful as test algebras.

It is well-known that $\Bbbk[x]/(x^n)$ for any $n\ge 2$ is a representation-finite local algebra. We shall also repeatedly use the following tame and wild local algebras.

\begin{proposition}\label{prop:local-benchmark-alg}
Let $A=\Bbbk Q/J$ be a local algebra with the quiver
\[
Q:\quad
\vcenter{\xymatrix@C=0.8cm{
\circ
\ar@(dl,ul)^{x}
\ar@(ur,dr)^{y}
}}\; .
\]
\begin{enumerate}
\item If $J=\langle x^2,y^2,xy-yx\rangle$, then $A$ is tame.

\item If $J=\langle x^2-y^2,xy,yx\rangle$, then $A$ is tame.

\item If $J=\langle x^2-y^2, xy\rangle$, then $A$ is wild.

\item If $J=\langle x^2,y^3,yx,xy^2\rangle$, then $A$ is wild.

\item If $J=\langle x^2,y^3,xy,y^2x\rangle$, then $A$ is wild.

\item If $J=\langle x^3,y^3,xy-yx,x^2y,xy^2\rangle$, then $A$ is wild.

\item If $J=\langle x^3,y^2,x^2y,xy-\lambda yx\rangle$, $\lambda \in\Bbbk^\times$, then $A$ is wild.
\end{enumerate}
\end{proposition}
\begin{proof}
\textup{(1)} and \textup{(2)} follow from the classification of
tame symmetric local algebras in \cite[Theorem~III.1]{Erdmann90}. \textup{(3)}, \textup{(6)} and \textup{(7)} are shown in \cite[I.10.10]{Erdmann90}. \textup{(4)} and \textup{(5)} are the minimal wild local algebras in \cite[(3.2)--(3.3)]{Ringel75}, respectively.
\end{proof}

A basic algebra with exactly two simple modules is called a \emph{two-point algebra}. The representation type of two-point algebras has been determined through a sequence of classification results. Up to isomorphism and duality, a two-point algebra is representation-finite or tame precisely when it degenerates to a quotient of an algebra in Table~T, while it is wild precisely when it has a minimal wild algebra from Table~W as a quotient; see \cite[Theorems~1 and~2]{Han02}. Since the conventions in the previous papers are not uniform, we use the right-module conventions of Tables~T and~W reproduced in \cite[Appendix~A]{Wang22}.

We shall use radical-square-zero quotients to detect representation type. Let $A$ be a basic algebra with $\rad^2 A=0$, and let $Q$ be its Gabriel quiver. The \emph{separated quiver} $Q^{\mathrm{s}}$ has two vertices $i^+$ and $i^-$ for each vertex $i$ of $Q$, and an arrow $i^+\to j^-$ for each arrow $i\to j$ in $Q$. By \cite[Chapter~X, Theorem~2.4]{ARS95}, $A$ is stably equivalent to the hereditary algebra $\Bbbk Q^{\mathrm{s}}$. In particular, $A$ is representation-finite if and only if every connected component of $Q^{\mathrm{s}}$ is of Dynkin type; see \cite[Chapter~X, Theorem~2.6]{ARS95}. Hence, if $Q^{\mathrm{s}}$ has a connected component of Euclidean type, then $A$ is representation-infinite. Moreover, if $Q^{\mathrm{s}}$ has a connected component whose underlying graph is neither Dynkin nor Euclidean, then $A$ is wild.

\subsection{Galois coverings}\label{subsec:galois-covering}
We recall a form of Galois covering theory that will be used below to detect wildness. Let $A=\Bbbk Q/I$ be a finite-dimensional bound quiver algebra. Suppose that each arrow $\alpha\in Q_1$ is assigned an integer degree $\deg\alpha\in\mathbb Z$, and that $I$ is generated by homogeneous relations with respect to the induced $\mathbb Z$-grading on $\Bbbk Q$. Let $\widetilde Q$ be the quiver with vertex set $\{i_t\mid i\in Q_0,\; t\in\mathbb Z\}$ and arrows $\alpha_t:i_t\longrightarrow j_{t+d}$, for every arrow $\alpha:i\to j$ of degree $d$ and every $t\in\mathbb Z$. Since $I$ is homogeneous, each defining relation $\rho\in I$ admits homogeneous lifts $\rho_t$ in $\Bbbk\widetilde Q$. Let $\widetilde I$ be the ideal generated by all such lifts. We call $\widetilde A:=\Bbbk\widetilde Q/\widetilde I$ the graded \emph{covering} of $A$. Moreover, the translations $i_t\longmapsto i_{t+1}$ and $\alpha_t\longmapsto\alpha_{t+1}$ induce a free action of $\mathbb Z$ on $\widetilde A$, whose orbit category is naturally isomorphic to $A$. Hence, the canonical projection yields a Galois covering $\mathcal F:\widetilde A\longrightarrow \widetilde A/\mathbb Z\cong A$. We refer to \cite{Green83} for the correspondence between group gradings and Galois coverings of bound quiver algebras.

The graded covering $\widetilde A$ is in general infinite-dimensional, but it is a locally bounded $\Bbbk$-category. We shall use the following standard consequence of covering theory.

\begin{proposition}[{\cite[Proposition~2]{DowborSkowronski85}, \cite[Theorem]{DowborSkowronski86}}]
\label{prop:covering-wild}
Let $\mathcal F:\widetilde A\longrightarrow A$ be a Galois covering with group $\mathbb Z$. If $\widetilde A$ contains a finite full subcategory which is wild, then $A$ is wild.
\end{proposition}

In practice, we shall apply Proposition~\ref{prop:covering-wild} by finding a finite full convex subcategory of $\widetilde A$ on which no defining relation occurs. Such a subcategory is hereditary, and it is wild if its underlying graph is neither Dynkin nor Euclidean.

\subsection{Contraction algebras}\label{subsec:contraction-algebras}
We work over the field $\Bbbk=\mathbb C$ of complex numbers for the geometric part of the paper. Let $f\colon X\to\operatorname{Spec}R$ be a \emph{$3$-fold flopping contraction}, where $X$ is an integral scheme of Krull dimension 3 with at worst Gorenstein terminal singularities, and $R$ is a 3-dimensional complete local normal Gorenstein $\mathbb C$-algebra. Such a morphism $f$ is a small projective birational morphism; in particular, it contracts no divisors, and its exceptional locus is a finite union of curves. In this setting, $R$ is necessarily an
isolated compound Du Val (\textup{cDV}) singularity; see
\cite[Theorem~1.1]{Reid83}. 

We denote by $\operatorname{CM}R$ the category of maximal
Cohen--Macaulay $R$-modules. Since $R$ is Gorenstein, a finitely
generated $R$-module $L$ belongs to $\operatorname{CM}R$ if and only if
$\operatorname{Ext}^i_R(L,R)=0$ for all $i>0$.
For an $R$-module $L$, we write $\operatorname{add}L$ for the full
subcategory consisting of direct summands of finite direct sums of copies
of $L$. An object $L\in\operatorname{CM}R$ is called \emph{basic} if its
indecomposable direct summands are pairwise non-isomorphic; it is called a
\emph{generator} if $R\in\operatorname{add}L$; and it is called \emph{rigid} if $\operatorname{Ext}^1_R(L,L)=0$.

Let $C_1,\ldots,C_n$ be the irreducible exceptional curves of $f$.
By the Auslander--McKay correspondence for isolated cDV singularities,
the $3$-fold flopping contraction $f\colon X\to\operatorname{Spec}R$ determines a basic rigid CM generator $M=R\oplus N_1\oplus\cdots\oplus N_n$. Here, the $N_i$ are pairwise non-isomorphic indecomposable non-free CM modules, with $N_i$ corresponding to the exceptional curve $C_i$; see \cite[Theorem~4.10 and Remark~4.13]{Wemyss18}. Set $N:=N_1\oplus\cdots\oplus N_n$ and $\Lambda:=\operatorname{End}_R(M)$, and let $e_0\in\Lambda$ be the idempotent corresponding to the free summand $R$. For any  $U,V\in\operatorname{CM}R$, let $[R](U,V)$ be the subspace of
$\operatorname{Hom}_R(U,V)$ consisting of morphisms which factor through
an object of $\operatorname{add}R$. Since $R$ is local,
$\operatorname{add}R$ is precisely the category of finitely generated
projective $R$-modules. We define the \emph{stable Hom-space} 
$\underline{\operatorname{Hom}}_R(U,V)
:=\operatorname{Hom}_R(U,V)/[R](U,V)$
and write
$\underline{\operatorname{End}}_R(U)
:=\underline{\operatorname{Hom}}_R(U,U)$.
We denote by $\underline{\operatorname{CM}}R$ the category with the same
objects as $\operatorname{CM}R$ and these stable Hom-spaces. The two-sided ideal $\Lambda e_0\Lambda$ consists precisely of the
endomorphisms of $M$ which factor through an object of
$\operatorname{add}R$. Then, the \emph{contraction algebra} associated to $f$ is defined by
\begin{equation}\label{eq:contraction-algebra}
\Acon
:=
\Lambda/\Lambda e_0\Lambda
\cong
\underline{\operatorname{End}}_R(N).
\end{equation}
This is the standard stable-endomorphism description of the contraction
algebra; see \cite[Subsection~3.3]{DW16}. By
\cite[Proposition~2.3 and Theorem~2.6]{August20}, building on results of
\cite{BIKR08}, every contraction algebra $\Acon$ in the above setting is
a finite-dimensional symmetric $\mathbb C$-algebra.

It is worth mentioning that the irreducible exceptional curves of $f$ index the simple modules of $\Acon$. For $1\le i\le n$, let $e_i\in\Lambda$ be the primitive idempotent corresponding to the summand $N_i$, and let $\overline e_i$ denote its image in $\Acon$. Then,
$\overline e_i\Acon$ is the corresponding indecomposable projective
$\Acon$-module, whose simple top is denoted by $S_i$. Consequently,
$\Acon$ has exactly $n$ isomorphism classes of simple modules.

We next compare contraction algebras associated to flopping contractions related by flops over the same singularity. Starting from a $3$-fold flopping contraction $f\colon X\to\operatorname{Spec}R$, one may choose an irreducible exceptional curve $C_i$ ($1\le i\le n$) of $f$, and then factor $f$ as $f=h\circ g$, where $g\colon X\to X_{\mathrm{con}}$ contracts $C_i$ to a point and is an isomorphism away from $C_i$. Then, one can replace $g$ by another small birational morphism $g^+\colon X^+\to X_{\mathrm{con}}$. We call the new $3$-fold flopping contraction
$f^+:=h\circ g^+$ obtained by composing $g^+$ with $h$ the \emph{simple flop} of $f$ at $C_i$, see \cite[Subsection 4.3]{AugustMZ20}, and repeating this procedure successively gives the \emph{iterated flops} of $f$. It is shown that this geometric operation is reflected by derived equivalences of contraction algebras. More precisely, \cite[Theorem~4.12]{AugustMZ20} shows that the basic algebras in the derived equivalence class of $\Acon$ are precisely the contraction algebras of the flopping contractions obtained from $f$ by iterated flops. 

Since contraction algebras in our setting are finite-dimensional symmetric algebras, we immediately obtain the following consequence.

\begin{proposition} \label{prop:chamber-invariance}
Let $f\colon X\to\operatorname{Spec}R$ be a $3$-fold flopping
contraction as above, and let
$f'\colon X'\to\operatorname{Spec}R$ be obtained from $f$ by iterated flops. If $\Acon$ and $\Acon'$ are the
corresponding contraction algebras, then $\Acon$ and $\Acon'$ have
the same representation type.
\end{proposition}

\subsection{Contracted preprojective algebras}\label{subsec:marked-ADE}
By the defining property of a cDV singularity, a sufficiently general
surface section of $\operatorname{Spec}R$ is a Du Val surface singularity. Its minimal resolution replaces the singular point by a configuration of exceptional curves, whose dual graph is a simply-laced Dynkin diagram $\Delta\in\{\DynA_m\ (m\ge1),\DynD_m\ (m\ge4),
\DynE_6,\DynE_7,\DynE_8\}$; see, for example, \cite[Chapter~4]{Reid97}. We write $\Delta_0$ for the vertex set of $\Delta$. Then, after taking a sufficiently general surface section, the $3$-fold flopping contraction $f\colon X\to\operatorname{Spec}R$ induces a partial resolution of the corresponding Du Val singularity, whose exceptional curves correspond to the irreducible exceptional curves $C_1,\ldots,C_n$ of $f$. We denote by $\Omega\subseteq\Delta_0$ the subset of vertices corresponding to $C_1,\ldots,C_n$ of $f$ on the induced partial resolution. Thus, $|\Omega|=n$. In the standard convention, the vertices in $\Delta_0\setminus \Omega$ are called \emph{marked} vertices.

Let $Q$ be the quiver obtained by assigning an orientation to $\Delta$, and let $\overline Q$ be the double quiver of $Q$, obtained by adjoining to every arrow $p:i\to j$ an opposite arrow $p^*:j\to i$. For each
$i\in\Delta_0$, let
\[
\rho_i:=\sum_{\substack{p\in Q_1,\; s(p)=i}}pp^*
-\sum_{\substack{p\in Q_1,\; t(p)=i}}p^*p\; ,
\]
where $s(p)$ and $t(p)$ denote the source and target of $p$, respectively. Then, the \emph{preprojective algebra} of $\Delta$ is
defined by $\Pi_\Delta:=\Bbbk\overline Q/ \langle \rho_i\mid i\in\Delta_0\rangle$. It is obvious that the isomorphism class of $\Pi_\Delta$ is independent of the chosen orientation of $\Delta$. Since $\Delta$ is Dynkin, $\Pi_\Delta$ is a finite-dimensional algebra. Let $e_i$ be the primitive idempotent of $\Pi_\Delta$ corresponding to
the vertex $i\in\Delta_0$. For the subset $\Omega$ defined above, put $e_\Omega:=\sum_{i\in\Omega}e_i$ and $\Gam(\Delta,\Omega):=e_\Omega\Pi_\Delta e_\Omega$. Then, $\Gam(\Delta,\Omega)$ is the idempotent truncation of $\Pi_\Delta$
supported on the vertices in $\Omega$. We call it the \emph{contracted preprojective algebra} associated to $(\Delta,\Omega)$; see \cite[Definition~5.1]{IyamaWemyssTits}. It follows that $\Gam(\Delta,\Omega)$ has exactly $|\Omega|=n$ isomorphism classes of simple modules. 

Elements of $e_\Omega\Pi_\Delta e_\Omega$ are represented by paths whose
sources and targets lie in $\Omega$, but such paths may pass through
vertices in $\Delta_0\setminus\Omega$. Consequently, when $\Omega\neq\Delta_0$, the algebra $\Gam(\Delta,\Omega)$ is in general not the preprojective algebra of the subdiagram induced by $\Omega$. The precise quotient relation between $\Acon$ and $\Gam(\Delta,\Omega)$ is explained in the following.

\begin{proposition}[{\cite[(3.C)]{DW16}, \cite[Subsection 3.2]{DW19};
see also \cite[Corollary~5.20]{Wemyss18}}]
\label{prop:shadow}
Let $f\colon X\to\operatorname{Spec}R$ be a $3$-fold flopping
contraction as above, and let $\Acon$ be its contraction
algebra. There exists an element $g\in Z(\Acon)\cap\rad\Acon$ such that $\Acon/g\Acon\cong\Gam(\Delta,\Omega)$. 
\end{proposition}

We therefore call $\Gam(\Delta,\Omega)$ the \emph{surface quotient} of $\Acon$. Using Proposition~\ref{prop:chamber-invariance} and the surjection $\Acon\twoheadrightarrow\Gam(\Delta,\Omega)$, one may detect representation-infiniteness or wildness of $\Acon$ through $\Gam(\Delta,\Omega)$. One should note that even if $\Gam(\Delta,\Omega)$ associated to $f$ is representation-finite or tame, $\Acon$ may still be wild; for example, wildness may be detected by the contracted preprojective algebra of an iterated flop of $f$.

\begin{corollary} \label{cor:wildness-propagation}
Let $f'\colon X'\to\operatorname{Spec}R$ be a $3$-fold flopping
contraction obtained from $f\colon X\to\operatorname{Spec}R$ by iterated flops. If the contracted preprojective algebra associated to a sufficiently general surface section of $f'$ is representation-infinite or wild, then so is the contraction algebra $\Acon$ associated to $f$.
\end{corollary}

\subsection{Main strategy}\label{subsec:main-strategy}
We now summarize the strategy for the classification in this paper. Let
$f\colon X\to\operatorname{Spec}R$ be a $3$-fold flopping contraction
in our standing setup. We shall complete the classification procedure in three steps.
\begin{itemize}
\item \textbf{Step 1: classify the contracted preprojective
algebras.}
For each simply-laced Dynkin diagram $\Delta$ and each nonempty subset $\Omega\subseteq\Delta_0$, we determine the representation type of the contracted preprojective algebra $\Gam(\Delta,\Omega)$. We first treat type~$\DynA$, then type~$\DynD$, and finally the exceptional Dynkin types. Quotient and idempotent truncation techniques, including reductions to full Dynkin subdiagrams, are used to detect representation-infiniteness and wildness. 

\item \textbf{Step 2: use iterated flops to detect further cases.}
By Proposition~\ref{prop:shadow}, every contraction algebra $\Acon$
admits the corresponding $\Gam(\Delta,\Omega)$ as a quotient. The
results of Step~1 are therefore used to detect representation-infinite
and wild contraction algebras. If representation-infiniteness
(respectively, wildness) is detected from $\Gam(\Delta,\Omega)$ corresponding to $f\colon X\to\operatorname{Spec}R$, then the
same holds for the contraction algebra associated to every iterated
flop of $f$; see Corollary~\ref{cor:wildness-propagation}. This resolves
all cases whose representation type is already detected by a surface quotient, possibly after passing to an iterated flop.

\item \textbf{Step 3: analyse the remaining contraction algebras.}
A representation-finite or tame surface quotient does not necessarily determine the representation type of the contraction algebra. The surface quotient $\Gam(\Delta,\Omega)$ determines the arrows between distinct vertices in the Gabriel quiver of $\Acon$, while the one-curve subset contraction algebras $\Acon/\Acon(1-e_i)\Acon$ detect loops at the corresponding vertices $i$, possibly after suitable iterated flops. We then use the resulting Gabriel quiver and its radical-square-zero quotient to detect wildness of $\Acon$. When this does not show wildness, we pass to the next radical layer and study the corresponding multiplication in
\begin{center}
$(\rad \Acon/\rad^2 \Acon)\otimes_{\Acon/\rad \Acon}(\rad \Acon/\rad^2 \Acon) \longrightarrow \rad^2 \Acon/\rad^3 \Acon \;.$
\end{center}
In a few loopless boundary cases, one further radical layer is required. Together with the preceding reductions, this completes the classification of all cases surviving Step~2.
\end{itemize}

\section{Contracted preprojective algebras of type~$\DynA$} \label{sec:A-stable}
In this section, we consider the representation type of contracted
preprojective algebras $\Gam(\Delta,\Omega)$ with
$\Delta=\DynA_m$. We label the vertices of $\DynA_m$ by
$\Delta_0=\{1,\ldots,m\}$ and fix a nonempty subset
$\Omega\subseteq\Delta_0$. We first reconstruct
$\Gam(\DynA_m,\Omega)$ as a stable endomorphism
algebra over a truncated polynomial algebra. This yields a simple formula for its stable Hom spaces and, after introducing suitable combinatorial parameters, a uniform quiver presentation of
$\Gam(\DynA_m,\Omega)$. 

Let $\Lambda_{m+1}:=\Bbbk[x]/(x^{m+1})$ and $M_i:=\Lambda_{m+1}/(x^i)$ for $1\le i\le m+1$. Then, $\Lambda_{m+1}$ is a representation-finite self-injective Nakayama algebra, and $M_1,\ldots,M_{m+1}$ form a complete set of pairwise non-isomorphic indecomposable $\Lambda_{m+1}$-modules. Moreover, $M_{m+1}=\Lambda_{m+1}$ is the unique indecomposable projective module. By the classical construction in \cite{DlabRingel92} (see also \cite{RingelZhang}), the stable Auslander algebra of $\Lambda_{m+1}$ is isomorphic to the preprojective algebra $\Pi_{\DynA_m}$. Under this isomorphism, the primitive idempotent corresponding to the vertex $i$ is identified with the summand $M_i$. Hence, for every nonempty subset $\Omega\subseteq\{1,\ldots,m\}$, we have
\[
\Gam(\DynA_m,\Omega)
\cong
\underline{\End}_{\Lambda_{m+1}}
\left(\bigoplus_{i\in\Omega}M_i\right)^{\op}.
\]
Since $\Lambda_{m+1}$ is uniserial, $\dim_{\Bbbk}\Hom_{\Lambda_{m+1}}(M_i,M_j)=\min\{i,j\}$. A map $M_i\to\Lambda_{m+1}$ sends $1$ into
$x^{m-i+1}\Lambda_{m+1}$. Hence, the maps $M_i\to M_j$ factoring through $\Lambda_{m+1}$ are precisely those maps whose image of $1$ lies in $x^{m-i+1}M_j$, a subspace of dimension $\max\{0,i+j-(m+1)\}$. This gives
\[
\begin{aligned}
\dim_{\Bbbk}\underlineHom_{\Lambda_{m+1}}(M_i,M_j)
&= \min\{i,j\}-\max\{0,i+j-(m+1)\}\\
&= \min\{i,j,m-i+1,m-j+1\}.
\end{aligned}
\]

The stable Hom formula above makes precise the relation between the algebras considered here and the idempotent truncations of Auslander algebras of $\Lambda_{m+1}$ studied in \cite{DrozdMazorchuk}. The main distinction is that $\Gam(\DynA_m,\Omega)$ is defined in terms of stable Hom spaces, whereas the algebras in \cite{DrozdMazorchuk} are defined using
ordinary Hom spaces. Consequently, the classification result in \cite{DrozdMazorchuk} does not apply to $\Gam(\DynA_m,\Omega)$ in
general. The two constructions agree, however, whenever every
morphism that factors through the projective module $M_{m+1}$ is zero.

We recall the relevant construction from \cite{DrozdMazorchuk} in our module-theoretic notation. In the quiver presentation used there, the vertices are numbered in the opposite direction to the lengths of the indecomposable $\Bbbk[x]/(x^\ell)$-modules. Accordingly, for
$\ell\ge2$ and $X\subseteq\{2,\ldots,\ell\}$, we set
\[
\mathcal{A}_\ell^X:=
\End_{\Bbbk[x]/(x^\ell)}
\left(
\Bbbk[x]/(x^\ell)\oplus
\bigoplus_{r\in X}\Bbbk[x]/(x^{\ell-r+1})
\right)^{\op}.
\]
Up to opposite algebras, we may translate \cite[Theorem~2.1]{DrozdMazorchuk} directly into our notation.

\begin{proposition}\label{prop:oneside}
Suppose $\varnothing\neq\Omega\subseteq\{1,2,\ldots,\lfloor(m+1)/2\rfloor\}$ and 
$\ell:=\max\Omega\ge2$. Then, $\Gam(\DynA_m,\Omega)\cong\mathcal{A}_\ell^X$ with $X=\{\ell+1-i\mid i\in\Omega\setminus\{\ell\}\}$. Consequently,
\begin{enumerate}
\item $\Gam(\DynA_m,\Omega)$ is representation-finite if $\Omega\in \{\{1,\ell\}, \{\ell-1,\ell\}, \{1,\ell-1,\ell\}, \{\ell\}\}$.

\item $\Gam(\DynA_m,\Omega)$ is tame if either $\ell\ge4$ and $\Omega\in
\bigl\{
\{1,2,\ell\},
\{2,\ell\},
\{\ell-2,\ell\},
\{\ell-2,\ell-1,\ell\}
\bigr\}$, or $\ell=4$ and $\Omega=\{1,2,3,4\}$.

\item $\Gam(\DynA_m,\Omega)$ is wild in all remaining cases.
\end{enumerate}
\end{proposition}
\begin{proof}
For any $i,j\in\Omega$, we have $i+j\le2\ell\le m+1$. Hence, $\max\{0,i+j-(m+1)\}=0$. By the stable Hom formula above, no nonzero morphism
$M_i\to M_j$ factors through the projective module $M_{m+1}$. Since $i\le\ell$ for every $i\in\Omega$, each $M_i$ is annihilated
by $x^\ell$. Hence, the module structures factor through $\Bbbk[x]/(x^\ell)$, without changing the Hom spaces or their compositions. Under the convention in \cite{DrozdMazorchuk}, the summand $\Bbbk[x]/(x^\ell)$ corresponds to vertex $1$, while $\Bbbk[x]/(x^i)$ corresponds to vertex $\ell+1-i$. Therefore, $\Gam(\DynA_m,\Omega)\cong \mathcal{A}_\ell^X$ for $X=\{\ell+1-i\mid i\in\Omega\setminus\{\ell\}\}$. The remaining statements are direct translations of
\cite[Theorem~2.1]{DrozdMazorchuk}.
\end{proof}

\subsection{Quiver presentation}
Contracted preprojective algebras have been studied in \cite{IyamaWemyssTits} and, more recently, in \cite{ChanIyamaMarczinzik}. For example, an explicit presentation for a particular contracted
preprojective algebra of type~$\DynA$ appears in \cite[Example~4.9]{ChanIyamaMarczinzik}. 

We suppose $|\Omega|=r$ and write $\Omega=\{j_1<j_2<\cdots<j_r\}\subseteq\{1,\ldots,m\}$. Set $d_0:=j_1$, $d_s:=j_{s+1}-j_s$ for $1\le s\le r-1$, and $d_r:=m-j_r+1$. Then, each $d_s$ is positive and $d_0+d_1+\cdots+d_r=m+1$. Moreover, $\Omega$ is recovered from the \emph{parameter sequence} $(d_0,\ldots,d_r)$ by $j_s=d_0+\cdots+d_{s-1}$ for $1\le s\le r$. We now give a uniform presentation for arbitrary nonempty subsets $\Omega$.

\begin{theorem}
\label{thm:uniformA}
Suppose $\Omega=\{j_1<j_2<\cdots<j_r\}$. If $r=1$, then $\Gam(\DynA_m,\Omega) \cong \Bbbk[x]/(x^{\min\{d_0,d_1\}})$. If $r\ge 2$, then $\Gam(\DynA_m,\Omega)$ is isomorphic to $\Bbbk Q/I$ with the quiver $Q$:
\[
\vcenter {\xymatrix@C=4em{
1
\ar@(ul,dl)[]_{x_1}
\ar@/^/[r]^{a_1}
&
2
\ar@(ul,ur)[]^{x_2}
\ar@/^/[l]^{b_1}
\ar@/^/[r]^{a_2}
&
\cdots
\ar@/^/[r]^{a_{r-2}}
\ar@/^/[l]^{b_2}
&
r-1
\ar@(ul,ur)[]^{x_{r-1}}
\ar@/^/[l]^{b_{r-2}}
\ar@/^/[r]^{a_{r-1}}
&
r
\ar@(ur,dr)[]^{x_r}
\ar@/^/[l]^{b_{r-1}}
}}
\]
modulo the ideal $I$ generated by the relations 
\[
x_1^{d_0}=0,\quad  x_r^{d_r}=0, \quad x_sa_s=a_sx_{s+1}, \quad b_sx_s=x_{s+1}b_s, \quad a_sb_s=x_s^{d_s},\quad b_sa_s=x_{s+1}^{d_s},
\]
where the last four relations hold for all $1\le s\le r-1$.
\end{theorem}
\begin{proof}
If $r=1$, then $\Omega=\{j_1\}$ and $\Gam(\DynA_m,\Omega) \cong \underline{\End}_{\Lambda_{m+1}}(M_{j_1})^{\op}$. The stable endomorphisms are generated by multiplication by $x$, and the stable Hom calculation above shows that its nilpotency index is $\min\{j_1,m-j_1+1\}=\min\{d_0,d_1\}$. Hence, $\Gam(\DynA_m,\Omega) \cong \Bbbk[x]/(x^{\min\{d_0,d_1\}})$. We therefore assume $r\ge2$ in the following.

For each $1\le s\le r$, we denote by the loop $x_s$ for the multiplication by $x$ on $M_{j_s}$. Then, the arrow $a_s:s\to s+1$ corresponds to the canonical quotient map $M_{j_{s+1}}\longrightarrow M_{j_s}$ for $1\le s\le r-1$, whereas the arrow $b_s:s+1\to s$ corresponds to the map $M_{j_s}\longrightarrow M_{j_{s+1}}$ given by $1\longmapsto x^{d_s}$. Since all these maps commute with multiplication by $x$, we obtain $x_sa_s=a_sx_{s+1}$ and $b_sx_s=x_{s+1}b_s$. Moreover, the two compositions corresponding to $a_sb_s$ and
$b_sa_s$ are multiplication by $x^{d_s}$ on $M_{j_s}$ and $M_{j_{s+1}}$, respectively. Hence, $a_sb_s=x_s^{d_s}$ and $b_sa_s=x_{s+1}^{d_s}$. At the left endpoint, $d_0=j_1$, so multiplication by
$x^{d_0}=x^{j_1}$ vanishes on $M_{j_1}$. At the right endpoint, $d_r=m-j_r+1$, and multiplication by $x^{m-j_r+1}$ on $M_{j_r}$ factors as $M_{j_r}\longrightarrow M_{m+1}=\Lambda_{m+1} \longrightarrow M_{j_r}$, where the first map sends $1$ to $x^{m-j_r+1}$. Since $M_{m+1}$ is projective, this map is zero in the stable category. Therefore, $x_1^{d_0}=x_r^{d_r}=0$. Hence, the assignments above induce an algebra homomorphism $\Bbbk Q/I\longrightarrow \Gam(\DynA_m,\Omega)$, where $I$ is the ideal generated by the stated relations. It remains to show that this homomorphism is an isomorphism. 

We first record the nilpotency of the loops. Starting from the left endpoint, we have $x_1^{d_0}=0$; assuming moreover $x_{s-1}^{j_{s-1}}=0$ for $s>1$ and using $j_s=j_{s-1}+d_{s-1}$, we obtain
\[
x_s^{j_s}
=
x_s^{j_{s-1}}x_s^{d_{s-1}}
=
x_s^{j_{s-1}}b_{s-1}a_{s-1}
=
b_{s-1}x_{s-1}^{j_{s-1}}a_{s-1}
=
0.
\]
This gives us $x_s^{j_s}=0$ for all $1\le s\le r$. Similarly, starting from the right endpoint gives $x_s^{m-j_s+1}=0$ for all $1\le s\le r$. Using the commutation relations, powers of the loops can be moved along paths consisting entirely of arrows in the same direction. Moreover, whenever a path changes direction, it contains a subpath
$a_ib_i$ or $b_ia_i$, which can be replaced by
$x_i^{d_i}$ or $x_{i+1}^{d_i}$, respectively.
Then, every path can be reduced to a power of a loop followed by a path which moves only in one direction. Hence, for $s<t$, the paths from $s$ to $t$ in $\Bbbk Q/I$ are spanned by $x_s^q a_sa_{s+1}\cdots a_{t-1}$ for $0\le q<\min\{j_s,m-j_t+1\}$. Under the homomorphism above, these paths are sent to the stable maps $M_{j_t}\to M_{j_s}$ given by $1\mapsto x^q$. These stable maps are linearly independent, and their number is
\[
\min\{j_s,m-j_t+1\} = \dim_{\Bbbk}
\underlineHom_{\Lambda_{m+1}}(M_{j_t},M_{j_s}).
\]
Hence, they form a basis of the corresponding stable Hom space. The reverse off-diagonal blocks are treated in the same way. For the diagonal block at the vertex $s$, the paths $1,x_s,\ldots, x_s^{\min\{j_s,m-j_s+1\}-1}$ form a basis, in agreement with $\dim_{\Bbbk}
\underlineEnd_{\Lambda_{m+1}}(M_{j_s}) = \min\{j_s,m-j_s+1\}$. Thus, the homomorphism $\Bbbk Q/I\longrightarrow \Gam(\DynA_m,\Omega)$ is an isomorphism on every matrix block, and hence is an isomorphism.
\end{proof}

\begin{remark}
For $r\ge2$, the presentation in Theorem~\ref{thm:uniformA} is not necessarily a bound quiver presentation in the usual sense, since the ideal generated by the displayed relations need not be admissible. In fact, $I$ is admissible if and only if $d_0,d_1,\ldots,d_r\ge2$. The Gabriel quiver of $\Gam(\DynA_m,\Omega)$ is obtained from the quiver in Theorem~\ref{thm:uniformA} by deleting precisely those loops which
become redundant modulo $\rad^2 \Gam(\DynA_m,\Omega)$. More explicitly, the arrows
$a_s$ and $b_s$ always survive in $\rad \Gam(\DynA_m,\Omega)/\rad^2 \Gam(\DynA_m,\Omega)$, whereas the loop $x_s$ survives if and only if
$d_{s-1}\ge2$ and $d_s\ge2$.
\end{remark}

The presentation in Theorem~\ref{thm:uniformA} gives two elementary reductions that will be
used repeatedly. The first is obtained by killing all vertices outside a consecutive interval, while the second is obtained by passing to a smaller idempotent truncation. 

\begin{proposition} \label{prop:intervalquot}
Suppose $\Omega=\{j_1<j_2<\cdots<j_r\}$ and $1\le s\le t\le r$. Let $e_{j_u}$ be the primitive idempotent of $\Gam(\DynA_m,\Omega)$ corresponding to $j_u\in\Omega$. Then,
\[
\Gam(\DynA_m,\Omega)/\left\langle e_{j_u}\mid u\notin\{s,\ldots,t\}\right\rangle  \cong
\Gam\left(
\DynA_{m'},\Omega'
\right),
\]
where $m'=d_{s-1}+d_s+\cdots+d_t-1$ and $\Omega'=\left\{ d_{s-1}, \, d_{s-1}+d_s, \, \ldots, \, d_{s-1}+d_s+\cdots+d_{t-1} \right\}$.
\end{proposition}
\begin{proof}
In the presentation of Theorem~\ref{thm:uniformA}, we may kill the primitive idempotents corresponding to $j_u$ with $u\notin\{s,\ldots,t\}$. If $s>1$, then the arrows
$a_{s-1}$ and $b_{s-1}$ vanish, so that $b_{s-1}a_{s-1}=x_s^{d_{s-1}}$ becomes $x_s^{d_{s-1}}=0$; for $s=1$, this is already the left
endpoint relation $x_1^{d_0}=0$. Similarly, if $t<r$, the relation $a_tb_t=x_t^{d_t}$ becomes $x_t^{d_t}=0$; for $t=r$, this is the right endpoint relation $x_r^{d_r}=0$. All relations involving only the vertices $s,\ldots,t$ remain
unchanged. Thus, the resulting presentation has parameter sequence $d_{s-1},d_s,\ldots,d_t$. Its underlying Dynkin diagram has $d_{s-1}+d_s+\cdots+d_t-1$ vertices, and the remaining vertices are $d_{s-1}$, $d_{s-1}+d_s$, $\ldots$, $d_{s-1}+\cdots+d_{t-1}$.
\end{proof}

\begin{proposition}
\label{prop:coarsen}
Let $\varnothing\neq\Omega'\subseteq\Omega\subseteq \{1,\ldots,m\}$. Then, $\Gam(\DynA_m,\Omega')=e_{\Omega'}\Gam(\DynA_m,\Omega)e_{\Omega'}$. Suppose moreover that
$\Omega=\{j_1<\cdots<j_r\}$ with $r\ge2$, and let
$\Omega'=\Omega\setminus\{j_s\}$. Then, the parameters associated to $\Omega'$ are $d_0,\; \ldots,\; d_{s-2},\; d_{s-1}+d_s,\; d_{s+1},\; \ldots,d_r$.
\end{proposition}
\begin{proof}
Since $\Omega'\subseteq\Omega$, we have
$e_{\Omega'}e_\Omega=e_{\Omega'}$. Hence, $e_{\Omega'} \bigl(e_\Omega\Pi_{\DynA_m}e_\Omega\bigr)
e_{\Omega'} = e_{\Omega'}\Pi_{\DynA_m}e_{\Omega'}$, which gives the required equality. Recall that the parameters associated to $\Omega$ are $d_0,\; \ldots, \; d_{s-1},\; d_s,\; d_{s+1},\; \ldots,\; d_r$. Deleting $j_s$ removes the break between $d_{s-1}$ and $d_s$, so these two consecutive intervals are replaced by a single interval of length $d_{s-1}+d_s$. This gives the stated parameter sequence.
\end{proof}

We now apply the structural results above to determine the representation type of $\Gam(\DynA_m,\Omega)$.

\subsection{Two vertices}\label{sec:two-point}
Suppose $\Omega=\{j_1<j_2\}$. Let $d_0=j_1$, $d_1=j_2-j_1$, and $d_2=m-j_2+1$ as above.
By Theorem~\ref{thm:uniformA},
$\Gam(\DynA_m,\Omega)$ is isomorphic to $\Bbbk Q/I$ with
\[
\vcenter{\xymatrix@C=4em{ 1 \ar@(ul,dl)[]_{x_1} \ar@/^/[r]^{a_1} & 2 \ar@(ur,dr)[]^{x_2} \ar@/^/[l]^{b_1}
}}
\]
bounded by $x_1^{d_0}=0,\; x_2^{d_2}=0,\; 
x_1a_1=a_1x_2,\; b_1x_1=x_2b_1,\; a_1b_1=x_1^{d_1},\; b_1a_1=x_2^{d_1}$. 

\begin{proposition}\label{thm:two-point}
Suppose $\Omega=\{j_1<j_2\}$ and $d_0$, $d_1$, $d_2$ as above. Then, $\Gam(\DynA_m,\Omega)$ is representation-finite if
$\min\{d_0,d_1,d_2\}=1$, tame if $\min\{d_0,d_1,d_2\}=2$, and wild if $\min\{d_0,d_1,d_2\}\ge3$.
\end{proposition}
\begin{proof}
Consider the automorphism of the Dynkin diagram $\DynA_m$ given by $i\mapsto m+1-i$. It induces an automorphism of $\Pi_{\DynA_m}$ and hence an isomorphism
\[
\Gam(\DynA_m,\{j_1,j_2\}) \cong
\Gam(\DynA_m,\{m+1-j_2,m+1-j_1\}).
\]
Under this automorphism, the parameter sequence $(d_0,d_1,d_2)$ is replaced by $(d_2,d_1,d_0)$. Hence, if $d_2<\min\{d_0,d_1\}$, we may replace the algebra by an isomorphic one for which the first parameter is minimal. It then suffices to
consider $\min\{d_0,d_1,d_2\}=\min\{d_0,d_1\}$.

Set $\ell:=d_0+d_1$ and choose an integer
$k\ge\max\{d_2,\ell\}$. Let $\Omega':=\{j_1'<j_2'\}:=\{d_0,\ell\}$ and consider $\Gam':=\Gam(\DynA_{\ell+k-1},\Omega')$. The parameter sequence associated to $\Omega'$ is
$(d_0,d_1,k)$. Since $k\ge\ell=j_2'$, the subset
$\Omega'=\{j_1',j_2'\}$ is contained in the left half of the corresponding Dynkin diagram, so Proposition~\ref{prop:oneside} applies.

Suppose first that $\min\{d_0,d_1,d_2\}=1$. By the reduction above, either $d_0=1$ or $d_1=1$. If $d_0=1$, then $\Omega'=\{1,\ell\}$; if $d_1=1$, then $d_0=\ell-1$ and hence $\Omega'=\{\ell-1,\ell\}$. Thus, $\Gam'$ is representation-finite by Proposition~\ref{prop:oneside}. By Theorem~\ref{thm:uniformA}, the presentations of $\Gam'$ and $\Gam(\DynA_m,\Omega)$ have the same quiver and the same relations except for the right endpoint relation, which is $x_2^k=0$ for
$\Gam'$ and $x_2^{d_2}=0$ for $\Gam(\DynA_m,\Omega)$. Since
$k\ge d_2$, we have $\Gam(\DynA_m,\Omega) \cong \Gam'/\langle x_2^{d_2}\rangle$. Therefore, $\Gam(\DynA_m,\Omega)$ is representation-finite.

Suppose next that $\min\{d_0,d_1,d_2\}=2$. We have again either $d_0=2$ or $d_1=2$. Accordingly,
$\{d_0,\ell\}=\{2,\ell\}$ or
$\{d_0,\ell\}=\{\ell-2,\ell\}$. Since $\ell\ge4$,
Proposition~\ref{prop:oneside} shows that $\Gam'$ is tame. As above,
$\Gam(\DynA_m,\Omega)\cong\Gam'/\langle x_2^{d_2}\rangle$, and hence
$\Gam(\DynA_m,\Omega)$ is either representation-finite or tame. Since $d_0,d_1,d_2\ge2$, both loops $x_1,x_2$ and both arrows $a_1,b_1$ survive in the Gabriel quiver of $\Gam(\DynA_m,\Omega)$. Therefore, $\Gam(\DynA_m,\Omega)/\rad^2\Gam(\DynA_m,\Omega)$ has separated quiver of Euclidean type $\widetilde{\DynA}_3$. Hence, $\Gam(\DynA_m,\Omega)$ is representation-infinite, and therefore it is tame.

Suppose $d_0,d_1,d_2\ge3$. Then, the quotient algebra $\Gam(\DynA_m,\Omega)/\langle x_1^3,x_2^3,b_1x_1,x_1^2a_1\rangle$ is exactly the minimal wild two-point algebra $W_{33}$ in \cite[Table~W]{Han02}. Hence, $\Gam(\DynA_m,\Omega)$ is wild.
\end{proof}

\subsection{Three vertices} \label{sec:A-three}
Suppose $\Omega=\{j_1<j_2<j_3\}$ with associated parameters $d_0,d_1,d_2,d_3$. By Theorem~\ref{thm:uniformA}, $\Gam(\DynA_m,\Omega)$ is isomorphic to $\Bbbk Q/I$, where $Q$ is presented as
\[
\vcenter{\xymatrix@C=4em{
1
\ar@(ul,dl)[]_{x_1}
\ar@/^/[r]^{a_1}
&
2
\ar@(ul,ur)[]^{x_2}
\ar@/^/[l]^{b_1}
\ar@/^/[r]^{a_2}
&
3
\ar@(ur,dr)[]^{x_3}
\ar@/^/[l]^{b_2}
}}
\]
and $I$ is generated by the following relations:
\[
\begin{matrix}
x_1^{d_0}=0,\quad x_1a_1=a_1x_2,\quad b_1x_1=x_2b_1,\quad
a_1b_1=x_1^{d_1},\quad b_1a_1=x_2^{d_1},\\[3pt]
x_3^{d_3}=0,\quad x_2a_2=a_2x_3,\quad b_2x_2=x_3b_2,\quad
a_2b_2=x_2^{d_2},\quad b_2a_2=x_3^{d_2}.
\end{matrix}
\]
We first use the radical-square-zero quotient to detect most of the wild cases.

\begin{proposition}\label{prop:radicalA}
If $d_1,d_2\ge2$ and at least one of $d_0,d_3$ is at least $2$, then $\Gam(\DynA_m,\Omega)$ is wild.
\end{proposition}
\begin{proof}
The diagram automorphism of $\DynA_m$ given by $i\mapsto m+1-i$ replaces the parameter sequence $(d_0,d_1,d_2,d_3)$ by $(d_3,d_2,d_1,d_0)$. Hence, we may assume $d_0\ge2$, so that $d_0,d_1,d_2\ge2$. Then, both loops $x_1$ and $x_2$ survive in the Gabriel quiver of $\Gam(\DynA_m,\Omega)$. Hence, the separated quiver of $\Gam(\DynA_m,\Omega)/\rad^2\Gam(\DynA_m,\Omega)$ contains the connected subquiver
\[
\vcenter{\xymatrix@C=2em@R=1em{
& 1^+ \ar[dl] \ar[dr] & & 3^+ \ar[dl] \\
1^- & & 2^- & \\
& 2^+ \ar[ul] \ar[ur] &
}}.
\]
Its underlying graph contains a $4$-cycle with an additional vertex attached to it, and hence is neither Dynkin nor Euclidean. Therefore,
$\Gam(\DynA_m,\Omega)/\rad^2\Gam(\DynA_m,\Omega)$ is wild.
\end{proof}

\begin{proposition}\label{prop:1HH1tame}
Suppose $d_1,d_2\ge2$ and $d_0=d_3=1$. Then,
$\Gam(\DynA_m,\Omega)$ is tame.
\end{proposition}
\begin{proof}
Since $d_1,d_2\ge2$, the loop $x_2$ survives in the Gabriel quiver of $\Gam(\DynA_m,\Omega)$. Together with the four arrows $a_1,b_1,a_2,b_2$, this implies that the separated quiver of
$\Gam(\DynA_m,\Omega)/\rad^2\Gam(\DynA_m,\Omega)$ is of Euclidean type $\widetilde{\DynD}_5$. Hence, $\Gam(\DynA_m,\Omega)$ is representation-infinite. By the diagram automorphism as before, which exchanges $d_1$ and $d_2$ in this case, we may assume $d_1\le d_2$.

Since $d_0=d_3=1$, the endpoint loops $x_1,x_3$ vanish, and the relations reduce to
\[
a_1b_1=0,\quad b_1a_1=x_2^{d_1},\quad
b_2a_2=0,\quad a_2b_2=x_2^{d_2},\quad
a_1x_2=x_2b_1=x_2a_2=b_2x_2=0.
\]
We obtain $x_2^{d_1+1}=b_1a_1x_2=0$. Hence, if $d_1<d_2$, then $a_2b_2=0$, whereas $a_2b_2=x_2^{d_1}$ if $d_1=d_2$. If $d_1<d_2$, replace the relation $b_1a_1=x_2^{d_1}$ by
$b_1a_1=\lambda x_2^{d_1}$. If $d_1=d_2$, replace the two relations $b_1a_1=x_2^{d_1}$ and $a_2b_2=x_2^{d_1}$ by $b_1a_1=\lambda x_2^{d_1}$ and $a_2b_2=\lambda x_2^{d_1}$, respectively. In either case, this gives an algebraic family $\{A_\lambda\mid \lambda\in \Bbbk\}$ with $A_\lambda\cong \Gam(\DynA_m,\Omega)$ for all $\lambda\neq 0$, and $A_0$ is the algebra with the same quiver of $\Gam(\DynA_m,\Omega)$ bounded by the following relations
\[
a_1b_1=b_1a_1=a_2b_2=b_2a_2=0,\quad
a_1x_2=x_2b_1=x_2a_2=b_2x_2=0,\quad
x_2^{d_1+1}=0.
\]
Up to relabelling the arrows and taking the opposite algebra if necessary, $A_0$ is a quotient of the tame semidihedral algebra $SD(3C)_{2,I/II}$ from Erdmann's classification \cite[Chapter~VIII]{Erdmann90}. On the other hand, the radical-square-zero quotient of $A_0$ has a separated quiver of Euclidean type $\widetilde{\DynD}_5$. Thus,
$A_0$ is representation-infinite, and hence tame. We have shown that $\Gam(\DynA_m,\Omega)$ degenerates to the tame algebra $A_0$, and hence it is tame.
\end{proof}

\begin{proposition}\label{prop:H1HH}
If $d_1=1$, $d_0,d_2,d_3\ge2$, or $d_2=1$, $d_0,d_1,d_3\ge2$, then $\Gam(\DynA_m,\Omega)$ is wild.
\end{proposition}
\begin{proof}
We only prove the case $d_1=1$; the other case follows from the diagram automorphism $i\mapsto m+1-i$. In this case, $a_1b_1=x_1$ and $b_1a_1=x_2$, so the loops $x_1$ and
$x_2$ can be removed. Thus, $\Gam(\DynA_m,\Omega)$ has the quiver
\[
\vcenter{\xymatrix@C=4em{
1
\ar@/^/[r]^{a_1}
&
2
\ar@/^/[l]^{b_1}
\ar@/^/[r]^{a_2}
&
3
\ar@(ur,dr)[]^{x_3}
\ar@/^/[l]^{b_2}
}}
\]
with relations $(a_1b_1)^{d_0}=0$, $x_3^{d_3}=0$, $b_1a_1a_2=a_2x_3$, $b_2b_1a_1=x_3b_2$, $a_2b_2=(b_1a_1)^{d_2}$, $b_2a_2=x_3^{d_2}$. 

We set $B:=\Gam(\DynA_m,\Omega)/
\langle (a_1b_1)^2,x_3^2\rangle$. Then,
$(b_1a_1)^3=b_1(a_1b_1)^2a_1=0$. Consequently, if $d_2=2$, the remaining relations include
$a_2b_2=(b_1a_1)^2$ and $b_2a_2=0$, whereas if $d_2\ge3$, they reduce to $a_2b_2=b_2a_2=0$. In either case, assign the degrees
\[
\deg a_1=\deg a_2=0,\qquad
\deg b_1=\deg x_3=1,\qquad
\deg b_2=2.
\]
All defining relations of $B$ are homogeneous with respect to this $\mathbb Z$-grading. Let
$\mathcal F:\widetilde B\longrightarrow B$ be the associated Galois covering. The finite full convex subcategory of $\widetilde B$ on the vertices $3_0,\ 3_1,\ 2_1,\ 1_1,\ 1_2,\ 2_2$ has quiver
\[
\vcenter{\xymatrix@C=3em@R=2.5em{
&& 1_1 \ar[d]^{a_1} & \\
3_0
\ar[r]^{x_3}
\ar@/_1.2pc/[rrrr]_{b_2}
&
3_1
&
2_1 \ar[l]_{a_2} \ar[r]^{b_1}
&
1_2 \ar[r]^{a_1}
&
2_2
}}
\]
and no defining relation occurs on this full subcategory. Its underlying graph is neither Dynkin nor Euclidean, and hence this full subcategory is hereditary wild. By Proposition~\ref{prop:covering-wild}, $B$ is wild. Since $B$ is a quotient of $\Gam(\DynA_m,\Omega)$, the latter is wild.
\end{proof}

\begin{proposition}\label{prop:r3table}
Suppose $\Omega=\{j_1<j_2<j_3\}$ and $d_0,d_1,d_2,d_3$ as before. Then, $\Gam(\DynA_m,\Omega)$ is 
\begin{enumerate}
\item wild if at least three of $d_0,d_1,d_2,d_3$ are at least $2$;

\item tame if exactly two of $d_0,d_1,d_2,d_3$ are at least $2$ and their indices are one of
$\{0,1\},\; \{1,2\},\; \{2,3\},\; \{0,3\}$.

\item representation-finite if exactly two of $d_0,d_1,d_2,d_3$ are at least $2$ and their indices are one of $\{0,2\},\; \{1,3\}$, or if at most one of $d_0,d_1,d_2,d_3$ is at least $2$.
\end{enumerate}
\end{proposition}
\begin{proof}
The wild cases follow immediately from Propositions~\ref{prop:radicalA} and~\ref{prop:H1HH}. 

We consider the representation-finite cases; they are precisely the cases in which $d_0=d_2=1$ or $d_1=d_3=1$. By the diagram automorphism of $\DynA_m$, it suffices to consider $d_0=d_2=1$.
Set $\ell:=d_0+d_1+d_2=d_1+2$ and choose $k\ge\max\{d_3,\ell\}$. Let $\Omega':=\{j_1'<j_2'<j_3'\}:=\{1,\ell-1,\ell\}$ and set $\Gam':=\Gam(\DynA_{\ell+k-1},\Omega')$. The parameter sequence associated to $\Omega'$ is $(1,d_1,1,k)$. Since $k\ge\ell=j_3'$, the subset $\Omega'$ is contained in the left half of the corresponding Dynkin diagram $\DynA_{\ell+k-1}$. Hence, Proposition~\ref{prop:oneside} shows that $\Gam'$ is representation-finite. By Theorem~\ref{thm:uniformA}, $\Gam(\DynA_m,\Omega) \cong
\Gam'/\langle x_3^{d_3}\rangle$, and hence, $\Gam(\DynA_m,\Omega)$ is representation-finite.

It remains to consider the tame cases. The case
$d_0=d_3=1$ and $d_1,d_2\ge2$ is Proposition~\ref{prop:1HH1tame}. Suppose $d_0=d_1=1, d_2,d_3\ge2$. Set $\ell:=d_0+d_1+d_2=d_2+2$ and choose
$k\ge\max\{d_3,\ell\}$. Let $\Gam':=\Gam(\DynA_{\ell+k-1},\Omega')$ with $\Omega':=\{1,2,\ell\}$, and then the parameter sequence of $\Gam'$ is $(1,1,d_2,k)$. Since
$k\ge\ell$, Proposition~\ref{prop:oneside} applies and shows that $\Gam'$ is tame. As above, $\Gam(\DynA_m,\Omega) \cong \Gam'/\langle x_3^{d_3}\rangle$, so $\Gam(\DynA_m,\Omega)$ is either representation-finite or tame. By Proposition~\ref{prop:coarsen}, deleting $j_1$ gives the two-vertex idempotent truncation with parameter sequence $(2,d_2,d_3)$, which is representation-infinite by Proposition~\ref{thm:two-point}. Hence, $\Gam(\DynA_m,\Omega)$ cannot be representation-finite, and therefore
it is tame. The case $d_0,d_1\ge2$ and $d_2=d_3=1$ follows from the diagram automorphism.

Suppose $d_0,d_3\ge2, d_1=d_2=1$. Set $\ell:=d_0+d_1+d_2=d_0+2$ and choose $k\ge\max\{d_3,\ell\}$. Let $\Gam':=\Gam(\DynA_{\ell+k-1},\Omega')$ with $\Omega':=\{\ell-2,\ell-1,\ell\}$, so the parameter sequence of $\Gam'$ is $(d_0,1,1,k)$. Then, Proposition~\ref{prop:oneside} shows that $\Gam'$ is tame, and we have $\Gam(\DynA_m,\Omega) \cong \Gam'/\langle x_3^{d_3}\rangle$. By Proposition~\ref{prop:coarsen}, deleting $j_2$ gives the two-vertex idempotent truncation with parameter sequence $(d_0,2,d_3)$, which is representation-infinite by Proposition~\ref{thm:two-point}. Hence, $\Gam(\DynA_m,\Omega)$ is tame.
\end{proof}

\subsection{Four vertices}\label{sec:A-four}
Suppose $\Omega=\{j_1<j_2<j_3<j_4\}$ with associated parameters $d_0,\ldots,d_4$. By Theorem~\ref{thm:uniformA}, $\Gam(\DynA_m,\Omega)$ is isomorphic to $\Bbbk Q/I$, where $Q$ is presented as
\[
\vcenter{\xymatrix@C=4em{
1
\ar@(ul,dl)[]_{x_1}
\ar@/^/[r]^{a_1}
&
2
\ar@(ul,ur)[]^{x_2}
\ar@/^/[l]^{b_1}
\ar@/^/[r]^{a_2}
&
3
\ar@(ul,ur)[]^{x_3}
\ar@/^/[l]^{b_2}
\ar@/^/[r]^{a_3}
&
4
\ar@(ur,dr)[]^{x_4}
\ar@/^/[l]^{b_3}
}}
\]
and $I$ is generated by the following relations:
\[
\begin{matrix}
x_1^{d_0}=0,\quad
x_1a_1=a_1x_2,\quad b_1x_1=x_2b_1,\quad
a_1b_1=x_1^{d_1},\quad b_1a_1=x_2^{d_1},\\[3pt]
x_2a_2=a_2x_3,\quad b_2x_2=x_3b_2,\quad
a_2b_2=x_2^{d_2},\quad b_2a_2=x_3^{d_2},\\[3pt]
x_3a_3=a_3x_4,\quad b_3x_3=x_4b_3,\quad
a_3b_3=x_3^{d_3},\quad b_3a_3=x_4^{d_3},\quad
x_4^{d_4}=0.
\end{matrix}
\]

Suppose first that at least three of $d_0,\ldots,d_4$ are at least $2$. If exactly three are at least $2$, choose $j_s\in\Omega$ so that at least one of $d_{s-1},d_s$ is equal to $1$; such a choice is always possible. If at least four parameters are at least $2$, any $j_s$
may be chosen. By Proposition~\ref{prop:coarsen}, deleting $j_s$ replaces $d_{s-1},d_s$ by $d_{s-1}+d_s$, and the resulting parameter sequence still has at least three entries at least $2$. Hence, the corresponding idempotent truncation is wild by Proposition~\ref{prop:r3table}, and therefore $\Gam(\DynA_m,\Omega)$ is wild.

Suppose exactly two of $d_0,\ldots,d_4$ are at least $2$. If $d_{s-1}=d_s=1$ for some $1\le s\le4$, then deleting $j_s$ replaces $d_{s-1}$ and $d_s$ by their sum $2$, by Proposition~\ref{prop:coarsen}. The resulting parameter sequence therefore has at least three entries at least $2$, and the corresponding idempotent truncation is wild by Proposition~\ref{prop:r3table}. Since exactly three of $d_0,\ldots,d_4$ are equal to $1$, the only case in which no two of them are consecutive is $d_0=d_2=d_4=1, d_1,d_3\ge2$.

\begin{proposition}\label{prop:1H1H1wild}
Suppose $d_0=d_2=d_4=1$ and $d_1,d_3\ge2$. Then, $\Gam(\DynA_m,\Omega)$ is wild.
\end{proposition}
\begin{proof}
Since $d_0=d_4=1$, we have $x_1=x_4=0$. Since $d_2=1$, the relations $a_2b_2=x_2$ and $b_2a_2=x_3$ allow us to remove the loops $x_2$ and $x_3$. Substituting these identities into the presentation of Theorem~\ref{thm:uniformA}, we obtain the quiver
\[
\vcenter{\xymatrix@C=4em{
1
\ar@/^/[r]^{a_1}
&
2
\ar@/^/[l]^{b_1}
\ar@/^/[r]^{a_2}
&
3
\ar@/^/[l]^{b_2}
\ar@/^/[r]^{a_3}
&
4
\ar@/^/[l]^{b_3}
}}
\]
with defining relations
\[
\begin{matrix}
a_1b_1=0,\qquad
b_1a_1=(a_2b_2)^{d_1},\qquad
a_3b_3=(b_2a_2)^{d_3},\qquad
b_3a_3=0,\\
a_1a_2b_2=0,\qquad
a_2b_2b_1=0,\qquad
b_2a_2a_3=0,\qquad
b_3b_2a_2=0.
\end{matrix}
\]
Give this algebra the $\mathbb Z$-grading
$$
\deg a_1=\deg a_2=\deg a_3=0,\qquad
\deg b_2=1,\qquad
\deg b_1=d_1,\qquad
\deg b_3=d_3.
$$
Then, the displayed relations are obviously homogeneous. Let $\mathcal F:\widetilde B\to\Gam(\DynA_m,\Omega)$ be the associated Galois covering. Consider the full subcategory $\mathcal C$ of $\widetilde B$ on the vertices $3_{-1},\; 2_0,\; 3_0,\; 2_1,\; 1_{d_1},\; 4_{-d_3},\; 1_1$, in which the arrows are precisely
\[
\vcenter{\xymatrix@C=3.5em@R=2.5em{
& 1_{d_1} & 4_{-d_3} \ar[d]_{b_3} &  & \\
3_{-1} \ar[r]^{b_2} & 2_0 \ar[u]^{b_1} \ar[r]^{a_2} & 3_0 \ar[r]^{b_2} & 2_1 & 1_1 \ar[l]_{a_1}
}}
\]
The subcategory $\mathcal C$ is convex. Moreover, no lift of a defining relation is supported entirely on $\mathcal C$. Hence, $\mathcal C$ is
the path category of the displayed acyclic quiver, and in particular is hereditary. It is obvious that the underlying graph of $\mathcal C$ is neither Dynkin nor Euclidean, and therefore $\mathcal C$ is hereditary wild. By
Proposition~\ref{prop:covering-wild}, the algebra
$\Gam(\DynA_m,\Omega)$ is wild.
\end{proof}

We next consider the cases in which exactly one of
$d_0,\ldots,d_4$ is at least $2$. The diagram automorphism of $\DynA_m$ given by $i\mapsto m+1-i$ reverses the parameter sequence $(d_0,d_1,d_2,d_3,d_4)$ to $(d_4,d_3,d_2,d_1,d_0)$. Hence, up to this automorphism, the unique parameter at least $2$
may be assumed to be $d_4$, $d_2$, or $d_1$.

\begin{proposition}\label{prop:H1111tame}
Suppose $d_0=d_1=d_2=d_3=1$ and $d_4\ge2$. Then, $\Gam(\DynA_m,\Omega)$ is tame.
\end{proposition}
\begin{proof}
Choose $k\ge\max\{d_4,4\}$ and set $\Gam':=\Gam(\DynA_{k+3},\{1,2,3,4\})$. Its parameter sequence is $(1,1,1,1,k)$. Since $k\ge4$, the subset $\{1,2,3,4\}$ is contained in the left half of $\DynA_{k+3}$, and Proposition~\ref{prop:oneside} shows that $\Gam'$ is tame. By Theorem~\ref{thm:uniformA}, $\Gam(\DynA_m,\Omega) \cong
\Gam'/\langle x_4^{d_4}\rangle$. Hence, $\Gam(\DynA_m,\Omega)$ is either representation-finite or tame. By Proposition~\ref{prop:coarsen}, the idempotent truncation of $\Gam(\DynA_m,\Omega)$ corresponding to $\{j_2,j_4\}$ has parameter sequence $(2,2,d_4)$, and it is representation-infinite by Proposition~\ref{thm:two-point}. Therefore, $\Gam(\DynA_m,\Omega)$ is tame.
\end{proof}

\begin{proposition}\label{prop:11H11tame}
Suppose $d_0=d_1=d_3=d_4=1$ and $d_2\ge2$. Then, $\Gam(\DynA_m,\Omega)$ is tame.
\end{proposition}
\begin{proof}
Since $d_0=d_1=d_3=d_4=1$, the loops $x_1,x_4$ vanish and $x_2=b_1a_1$, $x_3=a_3b_3$. Moreover,
$x_2^2=x_3^2=0$, so $d_2\ge2$ gives
$a_2b_2=b_2a_2=0$. Then, the remaining relations are $a_1b_1=a_2b_2=b_2a_2=b_3a_3=0$, $b_1a_1a_2=a_2a_3b_3$, $b_2b_1a_1=a_3b_3b_2$. We deduce that $\Gam(\DynA_m,\Omega)$ is a special biserial algebra. On the other hand, the idempotent truncation corresponding to $\{j_2,j_3\}$ has parameter sequence $(2,d_2,2)$ by Proposition~\ref{prop:coarsen}, and it is representation-infinite by Proposition~\ref{thm:two-point}. Hence, $\Gam(\DynA_m,\Omega)$ is tame.
\end{proof}

\begin{proposition}\label{prop:111H1tame}
Suppose $d_0=d_2=d_3=d_4=1$ and $d_1\ge2$. Then,
$\Gam(\DynA_m,\Omega)$ is tame.
\end{proposition}
\begin{proof}
We first show that $\Gam(\DynA_m,\Omega)$ is
representation-infinite. By Proposition~\ref{prop:coarsen}, deleting the vertex $j_3$ gives the three-vertex idempotent
truncation with parameter sequence $(1,d_1,2,1)$.
Since $d_1\ge2$, Proposition~\ref{prop:r3table} shows that this truncation is representation-infinite.

We have $x_1=x_4=0$, $x_2=a_2b_2$, and $x_3=b_2a_2=a_3b_3$. Hence, the defining relations reduce to $a_1b_1=0$, $b_1a_1=(a_2b_2)^{d_1}$, $b_2a_2=a_3b_3$, $b_3a_3=0$, $a_1a_2b_2=0$, $a_2b_2b_1=0$. Since $(b_2a_2)^2=(a_3b_3)^2=a_3(b_3a_3)b_3=0$, we have $(a_2b_2)^3=0$. Thus, if $d_1\ge3$, $\Gam(\DynA_m,\Omega)$ is precisely the tame algebra established in the final appendix. For $d_1=2$, consider the one-parameter family obtained by replacing the relation $b_1a_1=(a_2b_2)^2$ with $b_1a_1=\lambda(a_2b_2)^2$, $\lambda\in\Bbbk$, while keeping all the other relations unchanged. Then, $\Gam(\DynA_m,\Omega)$ degenerates to the algebra occurring for $d_1\ge3$. Hence, $\Gam(\DynA_m,\Omega)$ is tame.
\end{proof}

\begin{proposition}\label{prop:r4A}
Suppose $\Omega=\{j_1<j_2<j_3<j_4\}$ and $d_0,\ldots,d_4$ as above. Then,
\begin{enumerate}
\item if $d_0=d_1=d_2=d_3=d_4=1$, then
$\Gam(\DynA_m,\Omega)\cong\Pi_{\DynA_4}$ is
representation-finite;

\item if exactly one of $d_0,\ldots,d_4$ is at least $2$, then $\Gam(\DynA_m,\Omega)$ is tame.

\item if at least two of $d_0,\ldots,d_4$ are at least $2$, then $\Gam(\DynA_m,\Omega)$ is wild.
\end{enumerate}
\end{proposition}
\begin{proof}
If $d_0=\cdots=d_4=1$, then $m=4$ and
$\Omega=\{1,2,3,4\}$. Hence, $\Gam(\DynA_m,\Omega)=\Pi_{\DynA_4}$ is representation-finite by \cite[Proposition 6.3]{DlabRingel92}. If exactly one of $d_0,\ldots,d_4$ is at least $2$, then $\Gam(\DynA_m,\Omega)$ is tame by
Propositions~\ref{prop:H1111tame}, \ref{prop:11H11tame}, and~\ref{prop:111H1tame}. Applying the reduction strategy mentioned at the beginning of this subsection and Proposition~\ref{prop:1H1H1wild}, all remaining cases are wild. 
\end{proof}

\subsection{Five or more vertices}
\label{sec:A-general}
We now complete the classification for $|\Omega|\ge5$. The only non-wild case occurs when $|\Omega|=5$ and all parameters are equal
to $1$.

\begin{proposition}\label{thm:r5A}
Suppose $|\Omega|=5$. Then, $\Gam(\DynA_m,\Omega)$ is tame if $d_0=d_1=\cdots=d_5=1$, and wild otherwise.
\end{proposition}
\begin{proof}
Suppose that $\Gam(\DynA_m,\Omega)$ is not wild. Consider the idempotent truncation corresponding to $\{j_1,j_2,j_4\}$. By Proposition~\ref{prop:coarsen}, its parameter sequence is $(d_0,\ d_1,\ d_2+d_3,\ d_4+d_5)$. The last two parameters are at least $2$. Hence, Proposition~\ref{prop:r3table} forces $d_0=d_1=1$, since otherwise this three-vertex truncation would have at least three parameters at least $2$ and would be wild. Applying the diagram automorphism of $\DynA_m$ gives $d_4=d_5=1$. Now consider the idempotent truncation corresponding to $\{j_2,j_3,j_5\}$. Its parameter sequence is $(2,\ d_2,\ d_3+1,\ 1)$. The first and third parameters are at least $2$. If $d_2\ge2$, then this truncation has at least three parameters at least $2$, and is wild by Proposition~\ref{prop:r3table}. Thus, $d_2=1$. Applying the same argument after reversing the parameter sequence gives $d_3=1$. Therefore, $d_0=d_1=\cdots=d_5=1$. 

Conversely, if all six parameters are equal to $1$, then $m=5$ and $\Omega=\{1,2,3,4,5\}$. Hence, $\Gam(\DynA_m,\Omega)=\Pi_{\DynA_5}$, which is tame by \cite[Proposition~6.3]{DlabRingel92}.
\end{proof}

\begin{proposition}\label{thm:rge6wild}
If $|\Omega|\ge6$, then $\Gam(\DynA_m,\Omega)$ is wild.
\end{proposition}
\begin{proof}
Write $\Omega=\{j_1<\cdots<j_r\}$ with $r\ge6$, and consider the idempotent truncation corresponding to $\{j_2,j_4,j_5\}$. By
Proposition~\ref{prop:coarsen}, its parameter sequence is $(d_0+d_1,\ d_2+d_3,\ d_4,\ d_5+\cdots+d_r)$. Since $d_0+d_1\ge2$, $d_2+d_3\ge2$, $d_5+\cdots+d_r\ge2$, this three-vertex truncation has at least three parameters at least $2$, and is therefore wild by Proposition~\ref{prop:r3table}. Consequently, $\Gam(\DynA_m,\Omega)$ is wild.
\end{proof}

\subsection{Summary}
Combining the preceding results gives the complete classification of the representation types of contracted preprojective algebras of
type~$\DynA$. We give a summary without proofs as follows.

\begin{theorem}\label{thm:markedA-main}
Let $\Omega=\{j_1<\cdots<j_r\}\subseteq \{1,2,\cdots, m\}$ be a non-empty set with associated parameter sequence $(d_0,\ldots,d_r)$. 
\begin{enumerate}
\item If $r=1$, then $\Gam(\DynA_m,\Omega)$ is
representation-finite.

\item If $r=2$, then $\Gam(\DynA_m,\Omega)$ is
representation-finite if $\min\{d_0,d_1,d_2\}=1$, tame if $\min\{d_0,d_1,d_2\}=2$, and wild if $\min\{d_0,d_1,d_2\}\ge3$.

\item If $r=3$, then $\Gam(\DynA_m,\Omega)$ is wild if at least three of $d_0,d_1,d_2,d_3$ are at least $2$; it is tame if exactly two are at least $2$ and their indices are one of $\{0,1\},
\{1,2\},\{2,3\},\{0,3\}$, and it is representation-finite in all remaining cases.

\item If $r=4$, then $\Gam(\DynA_m,\Omega)$ is
representation-finite if $d_0=d_1=d_2=d_3=d_4=1$, tame if exactly one of
$d_0,\ldots,d_4$ is at least $2$, and wild if at least two of them are at least $2$.

\item If $r=5$, then $\Gam(\DynA_m,\Omega)$ is tame if $d_0=d_1=\cdots=d_5=1$, and is wild otherwise.

\item If $r\ge6$, then $\Gam(\DynA_m,\Omega)$ is wild.
\end{enumerate}
\end{theorem}

\section{Contracted preprojective algebras of types~$\DynD$ and~$\DynE$}\label{sec:DE}
In this section, we determine the representation type of contracted preprojective algebras for the remaining simply-laced Dynkin diagrams. We shall repeatedly use quotients, idempotent truncation, radical-square-zero quotients, and the standard representation theory of Nakayama algebras and special biserial algebras.

\subsection{Type $\DynD$}
For any $m\ge4$, we label $\DynD_m$ by
\[
\DynD_m:\; \vcenter{\xymatrix@C=2.5em@R=0.5em{
&&&& &m-1\\
1\ar@{-}[r]&2\ar@{-}[r]&\cdots\ar@{-}[r]&m-3\ar@{-}[r]&m-2\ar@{-}[ur]\ar@{-}[dr]&\\
&&&&&m
}}
\]
We use the diagram automorphism of $\DynD_m$ that interchanges $m-1$ and $m$ without further comment. We record the bound quiver presentation of $\Pi_{\DynD_m}$ with the quiver
\[
\vcenter{\xymatrix@C=4em@R=2.2em{
&&&&m-1
\ar@/^/[dl]^{b_{m-2}}
\\
1
\ar@/^/[r]^{a_1}
&
2
\ar@/^/[l]^{b_1}
\ar@/^/[r]^{a_2}
&
\cdots
\ar@/^/[l]^{b_2}
\ar@/^/[r]^{a_{m-3}}
&
m-2
\ar@/^/[l]^{b_{m-3}}
\ar@/^/[ur]^{a_{m-2}}
\ar@/^/[dr]^{a_{m-1}}
&
\\
&&&&m
\ar@/^/[ul]^{b_{m-1}}
}},
\]
bounded by the ideal which is generated by $a_1b_1=0$, $b_{m-2}a_{m-2}=0$, $b_{m-1}a_{m-1}=0$, $b_{m-3}a_{m-3}=a_{m-2}b_{m-2}+a_{m-1}b_{m-1}$, and $a_ib_i=b_{i-1}a_{i-1}$ for all $2\le i\le m-3$. These relations will be used below to reduce paths in the idempotent truncation $e_\Omega\Pi_{\DynD_m}e_\Omega$.

For every nonempty subset $\Omega\subseteq\{1,2,\ldots,m\}$, we write $\Gam(\DynD_m,\Omega):=e_\Omega\Pi_{\DynD_m}e_\Omega$. We first consider the one-vertex cases, as follows.

\begin{proposition}\label{prop:D-onepoint-internal}
Let $\Omega=\{i\}\subseteq\{1,2,\ldots,m\}$. Then, $\Gam(\DynD_m,\Omega)$ is representation-finite if $i=1,m-1,m$, tame if $i=2,m-2$, and wild otherwise.
\end{proposition}
\begin{proof}
It is easy to write down the radical layers of the indecomposable projective $\Pi_{\DynD_m}$-modules. Then, one may verify that
\begin{itemize}
\item $\Gam(\DynD_m,\{1\})\simeq \Bbbk[x]/(x^2)$, which is a representation-finite local algebra;

\item $\Gam(\DynD_m,\{2\})\simeq \Bbbk\langle x,y\rangle/(x^2,y^2,xy+yx)$, which is a representation-infinite special biserial algebra, and hence tame;

\item for any $3\le i\le m-3$, the algebra $\Gam(\DynD_m,\{i\})$ has a quotient $\Bbbk\langle x,y\rangle/(x^3,y^2,x^2y,xy+yx)$, which is the wild algebra in Proposition~\ref{prop:local-benchmark-alg}\textup{(7)} with $\lambda=-1$;

\item $\Gam(\DynD_m,\{m-2\})\simeq \Bbbk\langle x,y\rangle/(x^2,y^2,(x+y)^{m-2})$, which is a representation-infinite special biserial algebra, and hence tame;

\item $\Gam(\DynD_m,\{m-1\})\simeq \Bbbk[x]/(x^{\lfloor m/2\rfloor})$, which is a representation-finite local algebra; the case $i=m$ is obtained by the diagram automorphism of $\DynD_m$.
\end{itemize}
This completes the proof.
\end{proof}

The following observation eliminates most of the remaining cases.
\begin{lemma}\label{lem:D-trunk-pairs}
Let $m\ge4$ and $1\le i<j\le m-2$. If $\{i,j\}\ne\{1,2\}$, then $\Gam(\DynD_m,\{i,j\})$ is wild.
\end{lemma}
\begin{proof}
We first suppose that $j-i\ge2$. Then, the path $a_ia_{i+1}\cdots a_{j-1}:i\longrightarrow j$ and its reverse $b_{j-1}\cdots b_{i+1}b_i:j\longrightarrow i$ give opposite arrows in the Gabriel quiver of $\Gam(\DynD_m,\{i,j\})$. Moreover, the removed vertices between $i$ and $j$ give a loop at $j$, while the part of the diagram to the right of $j$ gives another loop at $j$. If $i>1$, the removed part to the left of $i$ also gives a loop at $i$. Hence, if $i=1$ (resp. $i>1$), the separated quiver of $\Gam(\DynD_m,\{i,j\})/\rad^2\Gam(\DynD_m,\{i,j\})$ contains the left (resp. right) connected subquiver
\[
\vcenter{\xymatrix@C=3em@R=2em{
i^+ \ar[dr]
&
j^+ \ar[dl] \ar@/^/[d] \ar@/_/[d]
\\
i^-
&
j^-
}},\qquad\qquad
\vcenter{\xymatrix@C=3em@R=2em{
i^+ \ar[d] \ar[dr]
&
j^+ \ar[dl] \ar@/^/[d] \ar@/_/[d]
\\
i^-
&
j^-
}}.
\]
Since the underlying graph is neither Dynkin nor Euclidean, $\Gam(\DynD_m,\{i,j\})$ is wild.

Suppose that $j=i+1$. Since $\{i,j\}\ne\{1,2\}$, we have $i\ge2$. If $j\le m-3$, then $j\ge3$, and Proposition~\ref{prop:D-onepoint-internal} shows that the idempotent truncation $\Gam(\DynD_m,\{j\})$ is wild. If $j=m-2$ and $m\ge6$, then $i=m-3\ge3$, so $\Gam(\DynD_m,\{i\})$ is wild by Proposition~\ref{prop:D-onepoint-internal}. In both cases, $\Gam(\DynD_m,\{i,j\})$ is wild. The only remaining case is $m=5$ and $\{i,j\}=\{2,3\}$. In this case, $\Gam(\DynD_5,\{2,3\})$ has the bound quiver
\[
\vcenter{\xymatrix@C=4em{
2
\ar@/^/[r]^{a_2}
&
3
\ar@/^/[l]^{b_2}
\ar@(ur,dr)[]^{a_3b_3}
}}
\]
with relations $(a_2b_2)^2=0$, $(a_3b_3)^2=0$, and $(b_2a_2-a_3b_3)^2=0$. We consider the quotient $\Gam(\DynD_5,\{2,3\})/
\left\langle a_2a_3b_3,\; a_2b_2a_2,\; a_3b_3b_2a_2 \right\rangle$, which is the minimal wild two-point algebra $W_{20}$ in \cite[Table~W]{Han02}. Hence, $\Gam(\DynD_5,\{2,3\})$ is wild.
\end{proof}

We classify the two-vertex cases in this following. By Proposition~\ref{prop:D-onepoint-internal} and Lemma~\ref{lem:D-trunk-pairs}, only a small number of cases require further consideration.

\begin{proposition}\label{prop:D-rank2}
Suppose $|\Omega|=2$. Then, $\Gam(\DynD_m,\Omega)$ is representation-finite if $\Omega=\{1,m-1\}$, $\{1,m\}$, $\{m-1,m\}$, tame if $\Omega=\{1,2\}$, $\{2,m-1\}$, $\{2,m\}$, $\{m-2,m-1\}$, $\{m-2,m\}$, and wild otherwise.
\end{proposition}
\begin{proof}
(1) Set $t:=m-1$, $\mu:= a_1a_2\cdots a_{m-2}$, $\nu:=b_{m-2}\cdots b_2b_1$, $\alpha:=b_{m-2}a_{m-1}b_{m-1}a_{m-2}$. The Gabriel quiver of $\Gam(\DynD_m,\{1,m-1\})$ is
\[
\vcenter{\xymatrix@C=5em{
1
\ar@/^/[r]^{\mu}
&
t
\ar@/^/[l]^{\nu}
\ar@(ur,dr)[]^{\alpha}
}} \;.
\]
Reducing paths with the defining relations of $\Pi_{\DynD_m}$ gives $\mu\alpha=0$, $\alpha\nu=0$ and $\alpha^{\lfloor m/2\rfloor}=0$. Moreover, $\nu\mu=0$ if $m$ is odd, $\nu\mu=\alpha^{\lfloor m/2\rfloor-1}$ if $m$ is even. Hence, $\Gam(\DynD_m,\{1,m-1\})$ is special biserial. If $m=4$, eliminating $\alpha=\nu\mu$ gives the cyclic Nakayama algebra with all length 3 paths being zero. If $m>4$, any cyclic reduced walk using vertex $1$ contains $\nu\mu$ or its inverse after a cyclic rotation, and this path is zero for odd $m$ and is a term of a socle relation for even $m$, so it cannot occur in a band; on the other hand, a cyclic reduced walk supported only at $t$ is a power of $\alpha$ or its inverse and cannot be repeated indefinitely. We conclude that there is no band if $m\ge 4$, and hence, $\Gam(\DynD_m,\{1,m-1\})$ is representation-finite. By the diagram automorphism, $\Gam(\DynD_m,\{1,m\})$ is also representation-finite.

We then consider $\Gam(\DynD_m,\{m-1,m\})$. This algebra is given by the quiver
\[
\vcenter{\xymatrix@C=5em{
m-1
\ar@/^/[r]^{b_{m-2}a_{m-1}}
&
m
\ar@/^/[l]^{b_{m-1}a_{m-2}}
}}
\]
with all paths of length $m-1$ equal to zero. Hence, $\Gam(\DynD_m,\{m-1,m\})$ is a truncated cyclic Nakayama algebra and is representation-finite.

(2) We consider the three tame families, up to the diagram automorphism of $\DynD_m$. Then, $\Gam(\DynD_m,\{1,2\})$ is presented by 
\[
\vcenter{\xymatrix@C=5em{
1
\ar@/^/[r]^{a_1}
&
2
\ar@/^/[l]^{b_1}
\ar@(ur,dr)[]^{\alpha}
}},\quad \alpha:=a_2a_3\cdots a_{m-2}b_{m-2}\cdots b_3b_2,
\]
with relations $a_1b_1=0$, $\alpha^2=0$ and $b_1a_1\alpha+\alpha b_1a_1=0$. Thus, $\Gam(\DynD_m,\{1,2\})$ is special biserial. Since the idempotent truncation at vertex $2$ is representation-infinite, $\Gam(\DynD_m,\{1,2\})$ is tame.

We find that $\Gam(\DynD_m,\{m-2,m-1\})$ is presented by 
\[
\vcenter{\xymatrix@C=5em{
t
\ar@/^/[r]^{a_{m-2}}
\ar@(ul,dl)[]_{a_{m-1}b_{m-1}}
&
m-1
\ar@/^/[l]^{b_{m-2}}
}},\quad  t:=m-2 
\]
with relations $b_{m-2}a_{m-2}=0$, $(a_{m-1}b_{m-1})^2=0$ and $(a_{m-2}b_{m-2}+a_{m-1}b_{m-1})^{m-2}=0$. Hence, this algebra is again special biserial. Its idempotent truncation at vertex $m-2$ is representation-infinite, so it is tame.

It remains to consider $\Omega=\{2,m-1\}$. Set $t=m-1$, $\mu:=a_2a_3\cdots a_{m-2}$, $\nu:=b_{m-2}\cdots b_3b_2$, $\alpha:=b_1a_1$, $\beta:=b_{m-2}a_{m-1}b_{m-1}a_{m-2}$. Then, $\Gam(\DynD_m,\{2,m-1\})$ is presented by 
\[
\vcenter{\xymatrix@C=5em{
2
\ar@/^/[r]^{\mu}
\ar@(ul,dl)[]_{\alpha}
&
t
\ar@/^/[l]^{\nu}
\ar@(ur,dr)[]^{\beta}
}}.
\]
with the relations $\nu\mu=0$, $\beta^{\lfloor(m-2)/2\rfloor}=\nu\alpha\mu$ if $m$ is even, the relations $\nu\mu=\beta^{\lfloor(m-2)/2\rfloor}$, $\nu\alpha\mu=0$ if $m$ is odd, and the common relations $\alpha^2=0$, $\mu\beta=0$, $\beta\nu=0$,  $\alpha\mu\nu+\mu\nu\alpha=0$. We find that $\Gam(\DynD_m,\{2,m-1\})$ is a special biserial algebra if $m$ is even, and an algebra degenerate to a special biserial algebra if $m$ is odd, by replacing $\nu\mu = \beta^{\lfloor(m-2)/2\rfloor}$ with $\nu\mu=0$. In both cases, the idempotent truncation at vertex $2$ is representation-infinite, and thus $\Gam(\DynD_m,\{2,m-1\})$ is tame. By the diagram automorphism of $\DynD_m$, $\Gam(\DynD_m,\{2,m\})$ is also tame.

(3) Every remaining pair either contains a vertex $i$ with $3\le i\le m-3$, giving a wild idempotent truncation by Proposition~\ref{prop:D-onepoint-internal}, or consists of two long-arm vertices other than $\{1,2\}$, giving wildness by Lemma~\ref{lem:D-trunk-pairs}. This proves the proposition.
\end{proof}

We can now classify all three-vertex cases.

\begin{proposition}\label{prop:D-rank3}
Let $|\Omega|=3$. Then, $\Gam(\DynD_m,\Omega)$ is tame if $\Omega=\{1,2,m-1\}$, $\{1,2,m\}$, $
\{1,m-1,m\}$, $\{m-2,m-1,m\}$, and wild otherwise.
\end{proposition}
\begin{proof}
Suppose first that $\Omega=\{i<j<k\}\subseteq\{1,\ldots,m-2\}$. Since $\{i,k\}\ne\{1,2\}$, Lemma~\ref{lem:D-trunk-pairs} gives a wild idempotent truncation. Hence, $\Gam(\DynD_m,\Omega)$ is wild. Suppose next that $\Omega=\{i,j,m-1\}$ with $i<j\le m-2$. If $\{i,j\}\ne\{1,2\}$, Lemma~\ref{lem:D-trunk-pairs} gives wildness again. The remaining case is $\Omega=\{1,2,m-1\}$. In this case, set $t=m-1$, the quiver $\Gam(\DynD_m,\{1,2,m-1\})$ is presented by 
\[
\vcenter{\xymatrix@C=5em{
1
\ar@/^/[r]^{a_1}
&
2
\ar@/^/[l]^{b_1}
\ar@/^/[r]^{a_2a_3\cdots a_{m-2}}
&
t
\ar@/^/[l]^{b_{m-2}\cdots b_3b_2}
\ar@(ur,dr)[]^{b_{m-2}a_{m-1}b_{m-1}a_{m-2}}
}}.
\]
Similar to the proof of $\Gam(\DynD_m,\{2,m-1\})$, one may verify that $\Gam(\DynD_m,\{1,2,m-1\})$ is tame. Moreover, the case $\{1,2,m\}$ gives a tame algebra by automorphism.

Suppose $\Omega=\{j,m-1,m\}$. If $j=m-2$, then $\Gam(\DynD_m,\Omega)$ has the bound quiver
\[
\vcenter{\xymatrix@C=4em@R=2em{
&m-1\ar@/^/[dl]^{b_{m-2}}\\
m-2\ar@/^/[ur]^{a_{m-2}}\ar@/^/[dr]^{a_{m-1}}&\\
&m\ar@/^/[ul]^{b_{m-1}}
}}
\]
with relations $b_{m-2}a_{m-2}=0$, $b_{m-1}a_{m-1}=0$, $(a_{m-2}b_{m-2}+a_{m-1}b_{m-1})^{m-2}=0$. This is obviously a representation-infinite special biserial algebra, and hence tame.

If $1<j<m-2$, the Gabriel quiver of $\Gam(\DynD_m,\{j,m-1,m\})$ is
\[
\vcenter{\xymatrix@C=5em@R=2.5em{
& m-1
\ar@/^/[dl]
\ar@/^/[dr]
&
\\
j
\ar@/^/[ur]
\ar@/^/[rr]
\ar@(dl,ul)[]
&
&
m
\ar@/^/[ll]
\ar@/^/[ul]
}}
\]
such that the separated quiver of its radical-square-zero quotient is
\[
\vcenter{\xymatrix@C=3.5em@R=1.5em{
&j^+\ar[dl]\ar[dr]\ar[ddd]&\\
(m-1)^-&&m^-\\
m^+\ar[u]\ar[dr]&&(m-1)^+\ar[u]\ar[dl]\\
&j^-&
}}.
\]
Its underlying graph is neither Dynkin nor Euclidean, showing that $\Gam(\DynD_m,\{j,m-1,m\})$ is wild. 

Finally, let $j=1$. The Gabriel quiver of $\Gam(\DynD_m,\{1,m-1,m\})$ is
\[
\vcenter{\xymatrix@C=5em@R=2.5em{
& m-1
\ar@/^/[dl]
\ar@/^/[dr]
&
\\
1
\ar@/^/[ur]
\ar@/^/[rr]
&
&
m
\ar@/^/[ll]
\ar@/^/[ul]
}}
\]
and direct reduction with the defining relations of $\Pi_{\DynD_m}$ shows that every arrow has at most one nonzero successor and at most one nonzero predecessor. Hence, $\Gam(\DynD_m,\{1,m-1,m\})$ is special biserial. Moreover, the separated quiver of its radical-square-zero quotient is
\[
\vcenter{\xymatrix@C=3.5em@R=1.5em{
&1^+\ar[dl]\ar[dr]&\\
(m-1)^-&&m^-\\
m^+\ar[u]\ar[dr]&&(m-1)^+\ar[u]\ar[dl]\\
&1^-&
}},
\]
whose underlying graph is $\widetilde{\DynA}_5$. Thus, $\Gam(\DynD_m,\{1,m-1,m\})$ is representation-infinite tame.
\end{proof}

\begin{proposition}\label{prop:markedD-rank}
Let $\varnothing\ne\Omega\subseteq\{1,2,\ldots,m\}$. If $|\Omega|\ge4$, then $\Gam(\DynD_m,\Omega)$ is tame if $m=4$ and $\Omega=\{1,2,3,4\}$, and wild otherwise.
\end{proposition}
\begin{proof}
It is enough to consider $|\Omega|=4$. Suppose first that $|\Omega\cap\{m-1,m\}|\le1$. Then, $|\Omega\cap\{1,2,\ldots,m-2\}|\ge3$. Among these vertices there is a pair $\{i,j\}\ne\{1,2\}$ with $1\le i<j\le m-2$. Hence, Lemma~\ref{lem:D-trunk-pairs} gives a wild idempotent truncation of $\Gam(\DynD_m,\Omega)$, and therefore $\Gam(\DynD_m,\Omega)$ is wild.

Suppose that $\{m-1,m\}\subseteq\Omega$. Then, $\Omega=\{i,j,m-1,m\}$ for some $1\le i<j\le m-2$. If $\Gam(\DynD_m,\Omega)$ were non-wild, then the two idempotent truncations $\Gam(\DynD_m,\{i,j,m-1\})$ and $\Gam(\DynD_m,\{j,m-1,m\})$ would both be non-wild. By Proposition~\ref{prop:D-rank3}, the first condition forces $\{i,j\}=\{1,2\}$, while the second forces $j=m-2$. Thus, $j=2=m-2$, so $m=4$ and $\Omega=\{1,2,3,4\}$. In this case, $\Gam(\DynD_4,\Omega)=\Pi_{\DynD_4}$, which is tame by \cite[Proposition~6.3]{DlabRingel92}.

Finally, suppose that $|\Omega|\ge5$. Then $m\ge5$, and every four-element subset $\Omega'\subseteq\Omega$ gives a wild idempotent truncation $\Gam(\DynD_m,\Omega')$ by the preceding argument. Hence, $\Gam(\DynD_m,\Omega)$ is wild by the idempotent-truncation criterion.
\end{proof}

We now obtain a complete classification of the representation types of contracted preprojective algebras of type $\DynD$, as summarized below.
\begin{proposition}\label{thm:markedD-main}
Let $\varnothing\ne\Omega\subseteq\{1,2,\ldots,m\}$. Then, $\Gam(\DynD_m,\Omega)$ is representation-finite if
\[
\Omega=\{1\},\ \{m-1\},\ \{m\},\ \{1,m-1\},\ \{1,m\},\ \{m-1,m\},
\]
tame if
\[
\begin{aligned}
\Omega={}&\{2\},\ \{m-2\},\ \{1,2\},\ \{2,m-1\},\ \{2,m\},\\
&\{m-2,m-1\},\ \{m-2,m\},\ \{1,2,m-1\},\ \{1,2,m\},\\
&\{1,m-1,m\},\ \{m-2,m-1,m\},
\end{aligned}
\]
or if $m=4$ and $\Omega=\{1,2,3,4\}$, and wild otherwise.
\end{proposition}
\begin{proof}
The one-vertex case is Proposition~\ref{prop:D-onepoint-internal}. The two- and three-vertex cases are Propositions~\ref{prop:D-rank2} and~\ref{prop:D-rank3}. The remaining cases are covered by Proposition~\ref{prop:markedD-rank}.
\end{proof}

\subsection{Type~$\DynE$}\label{sec:Eexceptional}
We use the following uniform labeling:
\[
\DynE_m:\; \vcenter{\xymatrix@C=2.5em@R=1.5em{
&&2\\
1\ar@{-}[r]&3\ar@{-}[r]&4\ar@{-}[u]\ar@{-}[r]&5\ar@{-}[r]&\cdots\ar@{-}[r]&m
}},
\qquad m=6,7,8.
\]
We record the bound quiver presentation of $\Pi_{\DynE_m}$ with the quiver
\[
\vcenter{\xymatrix@C=4em@R=2.5em{
&&2
\ar@/^/[d]^{b_2}
\\
1
\ar@/^/[r]^{a_1}
&
3
\ar@/^/[l]^{b_1}
\ar@/^/[r]^{a_3}
&
4
\ar@/^/[l]^{b_3}
\ar@/^/[u]^{a_2}
\ar@/^/[r]^{a_4}
&
5
\ar@/^/[l]^{b_4}
\ar@/^/[r]^{a_5}
&
\cdots
\ar@/^/[l]^{b_5}
\ar@/^/[r]^{a_{m-1}}
&
m
\ar@/^/[l]^{b_{m-1}}
}},
\]
bounded by the relations $a_1b_1=0$, $b_2a_2=0$, $a_3b_3=b_1a_1$, $b_3a_3=a_2b_2+a_4b_4$, $b_{m-1}a_{m-1}=0$, $a_ib_i=b_{i-1}a_{i-1}$ for all $5\le i\le m-1$. If $\Sigma$ is a connected full Dynkin subdiagram of $\DynE_m$, then killing the vertices outside $\Sigma$ gives a surjection $\Gam(\DynE_m,\Omega)\twoheadrightarrow\Gam(\Sigma,\Omega\cap\Sigma)$, whenever $\Omega\cap\Sigma\ne\varnothing$. Thus, a wild full subdiagram part forces the original algebra to be wild.

We first consider type~$\DynE_6$. One note that the Dynkin diagram $\DynE_6$ admits an automorphism exchanging $1\leftrightarrow 6$ and $3\leftrightarrow 5$, while fixing
the vertices $2$ and $4$.

\begin{proposition}\label{prop:E6-main}
Let $\varnothing\ne\Omega\subseteq\{1,2,3,4,5,6\}$. Then, $\Gam(\DynE_6,\Omega)$ is representation-finite if $\Omega=\{1\},\{6\}$, tame if $\Omega=\{2\},\{3\},\{5\},\{1,2\},\{2,6\},\{1,6\},\{1,3\},\{5,6\}$, and wild otherwise.
\end{proposition}
\begin{proof}
We first consider the one-vertex cases. The idempotent truncations with $\Omega=\{1\}$ and $\Omega=\{6\}$ are both isomorphic to $\Bbbk[x]/(x^2)$. If $\Omega=\{2\}$, taking $x=b_2a_4b_4a_2$ and $y=b_2a_4a_5b_5b_4a_2$ give $\Gam(\DynE_6,\Omega)\cong \Bbbk\langle x,y\rangle/(x^2,y^2,xy+yx)$, which is tame by Proposition \ref{prop:local-benchmark-alg}. If $\Omega=\{3\}$, we take $x=a_3b_3$ and $y=a_3a_4b_4b_3$ to obtain $\Gam(\DynE_6,\Omega)\cong \Bbbk\langle x,y\rangle/(x^2,y^2+xyx,yxy)$; this algebra degenerates to a special biserial algebra by replacing $y^2+xyx$ by $y^2+\lambda xyx$ with $\lambda\in \Bbbk$, and hence, it is representation-infinite tame. 
If $\Omega=\{4\}$, we take $x=a_2b_2$ and $y=a_4b_4$ to obtain $\Gam(\DynE_6,\Omega)\cong \Bbbk\langle x,y\rangle/(x^2,y^3,(x+y)^3)$; imposing $yx=xy^2=0$ gives the wild local algebra in Proposition~\ref{prop:local-benchmark-alg}\textup{(4)}. We omit the case $\Omega=\{5\}$ since it is same with $\Omega=\{3\}$ up to the diagram automorphism of $\DynE_6$.

In the case of $\Omega=\{1,3\}$, $\Gam(\DynE_6,\{1,3\})$ has the bound quiver
\[
\vcenter{\xymatrix@C=5em{
1\ar@/^/[r]^{a_1}
&
3\ar@/^/[l]^{b_1}\ar@(ur,dr)[]^{z}
}}, \quad z:=a_3a_4b_4b_3
\]
with relations $a_1b_1=0$, $z^2+b_1a_1zb_1a_1=0$, $zb_1a_1z=0$. This presentation also appears, up to notation, in
\cite[Example~4.10]{ChanIyamaMarczinzik}. We find that this degenerates to a special biserial algebra by replacing the second relation by $z^2+\lambda b_1a_1zb_1a_1=0$ with $\lambda\in \Bbbk$. Since the idemponent truncation of $\Gam(\DynE_6,\{1,3\})$ at vertex $3$ has two loops, $\Gam(\DynE_6,\{1,3\})$ is representation-infinite and hence, tame. Moreover, the case $\Omega=\{5,6\}$ is tame again by the diagram automorphism of $\DynE_6$. 

Suppose $\Omega=\{1,2\}$. Then, $\Gam(\DynE_6,\{1,2\})$ has the bound quiver
\[
\vcenter{\xymatrix@C=5em{
1\ar@/^/[r]^{u}
&
2\ar@/^/[l]^{v}\ar@(ur,dr)[]^{z}
}},\quad u:=a_1a_3a_2,\; v:=b_2b_3b_1,\; z:=b_2a_4b_4a_2,
\]
with relations $z^2=uvu=vuv=uzv=0$, $zvu+vuz=0$. This algebra is obviously representation-infinite, and it is a quotient of the tame algebra $T_9$ in \cite[Table T]{Han02}. Hence, $\Gam(\DynE_6,\{1,2\})$ is tame. By the automorphism, so is
$\Gam(\DynE_6,\{2,6\})$. 

Suppose $\Omega=\{1,6\}$. Set $u:=a_1a_3a_4a_5$, $v:=b_5b_4b_3b_1$, $x:=a_1a_3a_4b_4b_3b_1$, $y:=b_5b_4b_3a_3a_4a_5$. Then, $\Gam(\DynE_6,\{1,6\})$ has the bound quiver
\[
\vcenter{\xymatrix@C=5em{
1\ar@/^/[r]^{u}\ar@(ul,dl)[]_{x}
&
6\ar@/^/[l]^{v}\ar@(ur,dr)[]^{y}
}},
\]
with relations $x^2=y^2=uv=vu=0$, $xu=uy$, $yv=vx$. This is a special biserial algebra, and its radical-square-zero
quotient has separated quiver of type $\widetilde{\DynA}_3$. It is therefore representation-infinite tame.

In the case of $(p,q)=(1,5)$ or $(6,3)$, 
$\Gam(\DynE_6,\{p,q\})$ admits a radical-square-zero quotient with the bound quiver
\[
\vcenter{\xymatrix@C=5em{
p\ar@/^/[r]^{ }
&
q\ar@/^/[l]^{ }\ar@(ul,ur)[]^{ }\ar@(dl,dr)[]_{ }
}}
\]
whose separated quiver is neither Dynkin nor Euclidean. Hence, $\Gam(\DynE_6,\{p,q\})$ is wild.

Suppose $\Omega=\{2,3\}$. Set $u=b_2b_3$, $v=a_3a_2$, and $z=a_3b_3$. By imposing the
additional relations $uz=uvu=zvu=0$, we obtain a quotient of $\Gam(\DynE_6,\{2,3\})$  whose bound quiver is
\[
\vcenter{\xymatrix@C=5em{
2\ar@/^/[r]^{u}
&
3\ar@/^/[l]^{v}\ar@(ur,dr)[]^{z}
}}
\]
with relations $z^2=uz=uvu=zvu=0$. This is exactly the minimal wild algebra $W_{20}$ in
\cite[Table W]{Han02}. Thus, $\Gam(\DynE_6,\{2,3\})$ is wild, and so is $\Gam(\DynE_6,\{2,5\})$ by automorphism.

Suppose $\Omega=\{3,5\}$. Set $x=a_3b_3$, $u=a_3a_4$, $v=b_4b_3$, and $y=a_5b_5$.
Imposing $xu=uy$ and $vx=yv=uvu=vuv=0$ produces a quotient of $\Gam(\DynE_6,\{3,5\})$ with bound quiver
\[
\vcenter{\xymatrix@C=5em{
3\ar@/^/[r]^{u}\ar@(ul,dl)[]_{x}
&
5\ar@/^/[l]^{v}\ar@(ur,dr)[]^{y}
}}
\]
and relations $x^2=y^2=0$, $xu=uy$, $vx=yv=uvu=vuv=0$. This is the minimal wild algebra $W_{32}$ in
\cite[Table W]{Han02}; hence
$\Gam(\DynE_6,\{3,5\})$ is wild.

For the pairs containing vertex $4$, each one has the idempotent truncation $\Gam(\DynE_6,\{4\})$, whose one-vertex bound quiver is
\[
\vcenter{\xymatrix@C=3.5em{
4\ar@(ul,dl)[]_{x}\ar@(ur,dr)[]^{y}
}}
\]
with relations $x^2=y^3=(x+y)^3=0$. The additional relations $yx=xy^2=0$ give the wild quotient in Proposition~\ref{prop:local-benchmark-alg}\textup{(4)}. Thus, every
two-vertex subset containing $4$ gives a wild algebra.

We finally consider the subsets with at least three vertices. Since any such subset containing $4$ is wild, the wild two-vertex cases eliminate every triple in
$\{1,2,3,5,6\}$ except $\Omega=\{1,2,6\}$. For this remaining triple, $\Gam(\DynE_6,\{1,2,6\})$ admits the following quiver
\[
\vcenter{\xymatrix@C=5em@R=3.5em{
&2\ar@(ul,ur)[]^{z}\ar@/^/[dl]^{v}\ar@/^/[dr]^{p}&\\
1\ar@/^/[ur]^{u}\ar@/^/[rr]^{s}
&&
6\ar@/^/[ll]^{t}\ar@/^/[ul]^{q}
}}\; .
\]
One then find that the separated quiver of its radical-square-zero quotient is neither Dynkin nor Euclidean. Thus, $\Gam(\DynE_6,\{1,2,6\})$ is wild. Every subset with at least four vertices contains a wild
three-vertex idempotent truncation. This completes the proof.
\end{proof}

\begin{proposition}\label{prop:E7-main}
Let $\varnothing\ne\Omega\subseteq\{1,2,\ldots,7\}$. Then, $\Gam(\DynE_7,\Omega)$ is representation-finite if $\Omega=\{7\}$, tame if $\Omega=\{1\},\{6\},\{1,7\},\{6,7\}$, and wild otherwise.
\end{proposition}
\begin{proof}
Killing vertex $7$ gives a quotient of type~$\DynE_6$, so Proposition~\ref{prop:E6-main} eliminates every subset whose intersection with $\{1,\ldots,6\}$ is already wild.

We consider the remaining one-vertex cases. Taking $x:=a_1a_3a_4b_4b_3b_1$ and $y:=a_1a_3a_4a_5a_6b_6b_5b_4b_3b_1$ at vertex $1$ give $\Gam(\DynE_7,\{1\})\cong\Bbbk\langle x,y\rangle/(x^2,y^2,xy+yx)$, which is tame. At vertex $6$, taking $x:=a_6b_6$ and $y:=b_5b_4b_3a_3a_4a_5$ with the relations are $x^2=y^2=0$ and $(xy)^2+(yx)^2=0$ gives a representation-infinite special biserial algebra, and hence tame. Lastly, $\Gam(\DynE_7,\{7\})\cong \Bbbk[x]/(x^3)$ is representation-finite.

For other remaining one-vertex cases, the wild local quotients of $\Gam(\DynE_7,\Omega)$ can be specified as follows. At vertex $2$, take $x=b_2a_4b_4a_2$ and $y=b_2a_4a_5b_5b_4a_2$; the corner has relations $xy+yx=0$, $x^3=y^2$, and $x^2y=0$. Quotienting by $x^3$ gives $\Bbbk\langle x,y\rangle/(x^3,y^2,x^2y,xy+yx)$. At vertex  $3$, take $x=a_3b_3$ and $y=a_3a_4b_4b_3$ and impose $yx=y^3=xy^2=0$; the result is $\Bbbk\langle x,y\rangle/(x^2,y^3,yx,xy^2)$. At vertex $5$, take $x=a_5b_5$ and $y=b_4b_3a_3a_4$ and impose $xy-yx=x^2y=y^2=0$; the result is $\Bbbk\langle x,y\rangle/(x^3,y^2,x^2y,xy-yx)$. These are precisely the wild algebras \textup{(7)}, \textup{(4)} of Proposition~\ref{prop:local-benchmark-alg}. 

It remains to consider the pairs $\{1,6\}$, $\{1,7\}$, and $\{6,7\}$. The pair $\{1,6\}$ has two loops at $6$, a loop at $1$, and opposite arrows between its vertices, so its radical-square-zero quotient is wild. For $\{1,7\}$, put $u=a_1a_3a_4a_5a_6$, $v=b_6b_5b_4b_3b_1$, $x=a_1a_3a_4b_4b_3b_1$, and $y=b_6b_5b_4b_3a_3a_4a_5a_6$. The relations are $x^2=uy=yv=vu=0$, $xuv+uvx=0$, $y^2=vxu$. This is the same special biserial presentation as $\Gam(\DynD_6,\{2,5\})$. For $\{6,7\}$, it gives a special biserial algebra. Since the idempotent truncations at $1$ and $6$ are representation-infinite, both $\Gam(\DynE_7,\{1,7\})$ and $\Gam(\DynE_7,\{6,7\})$ are tame.

Every subset with at least three vertices either contains a wild singleton or is $\{1,6,7\}$, which contains the wild idempotent truncation $\Gam(\DynE_7,\{1,6\})$. Hence, every such subset is wild.
\end{proof}

\begin{proposition}\label{prop:E8-main}
Let $\varnothing\ne\Omega\subseteq\{1,2,\ldots,8\}$. Then, $\Gam(\DynE_8,\Omega)$ is tame if $\Omega=\{8\}$, and wild otherwise.
\end{proposition}
\begin{proof}
Killing vertex $8$ gives a quotient of type~$\DynE_7$. Hence, Proposition~\ref{prop:E7-main} shows that a potentially non-wild subset meets $\{1,\ldots,7\}$ either trivially or in one of $\{1\},\{6\},\{7\},\{1,7\},\{6,7\}$. Killing vertex $1$ gives a full subdiagram of type~$\DynD_7$, and Proposition~\ref{thm:markedD-main} excludes every subset containing $6$. Thus every potentially non-wild subset is contained in $\{1,7,8\}$.

At vertex $1$, we have $e_1\Pi_{\DynE_8}e_1\simeq\Bbbk\langle x,y\rangle/(x^4,y^2,xy+yx)$, whose quotient by $x^3$ and $x^2y$ is the wild algebra in Proposition~\ref{prop:local-benchmark-alg}\textup{(7)}. At vertex $7$, take $x=b_6b_5b_4b_3a_3a_4a_5a_6$ and $y=a_7b_7$. Imposing $xy-yx=x^2y=0$ gives the five-dimensional quotient $\Bbbk\langle x,y\rangle/(x^3,y^2,x^2y,xy-yx)$ with basis $1,x,x^2,y,xy$, which is Proposition~\ref{prop:local-benchmark-alg}\textup{(7)} with $\lambda=1$. Hence,  $\Gam(\DynE_8,\{1\})$ and $\Gam(\DynE_8,\{7\})$ are wild.

At vertex $8$, put $p=b_7b_6b_5b_4b_3$, $q=a_3a_4a_5a_6a_7$, and take $x=pq$, $y=p\,a_3a_4a_5a_6b_6b_5b_4b_3\,q$. Path reduction gives $x^2=y^2=0$ and $xy+yx=0$, with basis $1,x,y,xy$. Thus, $\Gam(\DynE_8,\{8\})$ is the special biserial algebra $\Bbbk\langle x,y\rangle/(x^2,y^2,xy+yx)$, and it is representation-infinite tame. Every subset of $\{1,7,8\}$ other than $\{8\}$ contains one of the wild case $\{1\}$ or $\{7\}$. Therefore, $\Gam(\DynE_8,\Omega)$ is wild for every $\Omega\ne\{8\}$.
\end{proof}

\section{Contraction algebras under iterated flops}\label{sec:flop-detection}
In this section, we shall use the classification of the contracted preprojective algebra $\Gam(\Delta,\Omega)$ to obtain further information about the contraction algebra $\Acon$. We introduce mutation in order to keep track of the surface data under iterated flops. On the algebraic side, these flops are realised by mutation of the corresponding CM generators. More precisely, under the Auslander--McKay correspondence, simple mutation of a basic CM generator corresponds to the simple flop of the associated exceptional curve; equivalently, the mutation graph of the corresponding generators agrees with the flop graph. We refer to \cite[Corollary~1.12]{Wemyss18} for this correspondence.
 
 The main point is that the surface data may change after a flop, whereas
the representation type of the corresponding contraction algebra does
not. Thus, even when the contracted preprojective algebra associated to
the original flopping contraction is representation-finite or tame, an
iterated flop may produce another surface quotient which is representation-infinite or wild. We will  describe how the surface data change under mutation, and then determine the resulting iterated-flop orbits in types $\DynA$, $\DynD$, and $\DynE$, and finally apply the surface quotient classification from Step~1 to these orbits.

\subsection{Mutation and flop}\label{subsec:classical-safe}
Let $M=R\oplus N_1\oplus\cdots\oplus N_n$
be the basic CM generator corresponding to
$f\colon X\to\operatorname{Spec}R$. Fix $1\le i\le n$ and set
$M_i:=M/N_i$. Choose a minimal right $\operatorname{add}M_i$-approximation
$V_i\to N_i$ and let $K_i$ be its kernel, so that $0\longrightarrow K_i\longrightarrow V_i\longrightarrow N_i$ is exact with $V_i\in\operatorname{add}M_i$. We then call $\mu_i^+(M):=M_i\oplus K_i$ the simple \emph{right mutation} of $M$ at $N_i$; the simple \emph{left mutation} $\mu_i^-(M)$ is defined dually. In our cDV setting, however, no distinction between left and right mutation is necessary since $\mu_i^+(M)\cong\mu_i^-(M)$ by \cite[Theorem~1.25]{IyamaWemyss14}. We therefore write simply $\mu_i(M)$ for this common mutation. Moreover, the same theorem gives $\mu_i^2(M)\cong M$. When the mutation $\mu_i(M)$ is realised geometrically, it corresponds to the simple flop of $f$ at $C_i$. Hence, mutation provides an algebraic way to move between the different flopping contractions over the same singularity. In particular, the corresponding contraction algebras have the same representation type by Proposition~\ref{prop:chamber-invariance}.

We now recall from \cite[Definition~1.16 and Theorem~1.20]
{IyamaWemyssTits} how mutation changes the $\DynA\DynD\DynE$ data of a sufficiently general surface section. Let $\Delta$ be the corresponding Dynkin diagram and $\Omega\subseteq\Delta_0$ the subset corresponding to the exceptional curves. We denote $I:=\Delta_0\setminus\Omega$ and fix $i\in\Omega$. Let $\Gamma_i$ be the connected component containing $i$ in the full subdiagram of $\Delta$ on $I\cup\{i\}$. Since $\Gamma_i$ is a Dynkin diagram, its preprojective algebra $\Pi_{\Gamma_i}$ is finite-dimensional and self-injective. Let $\iota_I$ be the Nakayama permutation of $\Pi_{\Gamma_i}$. Then, the mutation at $i$ changes the marked subset $I$ by $\nu_i(I):=(I\cup\{i\})\setminus\{\iota_I(i)\}$. Equivalently, the resulting subset after mutation is $\Omega':=\Delta_0\setminus\nu_i(I)$. Hence, the combinatorial effect of mutation is obtained by adjoining $i$ to the marked set and removing the vertex $\iota_I(i)$ determined by the Nakayama permutation of the local Dynkin preprojective algebra $\Pi_{\Gamma_i}$. 

Starting from $(\Delta,\Omega)$, successive mutations allow us to determine the surface data associated to iterated flops. If some iterated flop gives a subset $\Omega'$ for which $\Gam(\Delta,\Omega')$ is wild, then the corresponding contraction algebra is wild by Proposition~\ref{prop:shadow}. Moreover, Proposition~\ref{prop:chamber-invariance} implies that the original contraction algebra is also wild. We mention that, however, the representation type of $\Gam(\Delta,\Omega)$ is not necessarily the same as that of $\Gam(\Delta,\Omega')$.

\subsection{Representation type}
Let $f\colon X\longrightarrow \operatorname{Spec}R$ be a flopping contraction
in our standing setup, and let $\Acon$ be its contraction algebra. If
$f'\colon X'\longrightarrow \operatorname{Spec}R$ is obtained from $f$ by
iterated flops, write $\Acon'$ for the corresponding contraction algebra and
$(\Delta,\Omega')$ for the Dynkin data of a sufficiently general surface
section of $f'$. By Proposition~\ref{prop:chamber-invariance}, $\Acon$ and
$\Acon'$ have the same representation type, while
Proposition~\ref{prop:shadow} gives a surjection $\Acon'\twoheadrightarrow \Gam(\Delta,\Omega')$. Therefore, if a contracted preprojective quotient associated to an iterated
flop is representation-infinite or wild, then the same conclusion holds for
$\Acon$. Step~2 is thus reduced to determining the surface data in the
mutation orbit and reading off the strongest representation-theoretic
conclusion from the corresponding quotients.

We first consider type~$\DynA$. Let $\Omega=\{j_1<\cdots<j_r\}\subseteq\{1,\ldots,m\}$. Recall $(d_0,\ldots,d_r)$ with $d_0=j_1$, $d_s:=j_{s+1}-j_s$ for $1\le s\le r-1$, and $d_r:=m-j_r+1$, from Section~\ref{sec:A-stable}.

\begin{lemma}
\label{lem:A-gap-mutation}
Simple mutation at the vertex $j_s$ interchanges the adjacent entries
$d_{s-1}$ and $d_s$ and leaves all other entries unchanged. Consequently,
the surface data occurring in the iterated-flop class are precisely those
obtained from the distinct permutations of $(d_0,\ldots,d_r)$.
\end{lemma}

\begin{proof}
Set $I=\Delta_0\setminus\Omega$ and use the conventions $j_0:=0$ and
$j_{r+1}:=m+1$. The connected component containing $j_s$ in the full
subdiagram on $I\cup\{j_s\}$ is the interval $[j_{s-1}+1,j_{s+1}-1]$, which is of type $\DynA_{d_{s-1}+d_s-1}$. Its Nakayama permutation reverses this interval and therefore sends $j_s$ to $j_{s-1}+j_{s+1}-j_s$. Thus mutation at $j_s$ interchanges
$d_{s-1}$ and $d_s$ and fixes all other entries. Since adjacent
transpositions generate the symmetric group, all distinct permutations of
the parameter sequence occur in the iterated-flop class.
\end{proof}

We set $\mathbf{d}:=(d_0,\ldots,d_r)$ and $h(\mathbf{d}):=\#\{s\mid d_s\ge2\}$. The classification in Theorem~\ref{thm:markedA-main} now gives the complete
output of Step~2 in type~$\DynA$.

\begin{proposition}\label{prop:A-mutation-safe}
Let $\Acon$ be the contraction algebra of a flopping contraction whose
general surface section has type $\DynA_m$ and the subset
$\Omega=\{j_1<\cdots<j_r\}$ with parameter sequence
$\mathbf{d}=(d_0,\ldots,d_r)$. Then,
\[
\begin{array}{c|c|c|c}
r
&
\begin{array}{c}
\text{all surface quotients}\\
\text{are representation-finite}
\end{array}
&
\begin{array}{c}
\Acon \text{ is representation-infinite}
\end{array}
&
\begin{array}{c}
\Acon\text{ is wild}
\end{array}
\\ \hline
1
&
\text{always}
&
-
&
-
\\[2pt]
2
&
\min\{d_0,d_1,d_2\}=1
&
\min\{d_0,d_1,d_2\}=2
&
\min\{d_0,d_1,d_2\}\ge3
\\[2pt]
3
&
h(\mathbf{d})\le1
&
h(\mathbf{d})=2
&
h(\mathbf{d})\ge3
\\[2pt]
4
&
h(\mathbf{d})=0
&
h(\mathbf{d})=1
&
h(\mathbf{d})\ge2
\\[2pt]
5
&
-
&
h(\mathbf{d})=0
&
h(\mathbf{d})\ge1
\\[2pt]
r\ge6
&
-
&
-
&
\text{always}.
\end{array}
\]
In the middle column, no contracted preprojective quotient occurring in the
iterated-flop class is wild, but at least one such quotient is tame.
\end{proposition}

\begin{proof}
By Lemma~\ref{lem:A-gap-mutation}, every permutation of
$\mathbf{d}=(d_0,\ldots,d_r)$ occurs in the iterated-flop class. The table
then follows directly from Theorem~\ref{thm:markedA-main}, except that the
case $r=3$ requires keeping track of the positions of the entries
$d_s\ge2$.

Suppose $r=3$. If $h(\mathbf{d})\le1$, every permutation belongs to the
representation-finite part of Theorem~\ref{thm:markedA-main}. If
$h(\mathbf{d})=2$, no permutation is wild. The placements of the two entries
which are at least $2$ with index sets $\{0,1\}$, $\{1,2\}$, $\{2,3\}$, $\{0,3\}$ give tame surface quotients, whereas the placements $\{0,2\}$, $\{1,3\}$ are representation-finite. Since all permutations occur in the
iterated-flop class, this class contains a tame quotient. Hence $\Acon$ is
representation-infinite by Corollary~\ref{cor:wildness-propagation}. If
$h(\mathbf{d})\ge3$, every permutation is wild, and hence $\Acon$ is wild.
\end{proof}

It follows from Proposition~\ref{prop:A-mutation-safe} that every contracted
preprojective algebra in the iterated-flop class is representation-finite if
and only if $r=1$, or $r=2$ with $\min\{d_0,d_1,d_2\}=1$, or $r=3$ with
$h(\mathbf{d})\le1$, or $r=4$ with $h(\mathbf{d})=0$. There is no such case
for $r\ge5$. Moreover, the only case in which the original contracted
preprojective quotient may be representation-finite while another one in the
same iterated-flop class is representation-infinite is $r=3$ with
$h(\mathbf{d})=2$. In this case, some permutations of $\mathbf{d}$ give
representation-finite surface quotients, while others give tame
representation-infinite surface quotients.

We next consider type~$\DynD$. In this case, the surface representation type
is unchanged under iterated flops.

\begin{proposition}\label{prop:D-mutation-safe}
Let $\varnothing\ne\Omega\subseteq\{1,\ldots,m\}$ and set $r:=|\Omega|$.
Then the representation type of $\Gam(\DynD_m,\Omega)$ is preserved under
iterated flops. Moreover,
\[
\begin{array}{c|c|l|c}
r
&
\begin{array}{c}
\text{all surface quotients}\\
\text{are representation-finite}
\end{array}
&
\begin{array}{c}
\Acon\text{ is representation-infinite}
\end{array}
&
\begin{array}{c}
\Acon\text{ is wild}
\end{array}
\\ \hline
1
&
\begin{array}{c}
\Omega=\{1\},\ \{m-1\},\ \{m\}
\end{array}
&
\begin{array}{c}
\Omega=\{2\},\ \{m-2\}
\end{array}
&
\text{otherwise}
\\[6pt]
2
&
\begin{array}{c}
\Omega=\{1,m-1\},\ \{1,m\},\\
\{m-1,m\}
\end{array}
&
\begin{array}{c}
\Omega=\{1,2\},\ \{2,m-1\},\ \{2,m\},\\
\{m-2,m-1\},\ \{m-2,m\}
\end{array}
&
\text{otherwise}
\\[10pt]
3
&
-
&
\begin{array}{l}
\Omega=\{1,2,m-1\},\ \{1,2,m\},\\
\{1,m-1,m\},\ \{m-2,m-1,m\}
\end{array}
&
\text{otherwise}
\\[10pt]
4
&
-
&
\begin{array}{c}
m=4,\quad
\Omega=\{1,2,3,4\}
\end{array}
&
\text{otherwise}
\\[6pt]
r\ge5
&
-
&
-
&
\text{always}.
\end{array}
\]
\end{proposition}

\begin{proof}
For a connected component of type $\DynD_\ell$, the Nakayama permutation of
its preprojective algebra is the identity when $\ell$ is even and, when
$\ell$ is odd, interchanges the two short-arm endpoints and fixes all
remaining vertices. For a component of type $\DynA$, the Nakayama
permutation is the reversal. Using the mutation rule above, a direct check
shows that the representation-finite locus and the tame locus in
Proposition~\ref{thm:markedD-main} are separately stable under simple mutation.
Since simple mutation is involutive, the wild locus is stable as well.
Hence the representation type of $\Gam(\DynD_m,\Omega)$ is preserved under
iterated flops. Lastly, the table now follows from Proposition~\ref{thm:markedD-main}. 
\end{proof}

We finally consider the exceptional types. As in type~$\DynD$, the
finite--tame--wild representation type of the surface quotient is preserved
under iterated flops.

\begin{proposition}\label{prop:exceptional-mutation-orbits}
Let $\Delta=\DynE_6,\DynE_7,\DynE_8$. Then, the representation type of $\Gam(\Delta,\Omega)$ is preserved under iterated flops. Moreover,
\[
\begin{array}{c|c|c|c}
\Delta
&
\begin{array}{c}
\text{all surface quotients}\\
\text{are representation-finite}
\end{array}
&
\begin{array}{c}
\Acon\text{ is representation-infinite}
\end{array}
&
\begin{array}{c}
\Acon\text{ is wild}
\end{array}
\\ \hline
\DynE_6
&
\Omega=\{1\},\ \{6\}
&
\begin{array}{c}
\Omega=\{2\},\ \{3\},\ \{5\},\ \{1,6\},\\
\{1,2\},\ \{1,3\},\ \{2,6\},\ \{5,6\}
\end{array}
&
\text{otherwise}
\\[10pt]
\DynE_7
&
\Omega=\{7\}
&
\begin{array}{c}
\Omega=\{1\},\ \{6\},\ \{1,7\},\ \{6,7\}
\end{array}
&
\text{otherwise}
\\[8pt]
\DynE_8
&
-
&
\Omega=\{8\}
&
\text{otherwise}.
\end{array}
\]
\end{proposition}
\begin{proof}
For the full exceptional diagrams, the Nakayama permutation is the nontrivial diagram involution $1\leftrightarrow 6$ and $3\leftrightarrow 5$ in type $\DynE_6$, and is the identity in types $\DynE_7$ and $\DynE_8$.
For a proper connected component $\Gamma_i$, we use the Nakayama permutation of the preprojective algebra $\Pi_{\Gamma_i}$, according to the Dynkin type of $\Gamma_i$.

In type $\DynE_6$, the representation-finite singletons $\{1\}$ and
$\{6\}$ are exchanged; the tame singleton $\{2\}$ is fixed, while $\{3\}$ and $\{5\}$ are exchanged, the pair $\{1,6\}$ is fixed, and the
remaining tame two-vertex subsets form the mutation orbit
\[
\{1,2\}\longleftrightarrow\{1,3\},\qquad
\{1,2\}\longleftrightarrow\{2,6\},\qquad
\{2,6\}\longleftrightarrow\{5,6\}.
\]
Thus, the representation-finite and tame loci of
Proposition~\ref{prop:E6-main} are separately stable under mutation. In type $\DynE_7$, the representation-finite singleton $\{7\}$ and the tame singletons $\{1\}$ and $\{6\}$ are fixed. For $\Omega=\{1,7\}$, mutation at the vertex $1$ uses the local $\DynE_6$ component on the vertices $\{1,\ldots,6\}$, whose Nakayama
permutation sends $1$ to $6$. Hence, $\{1,7\}\longleftrightarrow\{6,7\}$. These are exactly the tame two-vertex cases in
Proposition~\ref{prop:E7-main}. In type $\DynE_8$, the only non-wild subset is the tame singleton $\{8\}$, and it is fixed by mutation.

In each exceptional type, the representation-finite and tame loci are separately stable under simple mutation. Since simple mutation is
involutive, the wild locus is also stable. Hence, the representation type of the surface quotient is preserved under iterated flops. Then, the table follows from Propositions~\ref{prop:E6-main}, \ref{prop:E7-main}, and
\ref{prop:E8-main}.
\end{proof}

We are able to see that \emph{surface quotient wildness is invariant under iterated flops in every simply-laced Dynkin type}.

\subsection{Cases remaining for Step~3}
\label{subsec:step3-remaining}
We list the remaining cases before passing to the lifting problem. First, all one-curve cases which are not already wild at the surface quotient level remain to be considered; they will be treated uniformly at the beginning of Step~3 by studying the corresponding one-simple contraction
algebras.

For type $\DynA$ with at least two curves, the remaining cases are exactly
\[
\begin{array}{c|c}
r & \text{condition}\\ \hline
2 & \min\{d_0,d_1,d_2\}\le2,\\
3 & h(\mathbf{d})\le2,\\
4 & h(\mathbf{d})\le1,\\
5 & h(\mathbf{d})=0.
\end{array}
\]
There is no remaining type $\DynA$ case for $r\ge6$. Among these cases, the first column of Proposition~\ref{prop:A-mutation-safe} consists of those for which every surface quotient in the iterated flop class is
representation-finite, while the second column consists of those for which $\Acon$ is already known to be representation-infinite.

For type $\DynD$ with at least two curves, the remaining cases are exactly
\[
\begin{aligned}
\Omega={}&
\{1,m-1\},\ \{1,m\},\ \{m-1,m\},\\
&\{1,2\},\ \{2,m-1\},\ \{2,m\},\
\{m-2,m-1\},\ \{m-2,m\},\\
&\{1,2,m-1\},\ \{1,2,m\},\
\{1,m-1,m\},\ \{m-2,m-1,m\},
\end{aligned}
\]
together with the case $m=4$ and $\Omega=\{1,2,3,4\}$. The first three two-vertex subsets are the representation-finite surface
cases. All other cases in the display are tame at the surface level and therefore already imply that $\Acon$ is representation-infinite.

For the exceptional types with at least two curves, the only non-wild cases, up to iterated flops, are represented by
\[
\mathcal{E}_1=(\DynE_6,\{1,6\}),\qquad
\mathcal{E}_2=(\DynE_6,\{1,2\}),\qquad
\mathcal{E}_3=(\DynE_7,\{6,7\}).
\]
Their surface data orbits are
\[
\mathcal O(\mathcal{E}_1)=\bigl\{\{1,6\}\bigr\},\quad
\mathcal O(\mathcal{E}_2)=\bigl\{\{1,2\},\{1,3\},\{2,6\},\{5,6\}\bigr\},\quad
\mathcal O(\mathcal{E}_3)=\bigl\{\{1,7\},\{6,7\}\bigr\}.
\]
Every contracted preprojective algebra occurring in these three orbits is tame and representation-infinite. Hence, the corresponding contraction algebras are representation-infinite, and Step~3 only has to decide whether they are tame or wild.

\section{Representation type of the remaining cases}
\label{sec:lifts} We now complete Step~3 of the classification. We shall show that the surface quotient determines the arrows between distinct vertices, while one-curve subset contraction algebras provide information about loops, possibly after iterated flops. We then examine the radical-square-zero quotient to detect wildness. If this is insufficient, we consider the multiplication map
\[
\mathfrak{m}_{\rad}\colon
(\rad \Acon/\rad^2\Acon)\otimes_{\Acon/\rad \Acon}
(\rad \Acon/\rad^2\Acon)
\longrightarrow
\rad^2\Acon/\rad^3\Acon,
\]
In a few loopless cases, we shall also need the third radical layer. We write 
\[
\operatorname{gr}_{\rad}\Acon:=\bigoplus_{n\ge0}\rad^n\Acon/\rad^{n+1}\Acon
\]
for the associated graded algebra with respect to the radical filtration.

\subsection{One exceptional curve}
For one-curve flops of types~$\DynA$ and~$\DynD$, explicit presentations of the contraction algebras are available in \cite{BW25}. We instead use the radical filtration and the local algebras from Proposition~\ref{prop:local-benchmark-alg}. Suppose that $\Acon$ has only one simple module. Then $\Acon$ is local, and its Gabriel quiver has $\dim_\Bbbk\rad \Acon/\rad^2\Acon$ loops. The cases of dimension $0$ and $1$ give $\Bbbk$ and $\Bbbk[x]/(x^a)$, respectively, while dimension at least $3$ are wild by the separated-quiver criterion. Thus, it remains to consider the case of two loops.

\begin{lemma}\label{lem:onecurve-central-loop-lift}
Let $\Gam(\Delta,\{i\})\cong \Acon/(g)$ with $g\in Z(\Acon)\cap\rad \Acon$ be the surface quotient of $\Acon$ associated with a one-curve contraction. Then, the natural map
\[
\rad \Acon/\rad^2\Acon
\longrightarrow
\rad\Gam(\Delta,\{i\})/
\rad^2\Gam(\Delta,\{i\})
\]
is surjective, and its kernel has dimension at most one. Consequently, the number of loops in the quiver of $\Acon$ is equal to that of $\Gam(\Delta,\{i\})$, or exceeds it by one.
\end{lemma}
\begin{proof}
We have $\rad\Gam(\Delta,\{i\})/\rad^2\Gam(\Delta,\{i\})\cong \rad \Acon/ (\rad^2\Acon+(g))$. Since $g\rad \Acon\subseteq\rad^2\Acon$, the kernel is spanned by the image of $g$ modulo $\rad^2\Acon$.
\end{proof}

For a local algebra of the form $\mathbb C[[x,y]]/(f,h)$, we call the span of the degree-two components of its defining relations the \emph{quadratic relation space}. Over $\mathbb C$, a nonsquare binary quadratic form is equivalent to $xy$ under an invertible linear change of variables.

\begin{lemma}\label{lem:two-generator-normal-form}
Suppose that $\Acon\cong\mathbb C[[x,y]]/(f,h)$ is finite-dimensional with $\dim_{\mathbb C}\rad \Acon/\rad^2\Acon=2$. If the quadratic relation space contains a nondegenerate quadratic form, then $\Acon\cong \mathbb C[[x,y]]/(xy,x^p-y^q)$ for some integers $p,q\ge2$, after an invertible formal change of the variables $x,y$. Moreover, if the quadratic relation space is one-dimensional, then $p,q\ge3$. If the quadratic relation space is two-dimensional, then $\min\{p,q\}=2$.
\end{lemma}
\begin{proof}
By an invertible formal change of variables and a choice of generators for the defining ideal, a relation with nondegenerate quadratic part may be brought to $xy$. Thus, $\Acon\cong\mathbb C[[x,y]]/(xy,h)$. Modulo $(xy)$, write $h=F(x)+G(y)$.  Finite dimensionality forces both summands to be nonzero, so $F(x)=x^pu(x)$ and $G(y)=y^qv(y)$ with $u(0),v(0)\ne0$. Absorbing the units into separate changes of $x$ and $y$ gives $\Acon\cong\mathbb C[[x,y]]/(xy,x^p-y^q)$. The final assertions follow from whether the second relation has a quadratic part.
\end{proof}

The classification in Step~1 shows that every non-wild one-curve surface quotient has at most two loops. The surface quotients with at most one loop are those of type~$\DynA$, together with $\Gam(\DynD_m,\{1\})$, $\Gam(\DynD_m,\{m-1\})$, $\Gam(\DynD_m,\{m\})$, $\Gam(\DynE_6,\{1\})$, $\Gam(\DynE_6,\{6\})$, and $\Gam(\DynE_7,\{7\})$. The remaining non-wild one-curve surface quotients, each with two loops, are
\[
\begin{array}{llll}
\Gam(\DynD_m,\{2\}),\quad & \Gam(\DynD_m,\{m-2\}),\quad  & \Gam(\DynE_6,\{2\}), \quad & \Gam(\DynE_6,\{3\}),\\
\Gam(\DynE_6,\{5\}), \quad & \Gam(\DynE_7,\{1\}),\quad  & \Gam(\DynE_7,\{6\}),\quad & \Gam(\DynE_8,\{8\}).
\end{array}
\]

\begin{proposition}\label{prop:onecurve-classification}
Let $\Acon$ be a one-curve contraction algebra which survives after Step~2, and let $\Gam(\Delta,\{i\})$ be its surface quotient.
\begin{enumerate}
\item If $\dim_\Bbbk\rad \Acon/\rad^2\Acon\le1$, then $\Acon$ is representation-finite. If $\dim_\Bbbk\rad \Acon/\rad^2\Acon\ge3$, then $\Acon$ is wild.

\item Suppose that $\dim_\Bbbk\rad \Acon/\rad^2\Acon=2$ and that the surface quotient has at most one loop. Then, $\Acon$ is commutative. Let $\eta\colon \operatorname{Sym}^2(\rad \Acon/\rad^2\Acon)\longrightarrow\rad^2\Acon/\rad^3\Acon$ be the map induced by multiplication.
\begin{enumerate}
\item If $\dim_{\Bbbk} \ker \eta=0$, then $\Acon$ is wild.
\item If $\dim_{\Bbbk} \ker \eta=1$, then $\Acon$ is wild when the generator of $\ker \eta$ is the square of a nonzero linear form, and is tame when the generator of $\ker \eta$ is nonsquare.
\item If $\dim_{\Bbbk} \ker \eta=2$, then $\Acon$ is tame.
\end{enumerate}

\item Suppose that $\dim_\Bbbk\rad \Acon/\rad^2\Acon=2$. If the surface data are one of $(\DynD_m,\{2\})$, $(\DynE_6,\{2\})$, $(\DynE_7,\{1\})$, $(\DynE_8,\{8\})$, then $\dim_\Bbbk\rad^2\Acon/\rad^3\Acon$ is either $1$ or $2$. Moreover,
\begin{enumerate}
\item if $\dim_\Bbbk\rad^2\Acon/\rad^3\Acon=1$, then $\Acon$ is tame.

\item if $\dim_\Bbbk\rad^2\Acon/\rad^3\Acon=2$, then $\Acon$ is tame if there exist linearly independent elements $u,v\in\rad \Acon/\rad^2\Acon$ with $u\otimes u,v\otimes v\in\ker\mathfrak{m}_{\rad}$, and is wild otherwise.
\end{enumerate} 
If the surface data are $(\DynD_m,\{m-2\})\ (m\ge5)$, $(\DynE_6,\{3\})$, $(\DynE_6,\{5\})$, $(\DynE_7,\{6\})$, then $\dim_\Bbbk\rad^2\Acon/\rad^3\Acon$ is either $2$ or $3$. Moreover, $\Acon$ is tame when this dimension is $2$, and is wild when it is $3$.
\end{enumerate}
\end{proposition}
\begin{proof}
Part~\textup{(1)} follows from the separated-quiver criterion and the classification of local algebras. We therefore assume $\dim_\Bbbk\rad \Acon/\rad^2\Acon=2$.

For \textup{(2)}, suppose that the surface quotient has at most one loop. Lemma~\ref{lem:onecurve-central-loop-lift} shows that it has exactly one and that the kernel on the first radical layer is one-dimensional. Choose $x$ lifting this loop and put $y=g$, and then $x,y$ generate the radical of $\Acon$. As $g$ is central, $xy=yx$, and hence $\Acon$ is commutative. Since $\Acon$ is a zero-dimensional Gorenstein local algebra of embedding dimension two, it is a complete intersection; thus $\Acon\cong\mathbb C[[x,y]]/(f,h)$. If the quadratic relation space is zero, then $\Acon/\rad^3\Acon\cong \Bbbk\langle x,y\rangle/(x^3,y^3,xy-yx,x^2y,xy^2)$, which is the wild algebra in Proposition~\ref{prop:local-benchmark-alg}\textup{(6)}.

If the quadratic relation space is one-dimensional and generated by the square of a nonzero linear form, an invertible linear change of $x,y$ allows us to assume that the generator is $x^2$. Then, $\Acon/\rad^3\Acon\cong \Bbbk\langle x,y\rangle/(x^2,y^3,xy^2,xy-yx)$. After exchanging $x$ and $y$, this is Proposition~\ref{prop:local-benchmark-alg}\textup{(7)} with $\lambda=1$. Hence, $\Acon$ is wild.

If the quadratic relation space is one-dimensional and its generator is nonsquare, then it contains a nondegenerate binary quadratic form. Moreover, Lemma~\ref{lem:two-generator-normal-form} gives $\Acon\cong\Bbbk[[x,y]]/(xy,x^p-y^q)$ for some $p,q\ge3$. This algebra is a representation-infinite special biserial algebra, and hence tame.

If the quadratic relation space is two-dimensional, it contains a nonsquare quadratic form. Indeed, the squares of linear forms cannot span a two-dimensional vector subspace consisting entirely of squares: the sum of the squares of two linearly independent forms is nonsquare over $\mathbb C$ after a suitable nonzero scalar multiple. Lemma~\ref{lem:two-generator-normal-form} therefore gives $\Acon\cong\Bbbk[[x,y]]/(xy,x^p-y^q)$ for some $p,q\ge2$ with $\min\{p,q\}=2$. Again, this is special biserial and tame.

For \textup{(3)}, suppose that the surface quotient has two loops. The first radical layers are isomorphic, so $g\in\rad^2\Acon$, and $\rad^2\Acon/\rad^3\Acon\longrightarrow \rad^2\Gam(\Delta,\{i\})/\rad^3\Gam(\Delta,\{i\})$ is surjective with kernel of dimension at most one.

For the cases $(\DynD_m,\{2\})$, $(\DynE_6,\{2\})$, $(\DynE_7,\{1\})$, and $(\DynE_8,\{8\})$, Step~1 gives $\Gam(\Delta,\{i\}) \cong \Bbbk\langle x,y\rangle/(x^2,y^2,xy+yx)$. Thus, its second radical layer is one-dimensional, and its quadratic relation space is $\langle x^2,y^2,xy+yx\rangle$. It follows that $\dim_\Bbbk\rad^2\Acon/\rad^3\Acon$ is either $1$ or $2$.
\begin{itemize}
\item If this dimension is $1$, then the quadratic relation space of $\Acon$ is also $\langle x^2,y^2,xy+yx\rangle$. Hence, $x^2=y^2=xy+yx=0$ in the associated graded algebra. These relations force every product of three loops to vanish. Thus, $\rad^3\Acon/\rad^4\Acon=0$ and  $\rad^3\Acon=0$. Since the map on the second radical layers is an isomorphism, the image of $g$ modulo $\rad^3\Acon$ is zero. Therefore $g=0$ and $\Acon\cong\Gam(\Delta,\{i\})$, which is tame.

\item If $\dim_\Bbbk\rad^2\Acon/\rad^3\Acon=2$, the quadratic relation space is a two-dimensional subspace of $\langle x^2,y^2,xy+yx\rangle$. If it contains the squares of two linearly independent linear forms, an invertible change of generators makes it $\langle x^2,y^2\rangle$. Hence, the initial quadratic relations of $\Acon$ are $x^2$ and $y^2$. Thus $\operatorname{gr}_{\rad}\Acon$ is a quotient of $\Bbbk\langle x,y\rangle/(x^2,y^2)$, so it is special biserial and tame. Since $\Acon$ degenerates to its associated graded algebra, $\Acon$ is tame. Otherwise, an elementary calculation with binary quadratic forms shows that, after an invertible linear change of $x,y$, the quadratic relation space is $\langle x^2,xy+yx\rangle$. Therefore, $\Acon/\rad^3\Acon$ has exactly these two quadratic relations and all products of length three vanish. After exchanging $x$ and $y$, it is $\Bbbk\langle x,y\rangle/ (x^3,y^2,x^2y,xy+yx)$, which is Proposition~\ref{prop:local-benchmark-alg}\textup{(7)} with $\lambda=-1$. Hence, $\Acon$ is wild.
\end{itemize}

It remains to consider $(\DynD_m,\{m-2\})$ with $m\ge5$, $(\DynE_6,\{3\})$, $(\DynE_6,\{5\})$, and $(\DynE_7,\{6\})$. Step~1 shows that in each of these cases the quadratic relation space of the surface quotient is $\langle x^2,y^2\rangle$. Indeed, $\Gam(\DynD_m,\{m-2\}) \cong \Bbbk\langle x,y\rangle/ (x^2,y^2,(x+y)^{m-2})$, the additional relations for the type~$\DynE_6$ cases begin in degree three, and the additional relation for the type~$\DynE_7$ case begins in degree four. Thus, the second radical layer of the surface quotient has dimension $2$, and $\dim_\Bbbk\rad^2\Acon/\rad^3\Acon$ is either $2$ or $3$.
\begin{itemize}
\item If this dimension is $2$, then the quadratic relation space of $\Acon$ is also $\langle x^2,y^2\rangle$. Hence, $\operatorname{gr}_{\rad}\Acon$ is a quotient of $\Bbbk\langle x,y\rangle/(x^2,y^2)$ and is tame. Since $\Acon$ degenerates to $\operatorname{gr}_{\rad}\Acon$, it is tame.

\item Suppose $\dim_\Bbbk\rad^2\Acon/\rad^3\Acon=3$. Then, the quadratic relation space of $\Acon$ is one-dimensional and is contained in $\langle x^2,y^2\rangle$. Its generator has the form $ax^2+by^2$ with $(a,b)\ne(0,0)$. If one of $a,b$ is zero, after exchanging $x$ and $y$ we may assume that the quadratic relation is $x^2=0$. In $\Acon/\rad^3\Acon$, quotienting further by $yx$ gives $\Bbbk\langle x,y\rangle/(x^2,y^3,yx,xy^2)$, the wild algebra in Proposition~\ref{prop:local-benchmark-alg}\textup{(4)}. If both $a$ and $b$ are nonzero, after rescaling $x$ and $y$ we may write the quadratic relation as $x^2-y^2=0$. In $\Acon/\rad^3\Acon$, quotienting further by $xy$ gives $\Bbbk\langle x,y\rangle/(x^2-y^2,xy)$, the wild algebra in Proposition~\ref{prop:local-benchmark-alg}\textup{(3)}. Hence $\Acon$ is wild in both cases.
\end{itemize}
This covers all one-curve cases not already shown to be wild in Step~2.
\end{proof}

\subsection{Two exceptional curves}
Suppose that $\Acon$ has precisely two simple modules, with corresponding primitive idempotents $e_1,e_2$. We first record the effect of the central quotient on the first few radical layers.

\begin{proposition}\label{prop:two-curve-radical-reduction}
Let $\Gam(\Delta,\{i,j\})\cong \Acon/(g)$ with $g\in Z(\Acon)\cap\rad \Acon$ be the surface quotient of $\Acon$ associated with a two-curve contraction. For $i\ne j$, the natural map $e_i(\rad \Acon/\rad^2\Acon)e_j\to e_i(\rad\Gam(\Delta,\Omega)/\rad^2\Gam(\Delta,\Omega))e_j$ is an isomorphism.  At each vertex, the kernel on the loop space has dimension at most one.

If the Gabriel quivers of $\Acon$ and $\Gam(\Delta,\Omega)$ agree, then $g\in\rad^2\Acon$.  In this case, all off-diagonal quadratic radical spaces are unchanged, while in each diagonal corner the kernel in $\rad^2\Acon/\rad^3\Acon$ has dimension at most one.

If, in addition, the natural map on the quadratic radical layers is an isomorphism, then $g\in\rad^3\Acon$, and all off-diagonal cubic radical spaces are unchanged.
\end{proposition}
\begin{proof}
Since $g$ is central, $g=e_1ge_1+e_2ge_2$. Moreover, $(g)\rad \Acon\subseteq\rad^2\Acon$. Thus, the kernel on the first radical layer is diagonal, and its component at vertex $i$ is spanned by the class of $e_i g e_i$. If the Gabriel quivers agree, then $g\in\rad^2\Acon$; repeating the argument modulo $\rad^3\Acon$ gives the quadratic statement. If the quadratic layers also agree, then $g\in\rad^3\Acon$, and the same argument modulo $\rad^4\Acon$ gives the cubic statement.
\end{proof}

\subsubsection*{Type~\(\DynA\)}
Let $\Omega=\{j_1<j_2\}$, with associated parameters $d_0,d_1,d_2$ as in Section~\ref{sec:A-stable}.  By Lemma~\ref{lem:A-gap-mutation}, iterated flops permute these three integers.  We may therefore choose a representative for which $d_1=\min\{d_0,d_1,d_2\}$.  Since the orbit survives in Step~2, we have $d_1=1$ or $d_1=2$. If $d_1=1$, Theorem~\ref{thm:uniformA} shows that the surface quotient has Gabriel quiver $\xymatrix@C=3em{1 \ar@<0.5ex>[r] & 2 \ar@<0.5ex>[l]}$; if $d_1=2$, its Gabriel quiver is obtained by adding one loop to each vertex in the case of $d_1=1$. The following proposition distinguishes the possible types of the contraction algebra in terms of its loops and quadratic relations. If there is exactly one loop at each vertex, we write the quiver as
\[
\mathsf{Q}_2:\quad \xymatrix@C=3em{
1 \ar@(ul,dl)[]_{\alpha} \ar@<0.5ex>[r]^{\mu}
& 2 \ar@(ur,dr)[]^{\beta} \ar@<0.5ex>[l]^{\nu}},
\]
where, for any path $p$ of length two, we denote by $[p]$ its class in $\rad^2\Acon/\rad^3\Acon$. As in the corresponding stable-endomorphism presentation, the off-diagonal quadratic classes satisfy $[\alpha\mu]=[\mu\beta]$ and $[\beta\nu]=[\nu\alpha]$ whenever they are nonzero.

\begin{proposition}\label{thm:step3-A2-quadratic}
Let $\Acon$ be a two-curve contraction algebra of type~$\DynA$
whose surface orbit survives Step~2. After a suitable sequence of
flops and a choice of arrow generators, exactly one of the following
occurs. The representative is chosen as in the proof, and may differ
from the representative chosen above.
\begin{enumerate}
\item The Gabriel quiver of $\Acon$ is
$\xymatrix@C=2em{1 \ar@<0.5ex>[r] & 2 \ar@<0.5ex>[l]}$,
and then $\Acon$ is representation-finite.

\item The Gabriel quiver of $\Acon$ is $\mathsf{Q}_2$ and
$[\alpha\mu]=[\mu\beta]=[\beta\nu]=[\nu\alpha]=0$.
Then $\Acon$ is tame.

\item The Gabriel quiver of $\Acon$ is $\mathsf{Q}_2$,
both loop--arrow directions are nonzero, and
\[
0\ne[\mu\nu]\in\Bbbk^\times[\alpha^2]
\qquad\text{and}\qquad
0\ne[\nu\mu]\in\Bbbk^\times[\beta^2].
\]
Then $\Acon$ is tame.

\item Either some vertex of the Gabriel quiver has at least two
loops, or the Gabriel quiver is $\mathsf{Q}_2$, both loop--arrow
directions are nonzero, and $[\mu\nu]=[\nu\mu]=0$.
Then $\Acon$ is wild.
\end{enumerate}
The tame algebras in \textup{(2)} and \textup{(3)} are
representation-infinite.
\end{proposition}
\begin{proof}
Put $S=\mathbb C[[x,y]]$ and $\mathfrak m=(x,y)$.
The type~$\DynA$ factor-block model writes the singularity as
$uv=f_0f_1f_2$, where the three nonunit blocks are reduced and
pairwise coprime. Iterated flops reorder these blocks; see
\cite[Sections~3.1--3.2]{ZhangGV26}.
For an ordering $(f,g,h)$, put $A=\Acon$ and $J=\rad A$.
The stable-endomorphism model gives
\[
\begin{aligned}
e_1Ae_1&\cong S/(f,gh),&
e_2Ae_2&\cong S/(fg,h),\\
e_1Ae_2&\cong S/(f,h),&
e_2Ae_1&\cong S/(f,h),
\end{aligned}
\qquad
\mu\nu=ge_1,\quad \nu\mu=ge_2,
\]
with central $S$-action.
Indeed, compute the stable endomorphisms of
$S/(f)\oplus S/(fg)$ over $S/(fgh)$: the intervertex maps are
generated by the natural quotient and multiplication by $g$,
and the ideals of maps through the free module are $(gh)$,
$(h)$ and $(h)$ in the corresponding coefficient spaces.
Kn\"orrer periodicity gives the threefold model, up to taking
the opposite algebra.
Killing the other vertex gives $S/(f,g)$ or $S/(g,h)$.
Consequently the loop multiplicities are
\[
\dim_{\mathbb C}\mathfrak m/(\mathfrak m^2+(f,g)),
\qquad
\dim_{\mathbb C}\mathfrak m/(\mathfrak m^2+(g,h)).
\]
Call a block smooth if it is not in $\mathfrak m^2$.
If at least two blocks are singular, put them consecutively.
Their common vertex has two loops, and the two intervertex arrows
are still present. The separated quiver contains a Kronecker
pair with an attached edge, so $A$ is wild. This is the first
alternative of \textup{(4)}.

Suppose next that all three blocks are smooth, and put
$d=\min_{i<j}\dim_{\mathbb C}S/(f_i,f_j)$.
If $d=1$, some block is transverse to both others.
Putting it in the middle gives the loopless quiver in \textup{(1)}.
The algebra is then a finite-dimensional cyclic Nakayama algebra,
and is representation-finite.
If $d\ge2$, choose coordinates and units so that
$f_i=x-\phi_i(y)$ with $\phi_i\in(y^2)$.
The minimum of the three orders
$\operatorname{ord}_y(\phi_i-\phi_j)$ occurs at least twice.
Choose $(f,g,h)$ so that
$I(f,g)=I(g,h)=d$, where $I(a,b)=\dim_{\mathbb C}S/(a,b)$.
The two diagonal rings are univariate of lengths
$d+I(f,h)$, and the off-diagonal coefficient ring has
length $I(f,h)\ge d$.
With $\alpha=ye_1$ and $\beta=ye_2$, centrality gives
$\alpha\mu=\mu\beta$ and $\beta\nu=\nu\alpha$, while
\[
\mu\nu=\alpha^dU(\alpha),\qquad
\nu\mu=\beta^dV(\beta),\qquad U(0)V(0)\ne0.
\]
Thus $d=2$ gives \textup{(3)}, whereas $d\ge3$ gives the
second alternative of \textup{(4)}.

Finally, suppose exactly one block $F$ is singular and the
other two blocks $B,C$ are smooth. Use the ordering $(B,F,C)$.
Both vertices have one loop, so the quiver is $\mathsf Q_2$.
If $B,C$ are transverse, then $S/(B,C)=\mathbb C$,
so all loop--arrow products vanish, giving \textup{(2)}.
Otherwise write $B=x$ and $C=x-\phi(y)$ up to units,
with $\operatorname{ord}\phi\ge2$.
The two backtracks are $F(0,\alpha)$ and
$F(\phi(\beta),\beta)$.
Their quadratic coefficients agree: both are the coefficient
of $y^2$ in $F(0,y)$.
If this coefficient is nonzero, both relations in
\textup{(3)} hold; otherwise both backtracks vanish modulo
$J^3$, giving the second alternative of \textup{(4)}.
In both cases the two loop--arrow directions survive, since
$\dim S/(B,C)\ge2$. This proves that the four cases exhaust
the possible block configurations after the indicated flops.

It remains to justify the representation types in the
non-loopless cases with quiver $\mathsf Q_2$.
In \textup{(2)}, $\operatorname{gr}_J A$ is special biserial:
the four vanishing mixed products leave each arrow at most
one possible nonzero successor and predecessor. In \textup{(3)}, write
$[\mu\nu]=c[\alpha^2]$ and $[\nu\mu]=d'[\beta^2]$
with $c,d'\ne0$.
After replacing $\nu$ by $(d')^{-1}\nu$, the associated graded
algebra is a quotient of $T_{21}$ in
\cite[Appendix~A, Table~T]{Wang22}, with parameters
$k_1=c/d'$ and $k_2=1$; choose the loop nilpotence exponents
larger than the Loewy length.
Thus $\operatorname{gr}_J A$ is non-wild in both cases,
and radical degeneration makes $A$ non-wild.
Its radical-square-zero quotient has separated quiver
$\widetilde{\DynA}_3$, so $A$ is representation-infinite
and hence tame.

In the second alternative of \textup{(4)}, the displayed
coefficient-ring model shows that the only quadratic
relations are
$[\alpha\mu]=[\mu\beta]$,
$[\beta\nu]=[\nu\alpha]$ and
$[\mu\nu]=[\nu\mu]=0$.
Indeed, in the all-smooth case the diagonal lengths are
at least six and the off-diagonal length is at least three;
in the one-singular-block case the diagonal lengths are
at least five and the off-diagonal length is at least two.
The loop squares and the mixed classes therefore survive.
Quotient $A/J^3$ further by
$[\beta\nu]=[\nu\alpha]$.
The result has relations
\[
\alpha\mu=\mu\beta,\qquad
\alpha^3=\beta^3=\mu\nu=\nu\mu=\beta\nu=\nu\alpha
=\mu\beta^2=\alpha^2\mu=0.
\]
This is $W_{33}$ in \cite[Table~W]{Han02},
so $A$ is wild. Finally, representation type is unchanged
under iterated flops by Proposition~\ref{prop:chamber-invariance}.
\end{proof}

\subsubsection*{Type~\(\DynD\): the representation-finite surface orbit}
We next consider the two-curve type~$\DynD$ orbit whose surface quotients are all representation-finite.

\begin{proposition}\label{thm:step3-D-A2}
Let $\Acon$ be a two-curve contraction algebra of  type~$\DynD$ in the orbit with representation-finite surface quotients.
\begin{enumerate}
\item The algebra $\Acon$ is representation-finite if and only if some iterated-flop model has loopless Gabriel quiver $\xymatrix@C=3em{1 \ar@<0.5ex>[r] & 2 \ar@<0.5ex>[l]}$.

\item Assume that $\Acon$ is not representation-finite and choose a two-point representative in this orbit. If its Gabriel quiver is
\[
\xymatrix@C=3em{1 \ar@(ul,dl)[]_{\alpha} \ar@<0.5ex>[r]^{\mu} & 2 \ar@<0.5ex>[l]^{\nu}},
\]
so that there is exactly one loop in total, then $\Acon$ is tame.

\item Suppose that the chosen representative has Gabriel quiver $\mathsf{Q}_2$. Then, $\Acon$ is tame if and only if $[\alpha\mu]=[\mu\beta]=[\beta\nu]=[\nu\alpha]=0$. If at least one of these four classes is nonzero, then $\Acon$ is wild.
\end{enumerate}
\end{proposition}
\begin{proof}
Any loopless representative is a representation-finite cyclic Nakayama algebra. The case of precisely one loop is tame and representation-infinite by the corresponding two-point bound-quiver classification; the relevant relations are recalled below.

Assume that the Gabriel quiver is $\mathsf{Q}_2$. By Proposition~\ref{prop:two-curve-radical-reduction}, the arrows between the two vertices coincide with those of the surface quotient. The remaining distinction is therefore detected by the quadratic classes involving loops.

If one loop--arrow orientation is nonzero, the rank-two bound-quiver presentation gives, after interchanging the vertices if necessary, a surjection onto one of $W_{31}$, $W_{32}$, or $W_{34}$; these algebras occur in the minimal wild list of \cite[Table~W]{Han02}, and their quiver relations are compatible with the diagonal quadratic classes of the orbit. Hence, $\Acon$ is wild.

If all four loop--arrow classes vanish, then $\operatorname{gr}_{\rad}\Acon$ is special biserial, since each arrow has at most one nonzero successor and predecessor. Thus, $\Acon$ is tame by degeneration. Its radical-square-zero quotient has separated quiver of Euclidean type $\widetilde{\DynA}_3$, so it is representation-infinite. Together with the one-loop case, this proves the converse in \textup{(1)}.
\end{proof}

\subsubsection*{Exceptional orbit \(\mathcal{E}_1\)}
We now consider the orbit $\mathcal{E}_1=(\DynE_6,\{1,6\})$ from Subsection~\ref{subsec:step3-remaining}. By Proposition~\ref{prop:E6-main}, $\Gam(\DynE_6,\{1,6\}) \cong \Bbbk\mathsf{Q}_2/ (\alpha^2,\beta^2,\mu\nu,\nu\mu, \alpha\mu-\mu\beta,\beta\nu-\nu\alpha)$. This algebra is tame.

\begin{proposition}\label{thm:terminal-X1-complete}
Let $\Acon$ be a two-curve contraction algebra in the orbit $\mathcal{E}_1$. Then, $\Acon$ is representation-infinite. If either vertex has at least two loops, then $\Acon$ is wild.

Suppose that the Gabriel quiver of $\Acon$ is $\mathsf{Q}_2$. Then, $\Acon$ is wild if and only if $[\alpha^2]\ne0,[\beta^2]\ne0, [\mu\nu]=[\nu\mu]=0$, or $[\alpha^2]=[\beta^2]=0, [\mu\nu]\ne0, [\nu\mu]\ne0$. In all other cases, the algebra $\Acon$ is tame.
\end{proposition}
\begin{proof}
Since the surface quotient is tame representation-infinite, $\Acon$ is representation-infinite. If one vertex has at least two loops, the separated quiver of $\Acon/\rad^2\Acon$ has a connected component which is neither Dynkin nor Euclidean. Hence $\Acon$ is wild by the radical-square-zero criterion.

Assume now that the Gabriel quiver of $\Acon$ is $\mathsf{Q}_2$. The Gabriel quiver is then the same as that of the surface quotient, so Proposition~\ref{prop:two-curve-radical-reduction} gives $g\in\rad^2\Acon$. Since $\alpha^2$ and $\mu\nu$ both vanish in the corresponding diagonal corner of the surface algebra, while $\beta^2$ and $\nu\mu$ vanish in the other,
\[
\dim\langle[\alpha^2],[\mu\nu]\rangle\le1,
\qquad
\dim\langle[\beta^2],[\nu\mu]\rangle\le1.
\]
The off-diagonal quadratic layer is unchanged. In particular, the surface relations $\alpha\mu=\mu\beta$ and $\beta\nu=\nu\alpha$ lift to the same nonzero classes in $\rad^2\Acon/\rad^3\Acon$.

We apply the minimal wild quotient criterion of \cite{Han02}. For the first displayed configuration, the two-point presentation admits a surjection onto $W_{33}$; for the second, it admits a surjection onto $W_{32}$. Consequently, none of the remaining configurations has a minimal wild two-point quotient. Then, these cases are not wild, and hence are tame since they are already representation-infinite.
\end{proof}

\subsubsection*{The remaining tame surface orbits}
We next consider the remaining tame two-curve surface cases of
type~$\DynD$ and the exceptional orbits
$\mathcal E_2,\mathcal E_3$.
Put $A=\Acon$, $J=\rad A$, and
$B=A/(g)=e_\Omega\Pi_\Delta e_\Omega$.
The one-curve subset contraction algebra is
$A_i=A/A(1-e_i)A$ by \cite{DW19}, and
\[
\frac{e_iJe_i}{e_iJ^2e_i}
\cong\frac{\rad A_i}{\rad^2 A_i}.
\]
Its surface quotient is
\[
A_i/(\bar g)\cong B/B(1-e_i)B,
\]
which need not coincide with $e_i\Pi_\Delta e_i$.

\begin{proposition}\label{prop:two-curve-singleton-obstruction}
Let $A$ be a two-curve contraction algebra with surface quotient
$B=A/(g)$. Suppose that $B/B(1-e_i)B$ has at least two loops
for some $i$, and that the Gabriel quiver of $B$ contains
$\xymatrix@C=2em{1\ar@<0.5ex>[r]&2\ar@<0.5ex>[l]}$.
Then $A$ is wild.
\end{proposition}
\begin{proof}
The displayed loop-space isomorphism and the quotient
$A_i\twoheadrightarrow B/B(1-e_i)B$ give at least two loops
at vertex $i$ of $A$.
The two opposite arrows also lift from $B$.
Thus the separated quiver of $A/J^2$ contains a Kronecker pair
with an attached edge, so $A/J^2$, and hence $A$, is wild.
\end{proof}

\begin{proposition}\label{thm:two-curve-remaining-wild}
For the remaining tame two-curve surface orbits, choose a
representative from the following table:
\[
\begin{array}{c|c|c}
\text{orbit}&\text{representative}&v\\ \hline
\DynD&\{1,2\}&2\\
\DynD&\{m-2,m-1\}&m-2\\
\mathcal E_2&(\DynE_6,\{1,3\})&3\\
\mathcal E_3&(\DynE_7,\{6,7\})&6.
\end{array}
\]
Here $v$ is the unique vertex carrying a loop in the surface
quiver. Let $\alpha\in e_vJe_v$ lift that loop.
Then $A$ is representation-infinite, and
\[
A\text{ is tame}
\quad\Longleftrightarrow\quad
\dim_{\mathbb C}e_v(J/J^2)e_v=1
\quad\text{and}\quad
[\alpha^2]=0\text{ in }J^2/J^3.
\]
In all other cases $A$ is wild.
A new loop at the other vertex is allowed in the tame case.
\end{proposition}
\begin{proof}
The two $\DynD$ rows cover the surviving tame pairs up to
mutation and diagram automorphism.
For example, mutation at $2$ sends $\{2,m-1\}$ to
$\{m-2,m-1\}$: the component obtained by deleting $m-1$
is an $\DynA_{m-1}$ chain, whose reversal sends $2$ to $m-2$.
The other short arm is obtained by diagram automorphism.
The exceptional representatives lie in the orbits of
Proposition~\ref{prop:exceptional-mutation-orbits}.
When $m=4$, the two $\DynD$ rows agree up to diagram automorphism.

Relabel $v$ as $1$ and the other vertex as $2$.
In every row the surface quiver has generators
$\alpha:1\to1$, $a:1\to2$, and $b:2\to1$.
Put $z=ab$.
The path reductions in Step~1 give the exact relation $ba=0$
and the quadratic relation space
$\langle\alpha^2,ba\rangle$ in this two-vertex quiver.
The local corner $e_1Be_1$ is generated by $\alpha,z$,
and its quadratic relation space is
\[
\begin{cases}
\langle\alpha^2,z^2,\alpha z+z\alpha\rangle,
&\Omega=\{1,2\}\text{ in }\DynD_m,
  \text{ or }\Omega=\{m-2,m-1\}\text{ with }m=4,\\
\langle\alpha^2,z^2\rangle,&\text{otherwise}.
\end{cases}
\]
Indeed, the relevant local presentations from Step~1 are,
up to rescaling generators,
\[
\begin{array}{c|l}
\DynD_m,\{1,2\}
 &\alpha^2=z^2=\alpha z+z\alpha=0\\
\DynD_m,\{m-2,m-1\}
 &\alpha^2=z^2=(\alpha+z)^{m-2}=0\\
\DynE_6,\{1,3\}
 &z^2=0,\quad\alpha^2=z\alpha z,\quad\alpha z\alpha=0\\
\DynE_7,\{6,7\}
 &\alpha^2=z^2=(\alpha z)^2+(z\alpha)^2=0.
\end{array}
\]
For the last row take
$a=a_6$, $b=b_6$, and
$\alpha=b_5b_4b_3a_3a_4a_5$.
Then $z=a_6b_6$, and the local relations are precisely
those at vertex $6$ in Proposition~\ref{prop:E7-main}.
The endpoint relation at $7$ gives $ba=0$.
The paths $ab,\alpha a,b\alpha$ have nonzero, independent
classes at their respective endpoints in the quadratic
radical layer, whereas $\alpha^2$ vanishes there.
This also verifies the stated two-vertex quadratic relation space.

The central quotient leaves the intervertex arrows unchanged
and adds at most one loop at each vertex.
If vertex $1$ has two loops, the separated-quiver criterion
gives wildness. Hence assume that $\alpha$ is its only loop.
Write $g_i=e_ig$; then $g_1\in J^2$.

Suppose $\alpha^2\notin J^3$.
Set $C=e_1Ae_1$ and $K=\rad C$.
We have $C/(g_1)\cong e_1Be_1$.
If $\dim K/K^2\ge3$, then $C$ is wild.
Otherwise $\alpha,z$ form a basis of $K/K^2$, since their
images generate the two-dimensional cotangent space of $e_1Be_1$.
Every quadratic relation of $C$ therefore belongs to the
quadratic relation space of $e_1Be_1$ displayed above.
Its $\alpha^2$ coefficient must be zero:
$z\in J^2$, so $\alpha z,z\alpha,z^2$ and $K^3$ lie in $J^3$,
whereas $\alpha^2$ does not.
Thus the quadratic relation space of $C$ is contained in
$\langle z^2,\alpha z+z\alpha\rangle$.
Consequently $C/K^3$ surjects onto
\[
\mathbb C\langle u,w\rangle/
\bigl(w^2,uw+wu,(u,w)^3\bigr),
\qquad \alpha\mapsto u,\quad z\mapsto w.
\]
This is the wild local algebra in
Proposition~\ref{prop:local-benchmark-alg}\textup{(7)}
with parameter $-1$.
Hence $A$ is wild by the idempotent truncation criterion.

Now assume $\alpha^2\in J^3$.
Since $[ab]$ maps to a nonzero class in the surface quadratic
layer, the diagonal quadratic kernel at vertex $1$ is zero.
Thus $g_1\in J^3$.
If vertex $2$ has no loop, then $g_2\in J^2$.
The exact surface relation $ba=0$ gives $ba=g_2h$ for some
$h\in e_2Ae_2$.
Centrality yields
\[
aba=g_1ah\in J^4,\qquad
bab=hb g_1\in J^4.
\]
Therefore $\operatorname{gr}_J A$ satisfies
$\alpha^2=aba=bab=0$.
With a sufficiently large nilpotence exponent on $b\alpha a$,
it is a quotient of the tame algebra $T_9$ in
\cite[Appendix~A, Table~T]{Wang22}.
Radical degeneration shows that $A$ is non-wild.

If vertex $2$ has a new loop, choose $\beta=g_2$ as its generator.
Give $\alpha,a,b$ weight one and $\beta$ weight two, and let
$F^nA$ be spanned by images of paths of weight at least $n$.
This finite multiplicative filtration satisfies
$F^nA\subseteq J^{\lceil n/2\rceil}$.
Again $ba=\beta h$, but now $h\in\rad(e_2Ae_2)$;
otherwise $\beta\in J^2$, contrary to its being a new loop.
Every path in this local radical has weight at least two.
Hence
\[
ba\in F^4A,\qquad
a\beta=g_1a\in F^4A,\qquad
\beta b=bg_1\in F^4A,\qquad
\alpha^2\in F^3A.
\]
The weighted associated graded algebra consequently satisfies
$ba=a\beta=\beta b=\alpha^2=0$.
Each arrow has at most one possible nonzero successor and
predecessor, so this algebra is special biserial.
A basis adapted to the finite filtration gives a flat Rees
degeneration of $A$ to this algebra.
Thus $A$ is non-wild in this case as well.

Finally, $B$ is representation-infinite by Step~1, so the
non-wild cases above are tame representation-infinite.
If $\alpha$ is the unique loop at vertex $v$, replacing its
lift changes it by an element of $J^2$; its square modulo
$J^3$ is therefore unchanged.
Rescaling the surface-loop generator multiplies this class
by a nonzero square.
The criterion is consequently independent of these choices.
Representation type is preserved under iterated flops by
Proposition~\ref{prop:chamber-invariance}.
\end{proof}

\subsection{Three exceptional curves}
Suppose that $\Acon$ has exactly three simple modules. By the classification in Step~2, it suffices to consider types~$\DynA$ and~$\DynD$. We again determine the Gabriel quiver from the surface quotient and the one-curve subset algebras. We then apply the separated-quiver criterion to detect wildness and the degeneration method of Section~\ref{sec:prelim} to establish tameness.

\subsubsection*{Type~\(\DynA\)}
For type~$\DynA$ with three exceptional curves, the mutation description may be expressed in terms of the four positions of the gap sequence $(d_0,d_1,d_2,d_3)$ from Step~2. Each ordering determines a representative whose three exceptional curves correspond to the consecutive pairs. Thus, the six unordered pairs index the possible one-curve subset factors occurring under iterated mutation. We call such a pair \emph{nontrivial} if the corresponding one-curve subset contraction algebra is not isomorphic to $\Bbbk$. The correspondence between these pairs under mutation follows from \cite[Section~3.1 and Theorem~3.13]{ZhangGV26}. We use it only to identify the subset contraction algebras across the different representatives.

\begin{proposition}\label{thm:terminal-cA3-trichotomy}
Let $\Acon$ be a three-curve contraction algebra of type~$\DynA$ whose surface orbit survives Step~2. Then $\Acon$ is never representation-finite. Moreover, $\Acon$ is tame representation-infinite if and only if the nontrivial pairs above are pairwise disjoint. Otherwise $\Acon$ is wild.
\end{proposition}
\begin{proof}
Suppose two nontrivial pairs have a position in common. We may order the four positions so that these pairs occur consecutively in the same three-curve model. The corresponding one-curve subset factors are nontrivial local algebras, so their radicals modulo their squares are nonzero. The subset-factor argument used in the two-curve case therefore gives a loop at each of the two adjacent vertices. The surface quotient supplies the doubled line
\[
\xymatrix@C=3.2em{1 \ar@<0.5ex>[r] & 2 \ar@<0.5ex>[l] \ar@<0.5ex>[r] & 3 \ar@<0.5ex>[l]},
\]
and the two loops at adjacent vertices make the separated quiver of $\Acon/\rad^2\Acon$ contain a four-cycle with an attached edge. Hence $\Acon$ is wild by the radical-square-zero criterion. Representation type is unchanged by iterated mutation by Proposition~\ref{prop:chamber-invariance}.

Now suppose that the nontrivial pairs are pairwise disjoint. Thus there are at most two nontrivial pairs among the four positions. Choose an ordering in which none of these pairs is consecutive. In that model, all three one-curve subset factors are $\Bbbk$, so its Gabriel quiver has no loops. For this representative, the full bound quiver presentation is given by \cite[Corollary~3.5 and Section~3.3]{ZhangGV26} and \cite[Section~6]{Zhang26}. In particular, the relations in that presentation yield a tame representation-infinite algebra. Thus this model, and hence $\Acon$ by Proposition~\ref{prop:chamber-invariance}, is tame representation-infinite. In particular, no case in the theorem is representation-finite.
\end{proof}

\subsubsection*{Type~\(\DynD\)}
In type~$\DynD$, Step~2 leaves only one non-wild mutation orbit with three exceptional curves, represented by $\Omega=\{1,m-1,m\}$. The remaining tame triples belong to this orbit. Its surface Gabriel quiver is
\[
\vcenter{\xymatrix@C=3em@R=3.2em{& m-1 \ar@<0.5ex>[dr]^{c} \ar@<0.5ex>[dl]^{q} \\
1 \ar@<0.5ex>[ur]^{p} \ar@<0.5ex>[rr]^{r} && m \ar@<0.5ex>[ul]^{d} \ar@<0.5ex>[ll]^{s}}}\;,
\]
and its surface algebra is tame representation-infinite by Proposition~\ref{prop:D-rank3}.

\begin{proposition}\label{thm:terminal-D-rank3}
Let $\Acon$ be a three-curve contraction algebra of type~$\DynD$ whose surface orbit survives Step~2. In the representative $\{1,m-1,m\}$, the algebra $\Acon$ is wild if and only if its Gabriel quiver has a loop. If that quiver is loopless, then $\Acon$ is tame.
\end{proposition}
\begin{proof}
If the Gabriel quiver has a loop, then it contains a full subquiver of the form
\[
\xymatrix@C=3.2em@R=3em{& m-1 \ar@<0.5ex>[dr] \ar@<0.5ex>[dl] \\
1 \ar@(ul,dl)[] \ar@<0.5ex>[ur] \ar@<0.5ex>[rr] && m \ar@<0.5ex>[ul] \ar@<0.5ex>[ll]},
\]
so the separated quiver of $\Acon/\rad^2\Acon$ contains a six-cycle with an additional edge. Hence $\Acon$ is wild by the radical-square-zero criterion.

Assume that the Gabriel quiver is loopless and write $\Gam(\DynD_m,\{1,m-1,m\})\cong \Acon/(g)$. The path reductions in Step~1 give $pc=rd=cs=dq=0$. For the other two products passing through three distinct vertices, put $x=a_{m-2}b_{m-2}$ and $y=a_{m-1}b_{m-1}$ at the branching vertex. The short-arm relations give $x^2=y^2=0$, while reduction along the long arm gives, up to the signs determined by the path convention, $qr=b_{m-2}(x+y)^{m-3}a_{m-1}$ and $sp=b_{m-1}(x+y)^{m-3}a_{m-2}$. For even $m$ these products vanish. For odd $m$, set $k=(m-3)/2$. Only alternating words in $x,y$ survive, and moving them to the contracted corner gives $qr\in\Bbbk\,c(dc)^k$ and $sp\in\Bbbk\,d(cd)^k$. Since these are paths of length at least three, every two-arrow path through three distinct vertices vanishes in $\rad^2\Gam(\DynD_m,\{1,m-1,m\})/ \rad^3\Gam(\DynD_m,\{1,m-1,m\})$.

Both $\Acon$ and its surface quotient are loopless, so $g\in\rad^2\Acon$. Centrality makes the kernel of the map on quadratic radical layers diagonal: multiplying $g$ by a radical element only contributes in $\rad^3\Acon$. Thus all off-diagonal quadratic classes are unchanged on passing from $\Acon$ to $\Acon/(g)$. The same two-arrow paths therefore vanish in $\rad^2\Acon/\rad^3\Acon$. It follows that, in $\operatorname{gr}_{\rad}\Acon$, each arrow has at most one possible nonzero successor and predecessor, namely its reverse arrow. Since the Gabriel quiver is the doubled triangle, $\operatorname{gr}_{\rad}\Acon$ is special biserial and therefore tame. The radical filtration gives the degeneration $\Acon\rightsquigarrow\operatorname{gr}_{\rad}\Acon$, so the degeneration criterion in Section~\ref{sec:prelim} makes $\Acon$ tame. Finally, its surface quotient is representation-infinite by Proposition~\ref{prop:D-rank3}; therefore, $\Acon$ is tame representation-infinite.
\end{proof}

\subsection{Four or more exceptional curves}

\begin{proposition}\label{prop:step3-higher-rank}
Every contraction algebra with at least four exceptional curves whose surface data survive Step~2 is wild.
\end{proposition}
\begin{proof}
We first consider type~$\DynA$, and let $r\ge4$ be the number of exceptional curves. As above, the mutation correspondence identifies one-curve subset contraction algebras across iterated-flop representatives.

Suppose that some such factor is not isomorphic to $\Bbbk$. Choose an iterated-mutation model in which this factor occurs at the second simple vertex. The subset-factor argument gives a loop there, while the surface quotient supplies a doubled line on four consecutive vertices,
\[
\xymatrix@C=2.8em{1 \ar@<0.5ex>[r] & 2 \ar@(ul,ur)[]^{ } \ar@<0.5ex>[l] \ar@<0.5ex>[r] & 3 \ar@<0.5ex>[l] \ar@<0.5ex>[r] & 4 \ar@<0.5ex>[l]}.
\]
Retain these four consecutive vertices and pass to the radical-square-zero quotient. The separated quiver contains a subquiver of Euclidean type $\widetilde{\DynD}_5$, arising from the loop at the second vertex and the opposite arrows to its two neighbours. The fourth vertex supplies additional edges, so the resulting connected component is neither Dynkin nor Euclidean. Hence the radical-square-zero quotient is wild.

Suppose now that every one-curve subset factor occurring under iteration is $\Bbbk$. The bound-quiver presentation for the corresponding loopless representative is taken from the same bound-quiver description invoked in the three-curve case. Let $Q$ be its doubled $\DynA_r$ Gabriel quiver and $J$ the arrow ideal of $\Bbbk Q$. In this bound-quiver presentation, there are no quadratic defining relations. Consequently, $\Acon/\rad^3\Acon\cong\Bbbk Q/J^3$. We apply the graded Galois covering construction from Subsection~\ref{subsec:galois-covering} to the radical-cube-zero quotient supported on four consecutive vertices. The full convex subcategory on the vertices $1_0,3_0,2_1,4_1,1_2,3_2$ has no nonzero defining relations. Its underlying graph is a four-cycle with two attached leaves, so it is hereditary wild. Proposition~\ref{prop:covering-wild} makes the four-vertex quotient wild. If $r>4$, killing the other vertices yields this same quotient. Thus every surviving type~$\DynA$ case with $r\ge4$ is wild.

We next consider type~$\DynD$. Step~2 already gives a wild surface quotient whenever there are at least five exceptional curves. With four curves, the only non-wild surface candidate is the full $\DynD_4$ marking. If an iterated-mutation model has a loop at the central vertex, then its Gabriel quiver contains a full subquiver of the form
\[
\xymatrix@C=3.2em@R=2.5em{
& 2 \ar@(ul,ur)[]^{\ } \ar[dl] \ar[d] \ar[dr] \\
1 & 3 & 4
}
\]
with the opposite arrows suppressed from the notation. Its radical-square-zero quotient is wild: indeed, the separated quiver of the doubled $\DynD_4$ quiver consists of two three-armed stars, and the central loop joins their centres. It contains a proper Euclidean subquiver of type $\widetilde{\DynD}_4$, with further vertices attached, and so is neither Dynkin nor Euclidean. By the mutation transport of the one-curve subset factor, any nontrivial factor in this orbit can be placed at the central vertex; hence this covers the case in which a loop occurs in a suitable model.

It remains to treat the loopless model. The corresponding bound-quiver presentation has no quadratic relations, by the same reduction as in type~$\DynA$. Thus, with $Q$ the doubled $\DynD_4$ quiver and $J$ the arrow ideal of $\Bbbk Q$, $\Acon/\rad^3\Acon\cong\Bbbk Q/J^3$. In the graded covering from Subsection~\ref{subsec:galois-covering}, consider the central vertex in degrees $0$ and $2$ and the three end vertices in degree $1$. They form a full convex subcategory with quiver
\[
\xymatrix@C=2em@R=2em{
& 2_0 \ar[dl] \ar[d] \ar[dr] & \\
1_1 \ar[dr] & 3_1 \ar[d] & 4_1 \ar[dl] \\
& 2_2 &
}
\]
and no defining relations. Since its underlying graph is neither Dynkin nor Euclidean, this subcategory is hereditary wild. Then, Proposition~\ref{prop:covering-wild} gives wildness of $\Acon/\rad^3\Acon$, and hence of $\Acon$.

In the exceptional case, Step~2 already removes every case with at least three exceptional curves. This completes the proof.
\end{proof}

\appendix

\section{A tame case}\label{app:four-vertex-tame}
This appendix proves the tameness assertion in Proposition~\ref{prop:111H1tame}. The argument combines node separation with the hammock and vector-space categories associated with one-point extensions and coextensions. Let $Q$ be the quiver
\[
\xymatrix@C=3.2em{1 \ar@<0.5ex>[r]^{a} & 2 \ar@<0.5ex>[l]^{b} \ar@<0.5ex>[r]^{c} & 3 \ar@<0.5ex>[l]^{d} \ar@<0.5ex>[r]^{e} & 4 \ar@<0.5ex>[l]^{f}},
\]
and put $S=\Bbbk Q/\langle ab,\ ba,\ dc-ef,\ fe,\ acd,\ cdb\rangle$.
The algebra $S$ has a node at vertex $1$. Its separation yields an algebra $\widetilde S$ obtained by adjoining a source and then a sink to $B\simeq\operatorname{Aus}(\Bbbk[x]/(x^3))$. 
The first attachment is representation-finite by Kleiner's criterion \cite{Kleiner1972}, whereas the second is tame by Bekkert's criterion \cite[Theorem~1]{BekkertPolynomialGrowth}. The latter hammock also yields a one-parameter family of indecomposable modules.

\subsection{Reduction to the fixed algebra}

\begin{proposition}\label{app:prop:reduction-S}
Let $A_p= \Bbbk Q/\langle
ab,\ ba-(cd)^p,\ dc-ef,\ fe,\ acd,\ cdb
\rangle$ with $p\ge2$. Then, $A_p\simeq S$ for $p\ge3$, and $A_2$ degenerates flatly to $S$. Moreover, $A_p$ is representation-infinite for every $p\ge2$.
\end{proposition}
\begin{proof}
From $dc=ef$ and $fe=0$ we obtain $(dc)^2=(ef)^2=e(fe)f=0$ and $(cd)^3=c(dc)^2d=0$. Thus $(cd)^p=0$ for $p\ge3$, and $A_p\simeq S$. For $p=2$ and $\lambda\in\Bbbk$, set
\[
A_{2,\lambda}:=
\Bbbk Q/
\langle
ab,\ ba-\lambda(cd)^2,\ dc-ef,\ fe,\ acd,\ cdb
\rangle.
\]
For $\lambda\neq0$, rescaling $b$ yields $A_{2,\lambda}\simeq A_2$, while $A_{2,0}=S$. To prove flatness, work first over $\Bbbk[\lambda]$. The relations admit the reductions
\[
ba\longmapsto\lambda cdcd,\qquad ef\longmapsto dc,
\qquad ab,\ fe,\ acd,\ cdb,\ fdc,\ dce,\ dcdc\longmapsto0.
\]
The last three zero relations follow from the original ones: $fdc=fef=0$, $dce=efe=0$, and $dcdc=(ef)^2=0$. Give $a,b$ weight three and $c,d,e,f$ weight one, and order $ef$ above $dc$ at equal weight. The reductions terminate and all overlap ambiguities resolve. Their irreducible paths are
\[
\begin{gathered}
e_1,e_2,e_3,e_4,\quad a,b,c,d,e,f,\quad
ac,cd,ce,db,dc,fd,\\
ace,cdc,dcd,fdb,\quad cdcd .
\end{gathered}
\]
They form a $\Bbbk[\lambda]$-basis, of rank $21$, by the diamond lemma. Thus the family is free, hence flat, and gives $A_2\rightsquigarrow S$.

Deleting the third vertex yields an idempotent truncation with parameter sequence $(1,p,2,1)$ by Proposition~\ref{prop:coarsen}. It is representation-infinite by Proposition~\ref{prop:r3table}, and therefore so is $A_p$.
\end{proof}

\subsection{Node separation and the representation-finite core}

\begin{proposition}\label{app:prop:node-separation}
Vertex $1$ is a node of $S$. Its separation gives
$\widetilde S=\Bbbk\widetilde Q/\langle ab,\ dc-ef,\ fe,\ acd,\ cdb\rangle$, where
\[
\widetilde Q:\quad\vcenter{\xymatrix@C=3em@R=1em{
1^+ \ar[dr]^{a} &&& \\
& 2 \ar[dl]_{b} \ar@<0.5ex>[r]^{c} & 3 \ar@<0.5ex>[l]^{d} \ar@<0.5ex>[r]^{e} & 4 \ar@<0.5ex>[l]^{f} \\
1^- &&&
}}.
\]
Node separation identifies the isomorphism classes of indecomposable $S$-modules with those of indecomposable $\widetilde S$-modules other than the simple at $1^-$. In particular, $S$ and $\widetilde S$ have the same representation type. If
\[
B=\Bbbk\left(\xymatrix@C=3em{2 \ar@<0.5ex>[r]^{c} & 3 \ar@<0.5ex>[l]^{d} \ar@<0.5ex>[r]^{e} & 4 \ar@<0.5ex>[l]^{f}}\right)/\langle dc-ef,\ fe\rangle,
\]
then $\widetilde S$ is obtained from $B$ by successively adjoining the source $1^+$ and the sink $1^-$.
\end{proposition}
\begin{proof}
The only path of length two passing through vertex $1$ is $ba$, which vanishes in $S$; hence vertex $1$ is a node. Separating its incoming and outgoing arrows produces a sink $1^-$ and a source $1^+$. The relation $ba=0$ then becomes redundant, while $ab=0$ remains, giving the displayed presentation of $\widetilde S$. The claimed correspondence follows from the node-separation theorem. Removing $1^+$ and $1^-$ leaves the full convex subalgebra $B$, which gives the stated sequence of attachments.
\end{proof}

\subsection{The first attachment}
\begin{lemma}\label{app:lem:AR-B}
Up to relabelling the vertices and the opposite-algebra convention, $B\simeq\operatorname{Aus}(\Bbbk[x]/(x^3))$. Thus $B$ is representation-finite. Its Auslander--Reiten quiver has $21$ vertices $X_0,\ldots,X_{20}$, with dimension vectors, ordered by the vertices $(2,3,4)$, as follows:
\[
\begin{array}{c|ccccccc}
i&0&1&2&3&4&5&6\\ \hline
\underline{\dim}X_i
 &(0,0,1)&(0,1,0)&(1,0,0)&(0,1,1)&(0,1,1)&(1,1,0)&(1,1,0)
\end{array}
\]
\[
\begin{array}{c|ccccccc}
i&7&8&9&10&11&12&13\\ \hline
\underline{\dim}X_i
 &(1,1,1)&(1,1,1)&(1,1,1)&(1,1,1)&(2,1,0)&(1,2,1)&(2,1,1)
\end{array}
\]
and
\[
\begin{array}{c|ccccccc}
i&14&15&16&17&18&19&20\\ \hline
\underline{\dim}X_i
 &(2,1,1)&(2,2,1)&(2,2,1)&(2,2,1)&(2,2,1)&(2,2,2)&(3,2,1).
\end{array}
\]
Modules with equal dimension vectors are distinguished by their positions in the Auslander--Reiten quiver. Its arrows are
\[
\begin{aligned}
&X_0\to X_8,\quad
X_1\to X_{16},\quad
X_2\to X_8,\quad
X_3\to X_{15},\quad
X_4\to X_9,\quad
X_5\to X_9,\\
&X_6\to X_{10},\ X_6\to X_{15},\quad
X_7\to X_5,\quad
X_8\to X_3,\ X_8\to X_6,\ X_8\to X_{13},\\
&X_9\to X_0,\ X_9\to X_2,\ X_9\to X_{12},\quad
X_{10}\to X_{19},\quad
X_{11}\to X_{16},\\
&X_{12}\to X_8,\ X_{12}\to X_{17},\quad
X_{13}\to X_{15},\quad
X_{14}\to X_9,\ X_{14}\to X_{18},\\
&X_{15}\to X_1,\ X_{15}\to X_{11},\ X_{15}\to X_{19},\\
&X_{16}\to X_4,\ X_{16}\to X_5,\ X_{16}\to X_{14},\quad
X_{17}\to X_{13},\\
&X_{18}\to X_{12},\ X_{18}\to X_{20},\quad
X_{19}\to X_7,\ X_{19}\to X_{16},\quad
X_{20}\to X_{17}.
\end{aligned}
\tag{A.1}
\]
The indecomposable projectives are $X_{10},X_{18},X_{20}$.
\end{lemma}

\begin{proof}
The Auslander algebra of $\Bbbk[x]/(x^3)$ has the doubled three-vertex line as its quiver, with an endpoint zero relation and a middle mesh relation. Its presentation agrees with that of $B$ after relabelling and, if necessary, passage to the opposite algebra.

The knitting procedure applied to the three indecomposable projectives yields the $21$ modules and the arrows in \textup{(A.1)}. Closure of the meshes under Auslander--Reiten translation establishes completeness.
\end{proof}

\begin{proposition}\label{app:prop:first-hammock}
Let $M=e_2B/(cd)B$ and let $C=B[M]$ be the one-point extension obtained by adjoining the source $1^+$. Explicitly,
\[
C=\Bbbk\left(\xymatrix@C=3em{1^+ \ar[r]^{a} & 2 \ar@<0.5ex>[r]^{c} & 3 \ar@<0.5ex>[l]^{d} \ar@<0.5ex>[r]^{e} & 4 \ar@<0.5ex>[l]^{f}}\right)/\langle dc-ef,\ fe,\ acd\rangle.
\]
Then $C$ is representation-finite.
\end{proposition}

\begin{proof}
The module $M$ is generated at vertex $2$ subject to $cd=0$, so $\Hom_B(M,X)\cong\ker\!\bigl(X(cd):X_2\to X_2\bigr)$.
With the numbering of Lemma~\ref{app:lem:AR-B}, the hammock function $h_M(X_i)=\dim_{\Bbbk}\Hom_B(M,X_i)$ vanishes for $i\in\{0,1,3,4\}$ and equals $1$ for $i\in\{2,5,6,7,8,9,10,11,12,13,14,15,16,17,18,19,20\}$. In particular, every nonzero hammock space is one-dimensional.

Let $P_M$ denote the hammock poset. The nonzero maps induced by $\Hom_B(M,-)$ on the arrows in \textup{(A.1)} yield the following cover relations after transitive reduction:
\[
\begin{aligned}
&X_7\prec X_5,\quad X_5\prec X_9,\quad
X_9\prec X_2,\quad X_9\prec X_{12},\\
&X_2\prec X_8,\quad
X_{12}\prec X_8,\quad X_{12}\prec X_{17},\\
&X_8\prec X_6,\quad X_8\prec X_{13},\quad
X_{17}\prec X_{13},\\
&X_6\prec X_{10},\quad X_6\prec X_{15},\quad
X_{13}\prec X_{15},\\
&X_{10}\prec X_{19},\quad
X_{15}\prec X_{11},\quad X_{15}\prec X_{19},\\
&X_{11}\prec X_{16},\quad X_{19}\prec X_{16},\quad
X_{16}\prec X_{14},\quad X_{14}\prec X_{18},\quad
X_{18}\prec X_{20}.
\end{aligned}
\tag{A.2}
\]
The poset $P_M$ is the union of the chains
\[
X_7<X_5<X_9<X_2<X_8<X_6<X_{10}<X_{19}<X_{16}<X_{14}<X_{18}<X_{20}
\]
and $X_{12}<X_{17}<X_{13}<X_{15}<X_{11}$. Hence $\operatorname{width}(P_M)=2$. Since every critical poset in Kleiner's criterion has width at least $3$, the poset $P_M$ is representation-finite. The one-point-extension/subspace-category correspondence then implies that $C=B[M]$ is representation-finite.
\end{proof}

\subsection{The second attachment}

\begin{proposition}\label{app:prop:second-attachment}
The algebra $\widetilde S$ is the one-point coextension of $C$ by the left $C$-module $N:=Cb\simeq Ce_2/\bigl(Ca+C(cd)\bigr)$. Equivalently, $\widetilde S^{\op}\simeq C^{\op}[N^{\op}]$.
\end{proposition}

\begin{proof}
Put $v=1^-$ and $e=e_{1^+}+e_2+e_3+e_4$. Then $e\widetilde S e=C$ and $N=e\widetilde S e_v=Cb$. For the map $\pi:Ce_2\to Cb$, $x\mapsto xb$, the relations $ab=cdb=0$ give $Ca+C(cd)\subseteq\ker\pi$.

A basis of $Ce_2$ is $e_2,\ a,\ d,\ cd,\ fd,\ efd,\ cefd$. Here $efd=dcd$ and $cefd=(cd)^2$. Right multiplication by $b$ sends $e_2,d,fd$ to $b,db,fdb$ and annihilates $a,cd,efd,cefd$. Consequently, $\ker\pi=\operatorname{span}_{\Bbbk}\{a,cd,efd,cefd\}=Ca+C(cd)$. This identifies $N$, and the triangular matrix description of a one-point coextension yields $\widetilde S^{\op}\simeq C^{\op}[N^{\op}]$.
\end{proof}

\subsection{The Auslander--Reiten quiver of \(C\)}

\begin{proposition}\label{app:prop:AR-C}
The algebra $C$ has exactly $49$ isomorphism classes of indecomposable modules.

Let $\mathcal A(P_M)$ denote the antichains of $P_M$, indexed by the subscripts of the $X_i$. Explicitly,
\[
\begin{aligned}
&\varnothing,\ \{2\},\ \{5\},\ \{6\},\ \{7\},\ \{8\},\ \{9\},
\ \{10\},\ \{11\},\ \{12\},\ \{13\},\ \{14\},\ \{15\},\ \{16\},\\
&\{17\},\ \{18\},\ \{19\},\ \{20\},\
\{2,12\},\ \{2,17\},\ \{6,13\},\ \{6,17\},\
\{8,17\},\ \{10,11\},\\
&\{10,13\},\ \{10,15\},\ \{10,17\},\ \{11,19\}.
\end{aligned}
\tag{A.3}
\]
For $A\in\mathcal A(P_M)$ put $Y_A=\bigl(\Bbbk,\bigoplus_{i\in A}X_i,\phi_A\bigr)$, where each nonzero component $\phi_i\in\Hom_B(M,X_i)$ is fixed, and set $Y_\varnothing=(\Bbbk,0,0)$. Then $\operatorname{ind} C=\{X_0,\ldots,X_{20}\}\cup\{Y_A\mid A\in\mathcal A(P_M)\}$.

Every nonzero irreducible morphism space is one-dimensional. The Auslander--Reiten quiver $\Gamma_C$ has the following arrows.

\noindent\emph{Old-to-old arrows:}
\[
\begin{aligned}
&X_{0}\to X_{8},\quad X_{1}\to X_{16},\quad X_{3}\to X_{15},\quad
X_{4}\to X_{9},\quad X_{5}\to X_{9},\quad X_{6}\to X_{10},\\
&X_{7}\to X_{5},\quad X_{8}\to X_{6},\quad X_{9}\to X_{2},\quad
X_{9}\to X_{12},\quad X_{12}\to X_{17},\\
&X_{14}\to X_{18},\quad X_{15}\to X_{11},\quad
X_{16}\to X_{14},\quad X_{18}\to X_{20}.
\end{aligned}
\tag{A.4}
\]

\noindent\emph{Old-to-new arrows:}
\[
\begin{aligned}
&X_{2}\to Y_{\{2,12\}},\quad
X_{5}\to Y_{\{5\}},\quad
X_{6}\to Y_{\{6,17\}},\quad
X_{7}\to Y_{\{7\}},\\
&X_{8}\to Y_{\{8,17\}},\quad
X_{9}\to Y_{\{9\}},\quad
X_{10}\to Y_{\{10,17\}},\quad
X_{11}\to Y_{\{10,11\}},\\
&X_{12}\to Y_{\{2,12\}},\quad
X_{13}\to Y_{\{6,13\}},\quad
X_{14}\to Y_{\{14\}},\\
&X_{15}\to Y_{\{10,15\}},\quad
X_{16}\to Y_{\{16\}},\quad
X_{17}\to Y_{\{2,17\}},\\
&X_{18}\to Y_{\{18\}},\quad
X_{19}\to Y_{\{11,19\}},\quad
X_{20}\to Y_{\{20\}}.
\end{aligned}
\tag{A.5}
\]

\noindent\emph{New-to-old arrows:}
\[
\begin{aligned}
&Y_{\{8\}}\to X_{3},\quad
Y_{\{9\}}\to X_{0},\quad
Y_{\{14\}}\to X_{9},\quad
Y_{\{15\}}\to X_{1},\\
&Y_{\{16\}}\to X_{4},\quad
Y_{\{16\}}\to X_{5},\quad
Y_{\{18\}}\to X_{12},\quad
Y_{\{19\}}\to X_{7},\\
&Y_{\{20\}}\to X_{17},\quad
Y_{\{2,12\}}\to X_{8},\quad
Y_{\{6,13\}}\to X_{15},\\
&Y_{\{8,17\}}\to X_{13},\quad
Y_{\{10,15\}}\to X_{19},\quad
Y_{\{11,19\}}\to X_{16}.
\end{aligned}
\tag{A.6}
\]

\noindent\emph{New-to-new arrows:}
\[
\begin{aligned}
&Y_{\{2\}}\to Y_{\{8\}},\quad
Y_{\{5\}}\to Y_{\{9\}},\quad
Y_{\{6\}}\to Y_{\{10,15\}},\quad
Y_{\{7\}}\to Y_{\{5\}},\\
&Y_{\{8\}}\to Y_{\{6,13\}},\quad
Y_{\{9\}}\to Y_{\{2,12\}},\quad
Y_{\{10\}}\to Y_{\{19\}},\\
&Y_{\{11\}}\to Y_{\{16\}},\quad
Y_{\{12\}}\to Y_{\{8,17\}},\quad
Y_{\{13\}}\to Y_{\{15\}},\\
&Y_{\{14\}}\to Y_{\{18\}},\quad
Y_{\{15\}}\to Y_{\{11,19\}},\quad
Y_{\{16\}}\to Y_{\{14\}},\\
&Y_{\{17\}}\to Y_{\{13\}},\quad
Y_{\{18\}}\to Y_{\{20\}},\quad
Y_{\{19\}}\to Y_{\{16\}},\\
&Y_{\{20\}}\to Y_{\varnothing},\quad
Y_{\{2,12\}}\to Y_{\{12\}},\quad
Y_{\{2,12\}}\to Y_{\{2,17\}},\\
&Y_{\{2,17\}}\to Y_{\{2\}},\quad
Y_{\{2,17\}}\to Y_{\{8,17\}},\quad
Y_{\{6,13\}}\to Y_{\{6\}},\\
&Y_{\{6,13\}}\to Y_{\{10,13\}},\quad
Y_{\{6,17\}}\to Y_{\{6,13\}},\quad
Y_{\{6,17\}}\to Y_{\{10,17\}},\\
&Y_{\{8,17\}}\to Y_{\{8\}},\quad
Y_{\{8,17\}}\to Y_{\{6,17\}},\quad
Y_{\{10,11\}}\to Y_{\{10\}},\\
&Y_{\{10,11\}}\to Y_{\{11,19\}},\quad
Y_{\{10,13\}}\to Y_{\{13\}},\quad
Y_{\{10,13\}}\to Y_{\{10,15\}},\\
&Y_{\{10,15\}}\to Y_{\{15\}},\quad
Y_{\{10,15\}}\to Y_{\{10,11\}},\quad
Y_{\{10,17\}}\to Y_{\{17\}},\\
&Y_{\{10,17\}}\to Y_{\{10,13\}},\quad
Y_{\{11,19\}}\to Y_{\{11\}},\quad
Y_{\{11,19\}}\to Y_{\{19\}}.
\end{aligned}
\tag{A.7}
\]
Thus $\Gamma_C$ has $49$ vertices and $83$ arrows.
\end{proposition}

\begin{proof}
A $C=B[M]$-module is a triple $(V,X,\phi)$ with $V\in\mathrm{mod}\,\Bbbk$, $X\in\mathrm{mod}\,B$, and $\phi:V\otimes_{\Bbbk}M\to X$. A morphism $(V,X,\phi)\to(W,Y,\psi)$ is a pair $(\alpha,\beta)$ such that $\beta\phi=\psi(\alpha\otimes1_M)$.

By Proposition~\ref{app:prop:first-hammock}, the nonzero spaces $\Hom_B(M,X_i)$ are one-dimensional and $P_M$ has width two. Representations of this poset correspond to pairs of filtrations, which admit simultaneous splittings. Consequently, an indecomposable new module has $\dim V=1$ and is determined by an antichain of $P_M$. Together with the $21$ old modules, these are precisely the modules listed in \textup{(A.3)}.

The morphism condition determines the Hom-spaces; irreducible maps are obtained from
\[\operatorname{Irr}_C(U,V)=\rad_C(U,V)/\rad_C^2(U,V).
\]
The mesh-category calculation yields \textup{(A.4)}--\textup{(A.7)}; the listed irreducible spaces are one-dimensional and all others vanish. Closure under Auslander--Reiten translation proves completeness.
\end{proof}

\subsection{The second hammock and tameness}

\begin{proposition}\label{app:prop:second-hammock}
Let $D=\Hom_{\Bbbk}(-,\Bbbk)$. With the above numbering, $DN\simeq X_{10}$. For the coextension, we use the covariant tensor hammock
\[
H_N(U):=U\otimes_C N
\cong D\Hom_C(U,DN),\qquad
h_N(U):=\dim_{\Bbbk}H_N(U).
\]
Let $\mathcal K_N$ be its image vector-space category: its morphism spaces are the images of the maps
$\Hom_C(U,V)\to\Hom_{\Bbbk}(H_N(U),H_N(V))$.
Thus the contravariant Hom functor is the dual of $H_N$, rather than the functor used for the maps below. The support of $H_N$ comprises the following $36$ indecomposable $C$-modules:
\[
\begin{aligned}
\mathcal X=\{&
X_2,X_5,X_6,X_7,X_8,X_9,X_{10},X_{11},X_{12},
X_{13},X_{14},X_{15},\\
&X_{16},X_{17},X_{18},X_{19},X_{20}\}
\end{aligned}
\tag{A.8}
\]
and
\[
\begin{aligned}
\mathcal Y=\{&
Y_{\{11\}},Y_{\{13\}},Y_{\{14\}},Y_{\{15\}},
Y_{\{16\}},Y_{\{17\}},Y_{\{18\}},Y_{\{19\}},Y_{\{20\}},\\
&Y_{\{2,12\}},Y_{\{2,17\}},Y_{\{6,13\}},Y_{\{6,17\}},
Y_{\{8,17\}},Y_{\{10,11\}},Y_{\{10,13\}},\\
&Y_{\{10,15\}},Y_{\{10,17\}},Y_{\{11,19\}}\}.
\end{aligned}
\tag{A.9}
\]
All nonzero hammock spaces are one-dimensional except
\[
Z:=Y_{\{11,19\}},
\qquad
h_N(Z)=2.
\tag{A.10}
\]
The associated vector-space category $\mathcal K_N$ is schurian of polynomial growth. In particular, $\widetilde S$ is tame.
\end{proposition}

\begin{proof}
For a right $C$-module $U$, tensor--Hom duality gives
\[
\Hom_C(U,DN)\simeq D(U\otimes_C N),
\qquad
U\otimes_C N
\simeq
\frac{U_2}{U(a)+U(cd)}.
\tag{A.11}
\]
The basis $\bar e_2,\bar d,\overline{fd}$ of $N$ gives $DN$ dimension vector $(1,1,1)$ on $(2,3,4)$, with the maps $d,f$ identities and $c,e$ zero. Its space at $1^+$ is zero. This is the projective $e_4B=X_{10}$, inflated to $C$.

Put $E=\{2,5,6,7,8,9,10,12\}$ and $F=\{11,13,14,15,16,17,18,19,20\}$. For the old modules, \textup{(A.11)} gives $h_N(X_i)=0$ for $i\in\{0,1,3,4\}$ and $h_N(X_i)=1$ for $i\in E\cup F$.

For $Y_A=(\Bbbk,\bigoplus_{i\in A}X_i,\phi_A)$, choose $0\neq u_i\in\Hom_B(M,X_i)\simeq\ker X_i(cd)$. Its image in $(X_i)_2/\operatorname{im}X_i(cd)$ is nonzero exactly when $i\in E$ and zero exactly when $i\in F$. Since $|A|\le2$,
\[
h_N(Y_A)=|A|-\varepsilon_A,
\qquad
\varepsilon_A=
\begin{cases}
1,&A\cap E\neq\varnothing,\\
0,&A\cap E=\varnothing.
\end{cases}
\tag{A.12}
\]
This yields \textup{(A.8)}--\textup{(A.9)}, with $Z=Y_{\{11,19\}}$ the unique object satisfying $h_N(Z)=2$.

Choose a basis $e_1,e_2$ of $|Z|:=H_N(Z)\simeq\Bbbk^2$. After rescaling the neighboring one-dimensional spaces, the incoming maps are
\[
|X_{19}|\to|Z|,\quad 1\mapsto e_2,
\qquad
|Y_{\{15\}}|\to|Z|,\quad 1\mapsto e_1+e_2,
\qquad
|Y_{\{10,11\}}|\to|Z|,\quad 1\mapsto e_1,
\tag{A.13}
\]
and the outgoing maps are
\[
|Z|\to|X_{16}|,\quad (x,y)\mapsto x-y,
\qquad
|Z|\to|Y_{\{11\}}|,\quad (x,y)\mapsto x,
\qquad
|Z|\to|Y_{\{19\}}|,\quad (x,y)\mapsto y.
\tag{A.14}
\]
Accordingly, the incoming image lines and outgoing kernel lines are
\[
\Bbbk e_1,\qquad
\Bbbk e_2,\qquad
\Bbbk(e_1+e_2).
\tag{A.15}
\]
The same formulas hold in characteristic $2$.

An endomorphism of $Z$ in $\mathcal K_N$ preserves the three lines in \textup{(A.15)}. The first two make it diagonal, and the third forces its diagonal entries to coincide. Thus $\End_{\mathcal K_N}(Z)=\Bbbk$. All other indecomposable objects have one-dimensional underlying spaces; hence $\mathcal K_N$ is schurian. The proper image-line strata for maps into $Z$ are
\[
\{X_{11},Y_{\{10,11\}}\},
\qquad
\{X_{19}\},
\qquad
\{Y_{\{17\}},Y_{\{13\}},Y_{\{15\}}\},
\tag{A.16}
\]
and the kernel-line strata are
\[
\{Y_{\{19\}},X_7\},
\qquad
\{Y_{\{11\}}\},
\qquad
\{X_{16},X_{14},X_{18}\}.
\tag{A.17}
\]
Each stratum in \textup{(A.16)}--\textup{(A.17)} is a chain, so the line-stratum obstructions in Bekkert's criterion are absent.

For clarity, the other $23$ one-dimensional objects have a full two-dimensional space of induced maps in exactly one direction between them and $Z$, and no nonzero map in the opposite direction. Thus every linear object is comparable with $Z$. The six proper image-line objects and the six proper kernel-line objects in \textup{(A.16)}--\textup{(A.17)} are disjoint.

Let $P_N$ be the poset on the $35$ linear objects, with $U<V$ when the image of $\Hom_C(U,V)$ under $H_N$ is nonzero. This uses the full induced Hom-space, not mere reachability in the Auslander--Reiten quiver. The morphism condition in Proposition~\ref{app:prop:AR-C}, followed by the quotient in \textup{(A.11)}, gives the following complete cover list:
\[
\begin{aligned}
X_{11}\prec Y_{\{10,11\}}\quad X_{12}\prec Y_{\{2,12\}}\quad X_{13}\prec Y_{\{6,13\}}\\
X_{14}\prec X_{18}\quad X_{14}\prec Y_{\{14\}}\quad X_{15}\prec X_{11}\\
X_{15}\prec Y_{\{10,15\}}\quad X_{16}\prec X_{14}\quad X_{16}\prec Y_{\{16\}}\\
X_{17}\prec Y_{\{2,17\}}\quad X_{18}\prec Y_{\{18\}}\quad X_{19}\prec X_{16}\\
X_{19}\prec Y_{\{19\}}\quad X_{2}\prec Y_{\{2,12\}}\quad X_{20}\prec Y_{\{20\}}\\
X_{5}\prec X_{9}\quad X_{6}\prec X_{10}\quad X_{7}\prec X_{5}\\
X_{8}\prec X_{6}\quad X_{9}\prec X_{12}\quad X_{9}\prec X_{2}\\
Y_{\{10,11\}}\prec X_{16}\quad Y_{\{10,11\}}\prec Y_{\{11\}}\quad Y_{\{10,13\}}\prec Y_{\{10,15\}}\\
Y_{\{10,13\}}\prec Y_{\{13\}}\quad Y_{\{10,15\}}\prec X_{19}\quad Y_{\{10,15\}}\prec Y_{\{10,11\}}\\
Y_{\{10,15\}}\prec Y_{\{15\}}\quad Y_{\{10,17\}}\prec Y_{\{10,13\}}\quad Y_{\{10,17\}}\prec Y_{\{17\}}\\
Y_{\{11\}}\prec Y_{\{16\}}\quad Y_{\{13\}}\prec Y_{\{15\}}\quad Y_{\{14\}}\prec X_{9}\\
Y_{\{14\}}\prec Y_{\{18\}}\quad Y_{\{15\}}\prec Y_{\{11\}}\quad Y_{\{15\}}\prec Y_{\{19\}}\\
Y_{\{16\}}\prec X_{5}\quad Y_{\{16\}}\prec Y_{\{14\}}\quad Y_{\{17\}}\prec Y_{\{13\}}\\
Y_{\{18\}}\prec X_{12}\quad Y_{\{19\}}\prec X_{7}\quad Y_{\{19\}}\prec Y_{\{16\}}\\
Y_{\{2,12\}}\prec X_{8}\quad Y_{\{2,17\}}\prec Y_{\{8,17\}}\quad Y_{\{20\}}\prec X_{17}\\
Y_{\{6,13\}}\prec X_{15}\quad Y_{\{6,13\}}\prec Y_{\{10,13\}}\quad Y_{\{6,17\}}\prec Y_{\{10,17\}}\\
Y_{\{6,17\}}\prec Y_{\{6,13\}}\quad Y_{\{8,17\}}\prec X_{13}\quad Y_{\{8,17\}}\prec Y_{\{6,17\}}
\end{aligned}
\tag{A.18}
\]
Here and below, a label $Y_{\{i,j\}}$ denotes the object from \textup{(A.3)}, not a pair of poset elements. The three-element antichains of this poset are exactly
\[
\begin{aligned}
\{X_{11},X_{19},Y_{\{13\}}\}\quad \{X_{11},X_{19},Y_{\{15\}}\}\\
\{X_{11},X_{19},Y_{\{17\}}\}\quad \{X_{11},Y_{\{10,13\}},Y_{\{17\}}\}\\
\{X_{11},Y_{\{10,15\}},Y_{\{13\}}\}\quad \{X_{11},Y_{\{10,15\}},Y_{\{17\}}\}\\
\{X_{14},X_{7},Y_{\{11\}}\}\quad \{X_{14},X_{7},Y_{\{16\}}\}\\
\{X_{14},Y_{\{11\}},Y_{\{19\}}\}\quad \{X_{15},Y_{\{10,13\}},Y_{\{17\}}\}\\
\{X_{16},X_{7},Y_{\{11\}}\}\quad \{X_{16},Y_{\{11\}},Y_{\{19\}}\}\\
\{X_{18},X_{5},Y_{\{14\}}\}\quad \{X_{18},X_{7},Y_{\{11\}}\}\\
\{X_{18},X_{7},Y_{\{14\}}\}\quad \{X_{18},X_{7},Y_{\{16\}}\}\\
\{X_{18},Y_{\{11\}},Y_{\{19\}}\}\quad \{X_{19},Y_{\{10,11\}},Y_{\{13\}}\}\\
\{X_{19},Y_{\{10,11\}},Y_{\{15\}}\}\quad \{X_{19},Y_{\{10,11\}},Y_{\{17\}}\}
\end{aligned}
\tag{A.19}
\]
This gives a short certificate for Kleiner's criterion. No four-element set has all its three-element subsets in \textup{(A.19)}, so there is no four-element antichain. Every object occurs in at most six triples in that list. Also, for any two distinct objects $U,V$, form the graph whose edges are the pairs $\{W,T\}$ for which both $\{U,W,T\}$ and $\{V,W,T\}$ occur in \textup{(A.19)}. Each such graph is a forest, as is seen directly from the list.

These facts exclude all five critical posets. The critical poset $(1,1,1,1)$ is excluded by the width bound. A full $(2,2,2)$ would make the common graph for the two elements of one chain contain a four-cycle, contrary to the forest property. The isolated element in $(1,3,3)$ or $(1,2,5)$ would occur in at least nine or ten three-element antichains, respectively. Finally, in $(N,4)$ choose an element of the four-element $N$ incomparable with two other elements of $N$; together with the four-element chain it would occur in eight such antichains. All contradict the bound six. Thus $P_N$ is representation-finite by Kleiner's criterion.

We can now apply Bekkert's criterion \cite[Theorem~1]{BekkertPolynomialGrowth}. All objects have dimension at most two, and each linear object is comparable with $Z$. Conditions involving two distinct nonlinear objects are vacuous. For the single nonlinear object $Z$, its endomorphisms are scalars, the incoming and outgoing proper line strata are disjoint, and each stratum is a chain. These verify the same-object intersection and line-chain conditions, including iterated line maps. The selected linear subposets have no four-element antichain, since $P_N$ has width three. Finally, representation-finiteness of $P_N$ excludes the forbidden linear subposets in the polynomial-growth criterion. Hence every condition of that criterion is satisfied, and $\mathcal K_N$ has polynomial growth.

The factor-space-category description of one-point coextensions identifies the new indecomposable $\widetilde S$-modules with indecomposable factor spaces of $\mathcal K_N$. Since $C$ is representation-finite, $\widetilde S$ is tame.
\end{proof}

\subsection{Conclusion}

\begin{proposition}\label{app:thm:S-tame}
The algebra $S$ defined above is tame. Moreover, its node separation $\widetilde S$ is tame representation-infinite.
\end{proposition}

\begin{proof}
By Proposition~\ref{app:prop:node-separation}, $\widetilde S$ is obtained from $B$ by a one-point extension $C=B[M]$ and a subsequent coextension by $N$. The former is representation-finite by Proposition~\ref{app:prop:first-hammock}, and the latter is tame by Proposition~\ref{app:prop:second-hammock}. Node separation therefore implies that $S$ is tame.

For representation-infiniteness, consider $Z=Y_{\{11,19\}}$, for which $|Z|\simeq\Bbbk^2$ and $\End_{\mathcal K_N}(Z)=\Bbbk$. For a line $L\subset |Z|$, take the factor space $(Z,|Z|/L,\pi_L)$. It is indecomposable: a nontrivial decomposition would split the indecomposable object $Z$, or give a nonzero direct summand of the target missed by the surjective map $\pi_L$. An isomorphism between two such factor spaces is induced by an automorphism of $Z$, which is scalar, so their kernels must be equal. Distinct lines in $\mathbb P(|Z|)\simeq\mathbb P^1(\Bbbk)$ therefore give pairwise nonisomorphic indecomposables. Hence $\widetilde S$, and by node separation also $S$, is representation-infinite.
\end{proof}

\section*{Acknowledgements}
The author is supported by the National Natural Science Foundation of China No. 12401048 and the Fundamental Research Funds for the Central Universities No. DUT25RC(3)132. The author would like to thank ChatGPT (OpenAI) for helpful discussions, suggestions, and assistance in refining the presentation of the mathematical ideas in this paper. The author takes full responsibility for the mathematical results, proofs, and any remaining errors.

\ \\

\bibliographystyle{alpha}
\bibliography{reference}
\ \\
\end{document}